\documentclass[11pt,reqno]{amsart}
\usepackage[colorlinks=true,
pdfstartview=FitV, linkcolor=cyan, citecolor=magenta,
urlcolor=blue]{hyperref}
\usepackage{amsmath,amsfonts,latexsym,amssymb}
\usepackage{mathrsfs}
\usepackage{verbatim}
\usepackage[utf8]{inputenc}
\usepackage[T1]{fontenc}
\usepackage{ae,aecompl}
\usepackage{braket}
\usepackage{pxfonts}
\usepackage{microtype}
\usepackage{graphicx}
\usepackage{subfig}
\newtheorem{theorem}{Theorem}[section]
\newtheorem{lemma}[theorem]{Lemma}
\newtheorem{proposition}[theorem]{Proposition}
\newtheorem{corollary}[theorem]{Corollary}

\renewcommand{\leq}{\leqslant}
\renewcommand{\geq}{\geqslant}
\makeatletter
\newcommand{\oset}[3][0ex]{%
  \mathrel{\mathop{#3}\limits^{
    \vbox to#1{\kern-2\ex@
    \hbox{$\scriptstyle#2$}\vss}}}}
\makeatother
\newcommand{\dt}[1]{\oset{\bullet}{#1}\!\!}

\newcommand{\defeq}{{}\mathrel{\mathop:}={}}

\newcommand{\mc}{\mathcal}
\newcommand{\hh}{{\bf H}^2}
\newcommand{\ba}[2]{\varepsilon^{#1}_{#2}}

\newcommand{\pslt}{\mathsf{PSL}_2(\mathbb R)}
\newcommand{\pglt}{\mathsf{PGL}_2(\mathbb R)}

\newcommand{\sln}{\mathsf{SL}_d(\mathbb R)}
\newcommand{\psln}{\mathsf{PSL}_d(\mathbb R)}
\newcommand{\pgln}{\mathsf{PGL}_d(\mathbb R)}

\newcommand{\Hn}{{\mc H}_{d}(S)}
\newcommand{\xHn}{{\widetilde{\mc H}}_{d}(S)}

\newcommand{\tr}{\operatorname{Tr}}

\newcommand{\II}{{\bf I}}
\newcommand{\JJ}{{\bf J}}

\newcommand{\p}{{\bf{p}}}

\newcommand{\Real}{\mathbb R}
\newcommand{\N}{\mathbb N}
\newcommand{\Ht}{{\mathbb H}^2}

\newcommand{\ms}{\mathsf}

\renewcommand{\hom}{\operatorname{Hom}}
\renewcommand{\ge}{\geqslant}
\renewcommand{\le}{\leqslant}
\renewcommand{\geq}{\geqslant}
\renewcommand{\leq}{\leqslant}
\def\eproof{$\Box$ \medskip}

\def\w{\wedge}
\title{Simple Length Rigidity for Hitchin Representations}
\author[Bridgeman]{Martin Bridgeman}
\address{Boston College, Chestnut Hill, MA 02467 USA}
\author[Canary]{Richard Canary}
\address{University of Michigan, Ann Arbor, MI 41809 USA}
\author[Labourie]{Fran\c cois Labourie}
\address{ Universit\'e C\^ote d'Azur, LJAD, Nice F-06000; FRANCE}
\thanks{Bridgeman was partially suppported by grants DMS-1500545  and DMS-1564410 and
Canary was partially supported by  grants DMS-1306992 and DMS-1564362, from the National Science Foundation.
Labourie was partially supported by the European Research Council under the {\em European Community}'s seventh Framework Programme (FP7/2007-2013)/ERC {\em grant agreement} ${\rm n}^{\tiny o}$ FP7-246918.}

\begin{document}
\begin{abstract} Nous montrons qu'une représentation de Hitchin est déterminée par les rayons spectraux des images de courbes simples et non séparantes. Comme application, nous caractérisons les isométries de la function d'intersection pour les composantes  de Hitchin  en dimension 3, ainsi que pour les composantes auto-duales en toutes dimensions.  Un outil important de notre démonstration est un  résultat de transversalité sur les quadruplets positifs de drapeaux.
\vskip 0.2 truecm
\centerline{---------------------}
We show that a Hitchin representation is determined by the spectral radii of  the images of simple,
non-separating closed curves. As a consequence, we classify isometries of the intersection function
on Hitchin components of dimension 3 and on the self-dual Hitchin components in all dimensions. As
an important tool in the proof, we establish a transversality result for positive quadruples of flags.
\end{abstract}
\maketitle
\section{Introduction}

Any discrete faithful representation of the fundamental group $\pi_1(S)$ of a closed oriented surface 
$S$  of genus greater than 1 
into $\pslt$ is determined, up to conjugacy in $\pglt$,  by the translation lengths of
(the images) of  a finite collection of elements represented by simple closed curves.
More precisely, a collection of $6g-5$ simple closed curves will be enough but $6g-6$ simple closed curve will not suffice, see Schmutz \cite{schmutz} and Hamenst\"adt \cite{hamenstadt-lengths}. In $\pslt$ the translation length of an element
is determined by the absolute value of the trace (which is well-defined, although the trace is not), so one may equivalently
say that a discrete faithful representation of $\pi_1(S)$ into $\pslt$ is determined by the
(absolute values of) the traces of a finite collection of elements represented by simple closed curves.

We establish analogues of this result for Hitchin representations.  The fact that traces of simple closed curves determine
the representation is more surprising in the Hitchin setting as the trace
does not even determined the conjugacy class of an element in $\psln$ if $d\ge 3$.

In the proof, we use Lusztig positivity to establish
transversality properties for limit curves of a Hitchin representations, and more generally for positive quadruples of flags.
We also establish a rigidity result which depends on {\em correlation functions} associated to triples of simple closed curves.
We hope that these transversality and rigidity results are of independent interest
and  that this  paper will serve as an introduction to the beautiful algebraic ideas  for mathematicians with a more geometric background. 

\subsection*{Hitchin representations}
A {\em Hitchin representation of dimension $d$} is a
representation of $\pi_1(S)$ into $\psln$ which may be continuously deformed to a {\em $d$-Fuchsian representation} that is the composition of
the irreducible representation of $\pslt$ into $\psln$ with a discrete faithful
representation of $\pi_1(S)$ into $\pslt$. The {\em Hitchin component} $\mc H_d(S)$
of all Hitchin representations of $\pi_1(S)$ into $\psln$,  considered up to conjugacy in $\pgln$, is homeomorphic to $\mathbb R^{-(d^2-1)\chi(S)}$.
In particular, $\mc H_2(S)$ is the Teichm\"uller space of $S$ -- see Section \ref{sec:Hitchin} for details and history.

A Hitchin representation is said to be {\em self dual} if it is  conjugate to its contragredient. 
Self dual Hitchin representations take values in $\ms{PSp}(2n,\mathbb R)$ and $\ms{PSO}(n,n+1)$, when $d=2n$ or $d=2n+1$ respectively. 
The set $\mc{SH}_d(S)$ of self dual representations into $\psln$ is a contractible submanifold of $\mc{H}_d(S)$ (see \cite{hitchin}).

\subsection*{Spectrum rigidity}

The {\em spectral length} of a conjugacy class $\gamma$ in $\pi_1(S)$ -- or equivalently a free homotopy class of curve in $S$ --   
with respect to a Hitchin representation  $\rho$ is 
$$L_\gamma (\rho)\defeq\log\Lambda(\rho(\gamma))$$
where $\Lambda(\rho(\gamma))$ is the spectral radius of $\rho(\gamma)$. 

The {\em  marked length spectrum}  of $\rho$ is the function from the set of conjugacy classes in $\pi_1(S)$ defined by
$$
L(\rho):\gamma\mapsto L_\gamma(\rho).
$$
Similarly, the  {\em marked trace spectrum} is the map
$$
\gamma\mapsto \vert\tr(\rho(\gamma))\vert,
$$
where $\vert\tr(A)\vert$ is the absolute value of the trace of a lift of a matrix $A\in\psln$ to $\sln$.

Our first main result is then

\begin{theorem}{\sc[Simple Marked Length Rigidity]}
\label{lengthrigidity} 
Two Hitchin representations of a closed orientable surface of genus greater than 2 are equal
whenever their marked length spectra coincide on simple non-separating curves.
\end{theorem}

The restriction on the genus may not only reflect the limit of our methods: we have extended this result to surfaces with boundary,
see Section \ref{positivereps},  and it is clear that  simple length rigidity fails for the pair of pants when $d>2$.

 We obtain a finer result for the trace spectrum

\begin{theorem}{\sc [Simple Marked Trace Rigidity]}
\label{tracerigidity}
Two Hitchin representations of a closed orientable surface of genus greater than 2 are equal
whenever their marked trace spectra coincide on simple non-separating curves.
Furthermore, if $S$ is a closed orientable surface of genus  greater than 2 and $d\ge 2$, then there exists a finite set $\mc L_d(S)$ of simple non-separating curves, so that  two Hitchin representations of $\pi_1(S)$ of dimension $d$  are equal whenever their marked trace spectra coincide on $\mc L_d(S)$.
\end{theorem}

Dal'bo and Kim \cite{dalbo-kim} earlier proved that Zariski dense representations of a group $\Gamma$ into a semi-simple Lie
group $\ms{G}$ without compact factor are determined, up to automorphisms of $\ms{G}$, by the marked spectrum of translation 
lengths of {\em all} elements on the quotient symmetric space $\ms{G}/\ms{K}$. Similar results were obtained by Charette and Drumm \cite{CD}
for subgroups of the affine Minkowski group. 
Bridgeman, Canary, Labourie and Sambarino \cite{BCLS} proved that Hitchin
representations, are determined up to conjugacy in $\pgln$ by the spectral radii of all elements. Bridgeman
and Canary \cite{BCQF} proved that discrete faithful representations of $\pi_1(S)$ into $\ms{PSL}(2,\mathbb C)$
are determined by the translation lengths of simple non-separating curves on $S$. 
Duchin, Leininger and Rafi \cite{DLR} showed that the simple marked length spectrum
determines a flat surface, but that no collection of finitely many simple closed curves suffices to determine a flat surface.
On the other hand, March\'e and Wolff \cite[Section 3]{marche-wolff} gave examples of non-conjugate, indiscrete, non-elementary
representations of a closed surface group of genus two into $\pslt$ with the same simple marked length spectra. 

In Section \ref{positivereps} we establish a version of Theorem \ref{lengthrigidity} for Hitchin representations
of  compact surfaces with boundary which are ``complicated enough,''
while in Section \ref{infinitesmal} we establish an infinitesimal version of Theorem \ref{lengthrigidity}.

\subsection*
{Isometry groups of the intersection} We apply Theorem \ref{lengthrigidity} to characterize diffeomorphisms preserving the  intersection function  of  representations in  $\Hn$. 

In Teichm\"uller theory, the {\em intersection} 
$\II(\rho,\sigma)$ of representations $\rho$ and $\sigma$ in $\mc T(S)$  is the length with respect to $\sigma$ of a random geodesic in 
$\hh/\rho(\pi_1(S))$ -- where $\hh$ is the hyperbolic plane.
Thurston showed that the Hessian of the intersection function gives rise to a Riemannian metric on
$\mc T(S)$, which Wolpert \cite{wolpert} showed was a multiple of the classical Weil--Petersson metric -- see also
 Bonahon \cite{bonahon}, McMullen \cite{mcmullen-pressure},
and Bridgeman \cite{bridgeman-wp} for further interpretation.  As a special case of their main result,  Bridgeman, Canary, Labourie and Sambarino \cite{BCLS} used the Hessian of
a {\em renormalized intersection}  function to construct a mapping class group invariant, analytic, Riemannian
metric on $\Hn$, called the {\em pressure metric} -- see Section \ref{sec:isom} for details.

Royden \cite{royden} showed that the isometry group of $\mc T(S)$, equipped with the Teichm\"uller metric,
is the extended mapping class group, while Masur and Wolf \cite{masur-wolf} established the same
result for the Weil--Petersson metric.

In our context, the {\em intersection isometry group} -- respectively {\em self dual intersection isometry group}--  is the set of those diffeomorphisms of $\mc H_d(S)$ --  respectively $\mc{SH}_d(S)$ -- preserving $\II$. 

\begin{theorem}{\sc[Self dual isometry group]}
\label{selfdualisometries}
For a surface of genus greater than 2, the self dual intersection isometry group coincides with the extended mapping class group
of $S$.
\end{theorem}

We have a finer result when $d=3$.

\begin{theorem}{\sc[Isometry Group In Dimension 3]}
\label{isometriesofintersectionnumber}
For a surface  $S$ of genus greater than 2, the  intersection isometry group  of  $\mc{H}_3(S)$ is
generated by  the extended mapping class group of $S$ and the contragredient involution.
\end{theorem}

Since, as we will see in the proof, isometries of
the intersection function are also isometries of the pressure metric, we view this as evidence for the conjecture
that this is also the isometry group of the pressure metric -- See Section \ref{intersection}
for precise definitions.

Our proof follows the outline suggested by the proof in Bridgeman--Canary \cite{BCQF} that the isometry group
of the intersection function on quasifuchsian space is generated by the extended mapping class group and
complex conjugation.

A key tool in the proof of Theorem \ref{isometriesofintersectionnumber} is a rigidity result for the
marked simple, non-separating Hilbert length spectrum for a representation into $\ms{PSL}(3,\mathbb R)$,
see Section \ref{SL3}. Kim \cite{kim-JDG}, see also Cooper-Delp \cite{cooper-delp}, had previously
proved a marked Hilbert length rigidity theorem for the full marked length spectrum.

\subsection*{Positivity and correlation functions} 
Every element of the image of a Hitchin representation is purely loxodromic, i.e. diagonalizable with real eigenvalues of distinct
modulus.  We introduce correlation functions which  record the relative positions of eigenspaces of elements in the image and
give rise to a rigidity result for the restrictions of Hitchin representation to certain three generator subgroups. This
new rigidity result relies crucially on a new transversality result for eigenbases of images of disjoint curves.

If $\rho$ is a Hitchin representation of dimension $d$, and $\gamma$ is a non-trivial element, 
a matrix representing $\rho(\gamma)$ may be written --see Section \ref{sec:Hitchin} -- as 
$$
\rho(\gamma)=\sum_{i=1}^d\lambda_i\left(\rho(\gamma)\right)\p_i\left(\rho(\gamma)\right),
$$
where $\lambda_1\left(\rho(\gamma)\right)>\ldots>\lambda_d\left(\rho(\gamma)\right)>0$ are the {\em eigenvalues}
(of some lift)
of $\rho(\gamma)$ and $\p_i\left(\rho(\gamma)\right)$ are the {\em projectors} onto
the corresponding 1-dimensional eigenspaces. Let  \begin{itemize}
	\item $\mc A=(\alpha_1,\ldots,\alpha_n)$ be an $n$-tuple  of non-trivial elements of $\pi_1(S)$, 
	\item $I=
\left(i_j\right)_{j\in\{1,\ldots,n\}}$ be an $n$-tuple of elements 
in $\{1,\ldots,d\}$.
\end{itemize}
The associated {\em correlation function}  ${\bf T}_{I}(\mc A)$ on $\mc H_d(S)$ is defined  by
$${\bf T}_{I}(\mc A): \rho \mapsto
\tr\left(\prod_{j=1}^n \p_{i_j}(\rho(\alpha_j))\right).$$ 

We show that finitely many of these correlation often suffice to determine  the restriction of a Hitchin representation
to a three generator subgroup. One may use this result to give an embedding of $\mathcal H_d(S)$ in some $\mathbb R^N$
and we hope that a refinement of these ideas could yield new parametrisations of $\mathcal H_d(S)$.
In the statement below, recall that a pair of disjoint simple closed curves is said to be {\em non-parallel}  if they do
not bound an annulus.

\begin{theorem}{\sc[Rigidity for correlations functions]}
\label{conjugateontriples}
Let $\rho$ and $\sigma$ be Hitchin representations in $\mc H_d(S)$.
Suppose that   $\alpha, \beta,\delta \in \pi_1(S)-\{1\}$ are represented by based loops which are freely homotopic
to a collection of  pairwise disjoint and non-parallel simple closed curves. 
Assume that
\begin{enumerate}
	\item for any  $\eta \in \{\alpha,\beta,\delta\}$, $\rho(\eta)$ and $\sigma(\eta)$ have the same eigenvalues,
	\item for all $i,j,k$ in $\{1,\ldots,d\}$
 $$\frac{{\bf T}_{i,j,k}(\alpha,\beta,\delta)}{{\bf T}_{j,k}(\beta,\delta)}(\rho)=
\frac{{\bf T}_{i,j,k}(\alpha,\beta,\delta)}{{\bf T}_{j,k}(\beta,\delta)}(\sigma),$$
\end{enumerate}
then $\rho$ and $\sigma$
are conjugate, in $\pgln$, on the subgroup of $\pi_1(S)$ generated by $\alpha$, $\beta$ and $\gamma$.
\end{theorem}


Before even stating that theorem,  we need to prove the relevant correlation functions never vanish. 
This will be a corollary of the following theorem. First recall that a Hitchin representation in $\mc H_d(S)$ defines a {\em limit curve} in the 
flag manifold of $\mathbb R^d$, so that any two distinct points are transverse. 
Recall also that any transverse pair flags $a$ and $b$ in $\mathbb R^d$ defines a decomposition of $\mathbb R^d$ 
into a sum of $d$ lines $L_1(a,b),\ldots L_d(a,b)$.

\begin{theorem} {\sc[Transverse bases]}
\label{transverse-bases-general0}
Let $\rho$ be a Hitchin representation of dimension $d$.  Let $(a,x,y,b)$ be four cyclically ordered points in the limit curve of $\rho$, 
then any  $d$  lines in
$$\{ L_1(a,b),\ldots,L_d(a,b),L_1(x,y),\ldots,L_d(x,y)\}$$
are in general position.
\end{theorem}

This last result is a consequence of the positivity theory developed by Lusztig \cite{lusztig-reductive} 
and used in the theory of Hitchin representations by Fock--Goncharov \cite{fock-goncharov} and  is actually a special
case of a more general result about positive quadruples, see Theorem \ref{transverse-bases-general1}. 
Theorem \ref{transverse-bases-general1} may be familiar to experts but we could not find a proper reference to it in the literature. 

We also establish a more general version of Theorem \ref{conjugateontriples}, see Theorem \ref{conjugateontriples-general}.
 
\subsection*{Structure of the proof} Let us sketch the proof of Theorem \ref{lengthrigidity}. 
The proof runs through the following steps. We first show, in Section \ref{lengthtrace}, that if the length spectra agree on 
simple non-separating curves, then all the eigenvalues agree for these curves. This follows by considering curves of 
the form $\alpha^n\beta$ when $\alpha$ and $\beta$ have geometric 
intersection one and using an asymptotic expansion. A similar argument yields  that ratio of correlation functions  agree for certain triples of curves that only exist in genus 
greater than 2, see Theorem \ref{conjugateontriples-trace}, and a repeated use of 
Theorem \ref{conjugateontriples} concludes the proof of  Theorem \ref{lengthrigidity}. 
Theorem \ref{transverse-bases-general0} is crucially  used several times to show that 
coefficients appearing in asymptotic expansions do not vanish.

\subsection*{Acknowledgements} Section \ref{transverse} uses ideas that are being currently developed by the third author in
collaboration with Olivier Guichard and Anna Wienhard. We have benefitted immensely from discussions with Yves Benoist, 
Sergey Fomin, Olivier Guichard, Andres Sambarino and Anna Wienhard and wish to thank them here.  
This material is partially based upon work supported by the National Science Foundation while
the third author was in residence at the Mathematical Sciences Research Institute in Berkeley, CA, during the
Fall 2016 semester. The authors  also gratefully acknowledge  support from U.S. National Science Foundation grants 
DMS 1107452, 1107263, 1107367 "RNMS: GEometric structures And Representation varieties" (the GEAR Network).

\section{Hitchin representations and limit maps}
\label{sec:Hitchin}

\subsection{Definitions}
Let $S$ be  a closed  orientable surface of genus $g\ge 2$.
A representation $\rho:\pi_1(S)\to \pslt$ is said to be {\em Fuchsian} if it is discrete and faithful.
Recall that Teichm\"uller space $\mc T(S)$ is the subset of
$${\rm Hom}(\pi_1(S),\pslt)/\pglt$$
consisting of (conjugacy classes of) Fuchsian representations. 

Let $\tau_d:\pslt \to\psln$ be the irreducible representation (which is well-defined up to conjugacy).
A representation $\sigma:\pi_1(S)\to\psln$ is said to be {\em $d$-Fuchsian} if it has the form $\tau_d\circ\rho$
for some Fuchsian representation $\rho:\pi_1(S)\to\pslt$. A representation $\sigma:\pi_1(S)\to \psln$
is a {\em Hitchin representation} if it may be continuously deformed to a $d$-Fuchsian representation.
The {\em Hitchin component} $\Hn$ is the component of the space of reductive representations up to conjugacy:
$${\rm Hom}^{\operatorname{red}}(\pi_1(S),\psln)/\pgln$$
consisting of (conjugacy classes of) Hitchin representations.  In analogy with Teichm\"uller space
$\mc T(S)=\mc H_2(S)$,
Hitchin proved that $\Hn$ is a real analytic manifold diffeomorphic to a cell.

\begin{theorem}{\em (Hitchin \cite{hitchin})}
If $S$ is a closed orientable surface of genus $g\ge 2$ and $d\ge 2$, then
$\Hn$ is a real analytic manifold diffeomorphic to $\mathbb R^{(d^2-1)(2g-2)}$.
\end{theorem}

The {\em Fuchsian locus} is the subset of $\Hn$ consisting of $d$-Fuchsian representations. It is naturally
identified with $\mc T(S)$. 

\subsection{Real-split matrices and proximality}

If $A \in \sln$ is real-split, i.e. diagonalizable over $\mathbb R$,
we may order the eigenvalues $\{\lambda_i(A)\}_{i=1}^d$ so that
$$|\lambda_1(A)| \ge |\lambda_2(A)|  \ge \cdots |\lambda_{d-1}(A)| \geq |\lambda_d(A)|.$$
Let $\{e_i(A)\}_{i=1}^d$ be a basis for $\mathbb R^d$ so that $e_i(A)$ is an eigenvector with eigenvalue $\lambda_i(A)$
and let $e^i(A)$ denote the linear functional so that $\braket{e^i(A) | e_i(A) }=1$ and
$\braket{e^i(A) | e_j(A) }=0$ if $i\ne j$.
Let $\p_i(A)$ denote the projection onto $\braket{e_i(A)}$ parallel to the hyperplane spanned by the other $d-1$ basis elements.
Then, 
$$\p_i(A)(v)=\braket{e^i(A)\ | \ v }\ e_i(A)$$
and we may write
$$A=\sum_{i=1}^d \lambda_i(A)\p_i(A).$$

We say that $A$ is {\em $k$-proximal} if 
$$|\lambda_1(A)| > |\lambda_2(A)| >\ldots |\lambda_k(A)| > |\lambda_{k+1}(A)|$$
and we say that $A$ is {\em purely loxodromic} if it is $(d-1)$-proximal,  in which case it is
diagonalizable over $\mathbb R$ with eigenvalues of distinct modulus.
If $A$ is $k$-proximal, then, for all $i=1,\ldots,k$, $\p_i(A)$ is well-defined and
$e_i(A)$ is well-defined up to scalar multiplication. Moreover, if $A$ is purely loxodromic
$\p_i(A)$ is well-defined and $e_i(A)$ and $e^i(A)$ are well-defined up to scalar multiplication for all $i$.
If $A\in\psln$, we say that $A$ is {\em purely loxodromic} if any lift of $A$ to an element of
$\sln$ is purely loxodromic.

\subsection{Transverse flags and associated bases}

A {\em flag} for $\mathbb R^d$ is a nested family 
$$f=(f^1, f^2,\ldots, f^{d-1})$$
of vector subspaces of $\mathbb R^d$ where $f^i$ has dimension $i$ and $f^i\subset f^{i+1}$ for each $i$. 
Let $\mc F_d$ denote
the space of all flags for $\mathbb R^d$. An $n$-tuple $(f_1,\ldots,f_n) \in \mc F_d^n$ is {\em transverse} if
$$f^{d_1}_1\oplus f^{d_2}_2\oplus\ldots\oplus f^{d_n}_n = \Real^d$$
for any partition $\{ d_i\}_{i\in\{1,\ldots,n\}}$  of $d$. Let $\mc F_d^{(n)}$ be the set of transverse $n$-tuples of flags, 
and note that $\mc F_d^{(n)}$is an open dense subset of  $\mc F_d^n$. 

Two transverse flags $(a,b)$ determine a decomposition of $\mathbb R^d$ as sum of lines $\{L_i(a,b)\}_{i=1}^d$
where 
$$L_i(a,b)=a^i\cap b^{d-i+1}$$
for all $i$.  A basis $\ba{a}{b}=\{e_i\}$ for $\mathbb R^d$ is 
{\em consistent} with $(a,b)\in\mc F_d^{(2)}$ if $e_i\in L_i(a,b)$ for all $i$, or, equivalently, if
$$a^j=\braket{e_1,\ldots,e_j}\quad {\rm and}\quad b^j=\braket{e_d,\ldots,e_{d-j+1}}$$
for all $j$.
In particular,
the choice of basis is well-defined up to scalar multiplication of basis elements.

\subsection{Limit maps}

Labourie \cite{labourie-anosov} associates a limit map from $\partial_\infty\pi_1(S)$ into $\mc F_d$
to every Hitchin representation. This map encodes many crucial properties of the representation.

\begin{theorem}{\rm (Labourie \cite{labourie-anosov})}
\label{hyperconvexity}
If $\rho \in \Hn$, then there exists a unique  continuous \hbox{$\rho$-equivariant} map 
$\xi_\rho: \partial_\infty\pi_1(S) \rightarrow {\mc F}_d$,
such that:
\begin{enumerate}
\item (Proximality) If $\gamma\in\pi_1(S)-\{1\}$, then $\rho(\gamma)$ is purely loxodromic and
$$\xi_\rho^i(\gamma^+)=\braket{e_1(\rho(\gamma)),\ldots,e_i(\rho(\gamma))}$$
for all $i$, where $\gamma^+\in\partial_\infty\pi_1(S)$ is the attracting fixed point of $\gamma$.
\item (Hyperconvexity) 
If $x_1,\ldots, x_k \in \partial_\infty\pi_1(S)$ are distinct and $m_1+\ldots  +m_k = d$, then
$$\xi^{m_1}(x_1)\oplus\ldots\oplus\xi^{m_j}(x_j)\oplus\ldots\oplus\xi^{m_k}(x_k) = \Real^d.$$
\end{enumerate}
\end{theorem}

Notice that if $\gamma\in\pi_1(S)-\{1\}$ and $\gamma^\pm\in \partial_\infty\pi_1(S)$ are its
attracting and repelling fixed points,  then 
$\rho(\gamma)$ is diagonal with respect to any basis consistent with $(\xi_\rho(\gamma^+),\xi_\rho(\gamma^-))$.
Moreover, if $\sigma$ is in the Fuchsian locus, then $\sigma(\gamma)$ has a lift to $\sln$ all of whose eigenvalues
are positive. Therefore, if $\rho\in\Hn$, then $\rho(\gamma)$ has a lift to $\sln$ with positive eigenvalues and
we define
$$\lambda_1(\rho(\gamma))>\lambda_2(\rho(\gamma))>\cdots>\lambda_d(\rho(\gamma))>0$$
to be the eigenvalues of this specific lift.

It will also be useful to note that any Hitchin representation \hbox{$\rho:\pi_1(S) \rightarrow \psln$} can be lifted to a 
representation $\tilde\rho:\pi_1(S)\to \sln$. Moreover, Hitchin \cite[Section 10]{hitchin} observed that
every Hitchin component lifts to a component of ${\rm Hom}^{\operatorname{red}}(\pi_1(S),\sln)/\sln$. 

\subsection{Other Lie groups and other length functions}

More generally, if $\ms G$ is a split, real simple adjoint Lie group, Hitchin \cite{hitchin}
studies the component
$$\mc H(S,\ms G)\subset {\rm Hom}^{\operatorname{red}}(\pi_1(S),\ms G)/\ms G$$
which contains the composition of a Fuchsian representation into $\pslt$
with an irreducible representation of $\pslt$
into $\ms G$ and shows that it is an analytic manifold diffeomorphic to $\mathbb R^{(2g-2){\rm dim}(\ms G)}$.

If $\rho\in \Hn$, then we define the {\em contragredient} representation $\rho^*\in\Hn$ by
$\rho^*(\gamma)=\rho(\gamma^{-1})^T$ for all $\gamma\in\pi_1(S)$. The {\em contragredient involution}
of $\Hn$ takes $\rho$ to $\rho^*$.

We define the {\em self dual Hitchin representations} -- and accordingly the {\em self dual Hitchin component} $\mc S\Hn$ -- 
to be the fixed points of the contragredient involution.
Since the contragredient involution is an isometry of the pressure metric, $\mc S\Hn$ is a 
totally geodesic submanifold of $\Hn$. 

 Observe then that if $\rho$ is a self dual Hitchin representation and $\gamma\in\pi_1(S)$, then the eigenvalues 
 $\lambda_1(\rho(\gamma)),\ldots,\lambda_d(\rho(\gamma))$ satisfy 
 $\lambda_i^{-1}(\rho(\gamma))=\lambda_{d-i+1}(\rho(\gamma))$ for all $i$.  On the other hand, Theorem 1.2 in \cite{BCLS}
 implies that  if $\lambda_1^{-1}(\rho(\gamma))=\lambda_{d}(\rho(\gamma))$ for all $\gamma$, then $\rho$
 is conjugate to its contragredient $\rho^*$.
Notice that the contragredient involution fixes each point in $\mc H(S,\ms{PSp}(2d,\mathbb R))$,
$\mc H(S,\ms{PSO}(d,d+1))$, and $\mc H(S,\ms{G}_{2,0})$ considered as subsets of $\mc H(S,\ms{PSL}(2d,\mathbb R))$, $\mc H(S,\ms{PSL}(2d+1,\mathbb R))$, and $\mc H(S,\ms{PSL}(7,\mathbb R))$ respectively. Conversely, a self dual representation, being conjugate to its contragredient, is not Zariski dense, hence belongs to such a subset by a result of Guichard \cite{guichard}.
In particular, $\mc S\mathcal H_{2d}(S)=\mc H(S,\ms{PSp}(2d,\mathbb R))$ and 
$\mc S\mathcal H_{2d+1}(S)=\mc H(S,\ms{PSO}(d,d+1))$.

In our work on isometries of the intersection function, it will be useful to consider
the {\em Hilbert length} $L_\gamma^H(\rho)$ of $\rho(\gamma)$ when $\gamma\in\pi_1(S)$ and $\rho\in {\mc H}_{d}(S)$,
where
\begin{eqnarray*}
L_\gamma^H(\rho) & \defeq & \log\lambda_1(\rho(\gamma))-\log\lambda_d(\rho(\gamma))\ ,
\end{eqnarray*}
and similarly the {\em Hilbert length spectrum} as a function on free homotopy 
classes.\footnote{This is called the Hilbert length, since when $d=3$ it is the length of the closed geodesic  in the
homotopy class of $\gamma$ in the Hilbert metric on the strictly convex real projective structure on $S$ with 
holonomy $\rho$, see, for example, Benoist \cite[Proposition 5.1]{benoist-divisible1}.} 
Notice that $L^H_\gamma(\rho)=L^H_{\gamma^{-1}}(\rho)=L^H_{\gamma}(\rho^*)$.
One readily observes that a representation is self dual if and only if  
$L^H_\gamma(\rho)=2L_\gamma(\rho)$ for all non-trivial $\gamma\in\pi_1(S)$.

\section{Transverse bases}
\label{transverse}

In this section, we prove a strong transversality property for ordered quadruples of flags in 
the limit curve of a Hitchin representation, which we regard as a generalization of the hyperconvexity property established
by Labourie \cite{labourie-anosov} (see Theorem \ref{hyperconvexity}).  (Recall that
any pair $(a,b)$ of transverse flags determines a decomposition of $\mathbb R^d$ into a sum of $d$ lines 
$L_1(a,b)\oplus\cdots\oplus L_d(a,b)$ where $L_i(a,b)=a^i\cap b^{d-i+1}$.)

\medskip\noindent
{\bf Theorem \ref{transverse-bases-general0}.}  {\em
Let $\rho$ be a Hitchin representation of dimension $d$ and let $(a,x,y,b)$ be four cyclically ordered points in the limit curve of $\rho$, 
then any  $d$  lines in
$$\{ L_1(a,b),\ldots,L_d(a,b),L_1(x,y),\ldots,L_d(x,y)\}$$
are in general position.
}

\medskip

The proof of Theorem \ref{transverse-bases-general0}
relies on the theory of positivity developed by Lusztig \cite{lusztig-reductive} and 
applied to representations of surface groups by Fock and Goncharov \cite{fock-goncharov}. It will follow from
a more general result for positive quadruples of flags, see Theorem \ref{transverse-bases-general1}.

\medskip\noindent
{\bf Remark:}
When $\rho\in\mathcal H_3(S)$, there exists a strictly convex domain $\Omega_\rho$ in $\mathbb{RP}^2$ with
$C^1$ boundary so that $\rho(\pi_1(S))$ acts properly discontinuously and cocompactly on $\Omega_\rho$,
see Benoist \cite{benoist-divisible1} and Choi-Goldman \cite{choi-goldman}.
If $\xi_\rho$ is the limit map of $\rho$, then $\xi_\rho^{1}$ identifies $\partial_\infty\pi_1(S)$ with $\partial\Omega_\rho$,
while $\xi_\rho^{2}(z)$ is the plane spanned by the (projective) tangent line to $\partial\Omega_\rho$ at $\xi_\rho^1(z)$.
In this case, Theorem \ref{transverse-bases-general0} is an immediate consequence of the strict convexity of $\Omega_\rho$,
since if $x$ and $y$ lie in the limit curve, then
$L_1(x,y)=x^1$, $L_3(x,y)=y^1$ and $L_2(x,y)$ is the intersection of the tangent lines to $\Omega_\rho$ at
$x^1$ and $y^1$.
Moreover, one easily observes that the analogue of Theorem \ref{transverse-bases-general0} does not hold for
cyclically ordered quadruples of the form $(a,x,b,y)$.

\subsection{Components of positivity}
\label{bruhat cells}
Given  a flag $a$, we define the {\em Schubert cell} $B_a \subset \mathcal F_d$ to be the set of all flags transverse to $a$. 
Let $U_a$ be the group of unipotent elements in the stabilizer of $a$, i.e. 
the set of unipotent upper triangular matrices with respect to a basis $\{ e_i\}$ consistent with $a$.
If $b\in B_a$, we can assume that $\{e_i\}$ is consistent with $(a,b)$, so it is apparent
that the stabilizer of $b$ in $U_a$ is trivial. The lemma below follows easily.

\begin{lemma}
\label{Udiffeo}
If $b\in B_a$, then $B_a=U_a(b)$. Moreover, the map
$$h_b:U_a\to B_a$$
defined by $h_b(u)=u(b)$ is a diffeomorphism.
\end{lemma}

Suppose that $(a,b)\in\mathcal F_d^{(2)}$ and $\ba{a}{b}$ is a basis consistent with the pair $(a,b)$.
Recall that 
$A\in\sln$ is {\em totally positive} with respect to $\ba{a}{b}$, if every minor
in its matrix with respect to the basis $\ba{a}{b}$ is positive.
Similarly,
we say that $A\in\sln$ is {\em totally non-negative} with respect to $\ba{a}{b}$, if every minor
in its matrix with respect to the basis $\ba{a}{b}$ is non-negative.  
Let $U(\ba{a}{b})_{\ge 0}\subset U_a$ be the set of totally non-negative unipotent upper triangular matrices
with respect to $\ba{a}{b}$.
We say that a minor is an
{\em upper minor} with respect to $\ba{a}{b}$ if it is non-zero for some element of $U(\ba{a}{b})_{\geq 0}$.
We then let $U(\ba{a}{b})_{>0}$ be the subset of $U(\ba{a}{b})_{\geq 0}$ consisting of elements all of
whose upper minors with respect to $\ba{a}{b}$ are positive. 
Moreover, let $\Delta(\ba{a}{b})_{>0}$ be the group of matrices which are diagonalizable with respect
to $\ba{a}{b}$ with positive eigenvalues.
Lusztig \cite{lusztig-reductive}  proves that

\begin{lemma}{\rm (Lusztig \cite[Sec. 2.12, Sec. 5.10]{lusztig-reductive}}
\label{U inclusion}
If $(a,b)\in \mathcal F_d^{(2)}$ and  $\ba{a}{b}$ is a basis consistent with the pair $(a,b)$,
then 
$$
U(\ba{a}{b})_{\ge0}U(\ba{a}{b})_{> 0}\subset U(\ba{a}{b})_{> 0} \quad \textrm{and}\quad
\overline{U(\ba{a}{b})_{> 0}}=U(\ba{a}{b})_{\ge 0} \subset U_a.$$

\end{lemma}

If $i\ne j$ and $t\in\mathbb R$, the elementary Jacobi matrix $J_{ij}(t)$ with respect to 
$\ba{a}{b}=\{e_i\}$ is the matrix such that $J_{ij}=e_j+te_i$ and
$J_{ij}(e_k)=e_k$ if $k\ne i$.  If $i<j$ and $t>0$, then $J_{ij}(t)\in U(\ba{a}{b})_{\ge 0}$.
Moreover, $U(\ba{a}{b})_{\ge 0}$ is generated by elementary Jacobi matrices of this  form
(see, for example, \cite[Thm. 12]{fomin-zelevinsky}). So,
\begin{enumerate}
	\item the semigroup $U(\ba{a}{b})_{\ge 0}$ is connected, and
	\item 
	if $g\in \Delta(\ba{a}{b})$, then 
$g U(\ba{a}{b})_{>0} g^{-1}=U(\ba{a}{b})_{>0}$.
\end{enumerate}  

We define the {\em component of positivity} for $\ba{a}{b}$ as
$$V(\ba{a}{b})\defeq U(\ba{a}{b})_{>0}(b).$$

Lusztig \cite[Thm. 8.14]{lusztig-reductive} (see also Lusztig \cite[Lem. 2.2]{lusztig-partialflag}) identifies
$V(\ba{a}{b})$ with a component of the intersection $B_a\cap B_b$ of two opposite Schubert cells.

\begin{lemma}{\rm (Lusztig \cite[Thm. 8.14]{lusztig-reductive})}
If $(a,b)\in \mathcal F_d^{(2)}$ and  $\ba{a}{b}$ is a basis consistent with the pair $(a,b)$,
then 
$V(\ba{a}{b})$ is a connected component of $B_a\cap B_b$.
\end{lemma}

\subsection{Positive configurations of flags}\label{sec:posi}

We now recall the theory of positive configurations of flags as developed by Fock and Goncharov
\cite{fock-goncharov}.

A triple $(a,x,b)\in\mathcal F_d^{(3)}$ is {\em positive} with respect to a basis
$\ba{a}{b}$ consistent with $(a,b)$ if $x=u(b)$ for some $u\in U(\ba{a}{b})_{> 0}$. 
If $x\in V(\ba{a}{b})$, we define
$$V(a,x,b)=V(\ba{a}{b})$$ 
and notice that $V(a,x,b)$ is the component of $B_a\cap B_b$  which contains $x$.

More generally, a $(n+2)$-tuple $(a,x_n,\ldots,x_1,b)\in\mathcal F_d^{(n+2)}$ of flags is {\em positive} if there exist 
$u_i\in U(\ba{a}{b})_{> 0}$ so that  $x_p=u_1\cdots u_p(b)$ for all $p$. 
By construction, the set of positive $(n+2)$-tuples of flags is connected.
Since
$U(\ba{a}{b})_{> 0}$ is a semi-group, $(a,x_i,b)$ is a positive triple for all $i$ and, more generally,
$(a,x_{i_1},\ldots,x_{i_k},b)$ is
a positive $(k+2)$-tuple whenever $1\le i_i<\cdots<i_k\le n$. 

Fock and Goncharov showed that the positivity of a $n$-tuple is invariant 
under the action of the dihedral group on $n$ elements. 

\begin{proposition}{\em (Fock-Goncharov \cite[Thm. 1.2]{fock-goncharov})}\label{flip}
If $(a_1,\ldots,a_n)$ is a positive $n$-tuple of flags in $\mathcal F_d$, 
then $(a_2,a_3,\ldots,a_n,a_1)$  and $(a_n,a_{n-1},\ldots,a_1)$ are both positive as well.	
\end{proposition}

As a consequence, we see that every sub $k$-tuple of a positive $n$-tuple is itself positive.

\begin{corollary}
\label{sub-positive}
If $(a_1,\ldots,a_n)$ is a positive $n$-tuple of flags in $\mathcal F_d$ and
\hbox{$1\le i_1<i_2<\cdots< i_k\le n$}, then $(a_{i_1},a_{i_2},\ldots,a_{i_k})$ is positive.
\end{corollary}

\begin{proof} 
It suffices to prove that every sub $(n-1)$-tuple of a positive $n$-tuple is positive.
By Proposition \ref{flip}, we may assume that the sub $(n-1)$-tuple has the form
$(a_1,a_3,\ldots,a_n)$ and we have already seen that this $(n-1)$-tuple is positive.
\end{proof}

The main result of the section can now be formulated more generally as a result about positive quadruples.
Its proof will be completed in Section \ref{transverse-bases-proof}.

\begin{theorem}{\sc[Transverse bases for quadruples]}
\label{transverse-bases-general1}
Let  $(a,x,y,b)$ be a positive quadruple in $\mathcal F_d$, then any  $d$  lines in
$$\{ L_1(a,b),\ldots,L_d(a,b),L_1(x,y),\ldots,L_d(x,y)\}$$
are in general position.
\end{theorem}

\subsection{Positive maps}
If $\Sigma$ is a cyclically ordered set with at least 4 elements,
a  map $\gamma:\Sigma\to \mathcal F_d$ is said to be {\em positive} 
if whenever $(z_1,z_2,z_3,z_4)$ is an ordered quadruple
in $\Sigma$, then its image $(\gamma(z_1),\gamma(z_2),\gamma(z_3),\gamma(z_4))$ is a 
positive quadruple in $\mathcal F_d^{(4)}$. 

For example, given an irreducible representation 
$$\tau_d:\pslt\to\psln$$
the $\tau_d$-equivariant {\em Veronese embedding} 
$$\nu_\tau:\partial \hh={\bf P}^1(\mathbb R)\to\mathcal F_d$$
(where $\nu_\tau$ takes the attracting fixed point of $g\in\pslt$ to the attracting fixed point of $\tau_d(g)$)
is a positive map. More generally,
Fock and Goncharov, see also Labourie-McShane \cite[Appendix B]{labourie-mcshane},
showed that the limit map of a Hitchin representation is positive. 

\begin{theorem}{\rm (Fock-Goncharov \cite[Thm 1.15]{fock-goncharov})}
\label{fockgoncharovpositivity}
If $\rho\in\Hn$, then the associated limit map 
$\xi_\rho:\partial_\infty \pi_1(S) \rightarrow {\mathcal F}_d$ is positive.
\end{theorem} 

Notice that Theorem \ref{transverse-bases-general0} follows immediately from Theorems \ref{transverse-bases-general1}
and \ref{fockgoncharovpositivity}.

\medskip

We observe that one may detect the positivity of a $n$-tuple using only quadruples, which immediately
implies that positive maps take cyclically ordered subsets to positive configurations.

\begin{lemma}
\label{quadruple check}
If $n\ge 2$, then an
$(n+2)$-tuple $(a,x_n,\ldots,x_1,b)$ is positive if and only if $(a,x_{i+1},x_i,b)$ is positive for
all $i=1,\ldots,n-1$.
\end{lemma}

\begin{proof}
Corollary \ref{sub-positive} implies that  if $(a,x_n,\ldots,x_1,b)$ is positive,
then $(a,x_{i+1},x_i,b)$ is positive for all $i$.

Now suppose that  $(a,x_{i+1},x_i,b)$ is positive for all $i=1,\ldots,n-1$. Since $(a,x_2,x_1,b)$ is positive,
there exists $u_1,u_2\in U(\ba{a}{b})_{> 0}$ so that
$x_1=u_1(b)$ and $x_2=u_1u_2(b)$. If we assume that there exists $u_i\in U(\ba{a}{b})_{> 0}$, for all $i\le k<n$,
so that $x_p=u_1\cdots u_p(b)$ for all $p\le k$, then, since $(a,x_{k+1},x_k,b)$ is positive, there exists 
$u_{k+1},v_k\in U(\ba{a}{b})_{> 0}$ such that $x_{k+1}=v_ku_{k+1}(b)$ and $x_k=v_k(b)$. However, 
Lemma \ref{Udiffeo} implies that $v_k=u_1\cdots u_k$. Iteratively applying this argument, we see that
$(a,x_n,\ldots,x_1,b)$ is positive.
\end{proof}

\begin{corollary}
\label{positive maps are positive}
If $\Sigma$ is a cyclically ordered set, $f:\Sigma\to\mathcal F_d$ is a positive map
and $(a_1,\ldots,a_n)$ is a cyclically ordered $n$-tuple in $\Sigma$,
then $(f(a_1),f(a_2),f(a_3),\ldots,f(a_n))$ is a positive $n$-tuple in $\mathcal F_d$.
\end{corollary}

The following result allows one to simplify the verification that a map of a finite set into ${\mathcal F}_d$ is positive,
see also Section 5.11 in Fock-Goncharov \cite{fock-goncharov}

\begin{proposition}\label{prop:triangul}
Let $P$ be a finite  set in $\partial_\infty\hh$ and $\mathcal T$ be an
ideal triangulation  of the convex polygon spanned by $P$.
A map $f:P\to \mathcal F_d$ is positive if whenever $(x,y,z,w)$ are the (cyclically ordered) vertices of 
two ideal triangles in $\mathcal T$ which share an edge, then  $(f(x),f(y),f(z),f(w))$ is a positive quadruple.
\end{proposition}

\begin{proof}
Suppose $\mathcal T'$ is obtained from $\mathcal T$ by replacing an internal edge of $\mathcal T$
by an edge joining the opposite vertices of the adjoining triangles. In this case, we say that $\mathcal T'$ is obtained from
$\mathcal T$ by performing an {\em elementary move}.
 Label the vertices of the original edge by $a$ and $b$
and the vertices of the new edge by $x$ and $y$, so that the vertices occur in the order $(a,x,b,y)$ in $\partial_\infty\hh$.
If the edge $(y,a)$ abuts another triangle with additional vertex $z$, then $(a,x,y,z)$ is a cyclically ordered collection
of points in $P$ which are the vertices of two ideal triangles in $\mathcal  T'$  which share an edge.
By our original assumption on $\mathcal T$, $(f(a),f(x),f(b),f(y))$ and $(f(a),f(b),f(y),f(z))$ are positive,
so, by Proposition \ref{flip}, $(f(y),f(a),f(x),f(b))$ and $(f(y),f(z),f(a),f(b))$ are positive.
Lemma \ref{quadruple check} then implies   that $(f(y),f(z),f(a),f(x),f(b))$ is positive.
Another  application of Proposition \ref{flip} gives that $(f(a),f(x),f(b),f(y),f(z))$ is positive,
so $(f(a),f(x),f(y),f(z))$ is positive. One may similarly check that all the images of
cyclically ordered vertices of two ideal triangles which share an edge in $\mathcal T'$ have positive image.
Since any two ideal triangulations can be joined
by a sequence of triangulations so that consecutive triangulations differ by an elementary move,
any ordered sub-quadruple of $P$ has  positive image. Therefore, $f$ is a positive map.
\end{proof}

\subsection{Complementary components of positivity}

If $(a,b)\in\mathcal F_d^{(2)}$ and $\ba{a}{b}=\{e_i\}$ is a basis consistent with $(a,b)$,
then one obtains a complementary basis $\sigma(\ba{a}{b})=\{(-1)^i e_i\}$ which is also
consistent with $(a,b)$. We first observe that for a positive sextuple $(x,y,a,u,v,b)$, then the
the components of positivity for $(a,b)$ containing $\{u,v\}$ and $\{x,y\}$ are
associated to complementary bases.
The proof proceeds by first checking the claim for configurations in the image of a Veronese embedding
and then applying a continuity argument.

\begin{lemma}
\label{bruhat complement}
If $(x,y,a,u,v,b)$ is  a positive sextuple of flags and $\ba{a}{b}$ is a basis consistent with $(a,b)$ so
that $V(\ba{a}{b})$ contains $\{u,v\}$, 
then $V(\sigma(\ba{a}{b}))$ contains $\{x,y\}$.
\end{lemma}

\begin{proof}
Consider the irreducible representation 
\hbox{$\tau_d:\pslt\to\psln$}
taking matrices diagonal in the
standard basis for $\mathbb R^2$ to matrices diagonal with respect to $\ba{a}{b}$. 
This gives rise to a Veronese embedding 
\hbox{$\nu_\tau:\partial \hh=S^1\to\mathcal F_d$}
taking $\infty$ to $a$ and $0$ to $b$.  

The involution of $\mathcal F_d$ induced by 
conjugating by the diagonal matrix $D$, in the basis $\ba{a}{b}$,
with entries $((-1)^i)$ interchanges
the components of  \hbox{$\nu_\tau(S^1)-\{a,b\}$} and interchanges $V(\ba{a}{b})$ and
$V(\sigma(\ba{a}{b}))$. Therefore, our result holds when $x$, $y$, $z$ and $w$ lie in the image of $\nu_\tau$.

Since $\nu_\tau$ is positive and the set of positive sextuples is connected, there is a
 family of positive maps $\xi_t:\{x,y,a,u,v,b\}\to \mathcal F_d$ so that the image of $\xi_0$ lies on the image of the 
Veronese embedding and $\xi_1=\operatorname{Id}$. Since  $\psln$ acts transitively on  space of 
pairs of transverse flags, we may assume that $\xi_t(a)=a$ and $\xi_t(b)=b$ for all $t$.
Notice that each of  $\xi_t(\{x,y\})$ and $\xi_t(\{x,y\})$ lies in a component of $B_a\cap B_b$ for all $t$.
Since $\xi_1(\{u,v\})\subset V(\ba{a}{b})$, $\xi_t(\{u,v\})\subset V(\ba{a}{b})$ for all $t$.
Since $\xi_0(\{u,v\})\subset V(\ba{a}{b})$ and $\xi_0(x,y,a,u,v,b)$ lies in the image of an Veronese embedding,
$\xi_0(\{x,y\})\subset V(\sigma(\ba{a}{b}))$, which in turn implies that
$\xi_t(\{x,y\})\subset V(\sigma(\ba{a}{b}))$ for all $t$.
\end{proof}

We next observe that  the closures of  complementary components of positivity intersect in at most one
point within an associated Schubert cell.

\begin{proposition}
\label{circle prop}
If $(a,b)\in\mathcal F_d^{(2)}$ and $\ba{a}{b}$ is a basis consistent with $(a,b)$, then
$$B_a\cap \overline{V(\ba{a}{b})}\cap \overline{V(\sigma(\ba{a}{b}))}=\{b\}.$$
\end{proposition}

\begin{proof}
By Lemma \ref{Udiffeo},  
$$V(\ba{a}{b})=h_b(U(\ba{a}{b})_{>0})\subset h_b(U(\ba{a}{b})_{\ge 0})\subset h_b(U_a)=B_a$$
and $h_b(U(\ba{a}{b})_{\ge 0})$ is a closed subset of $B_a$, since $h_b$ is a diffeomorphism.
So 
$$B_a\cap \overline{V(\ba{a}{b})}\subset h_b(U(\ba{a}{b})_{\ge 0})\quad {\rm and}\quad
B_a\cap\overline{V(\sigma(\ba{a}{b}))}\subset h_b(U(\sigma(\ba{a}{b}))_{\ge 0}).$$

Thus, again since $h_b$ is a diffeomorphism,
\begin{eqnarray*}
B_a\cap \overline{V(\ba{a}{b})}\cap \overline{V(\sigma(\ba{a}{b}))} & \subset &
h_b(U(\ba{a}{b})_{\ge 0})\cap h_b(U(\sigma(\ba{a}{b}))_{>0}\\
& = & h_b\left(U(\ba{a}{b})_{\ge 0}\cap U(\sigma(\ba{a}{b}))_{\ge 0}\right)\\
& = &  \left(U(\ba{a}{b})_{\ge 0}\cap U(\sigma(\ba{a}{b}))_{\ge 0}\right)(b)\\
\end{eqnarray*}

So Proposition \ref{circle prop} follows from the following lemma:

\begin{lemma}
$$U(\ba{a}{b})_{\ge 0}\cap U(\sigma(\ba{a}{b}))_{\ge 0}=\{I\}.$$
\end{lemma}

\begin{proof}
Let $A=(a_{ij})\in U(\ba{a}{b})_{\ge 0}\cap U(\sigma(\ba{a}{b}))_{\ge 0}$
be written with respect to the basis $\ba{a}{b}$. Notice that if we let $\overline{a_{ij}}$ be
the matrix coefficients for $A$ with respect to the basis $\sigma(\ba{a}{b})$, then
$a_{ij}=(-1)^{i+j}\overline{a_{ij}}$. It follows immediately that $a_{ij}=0$ if $i+j$ is odd.

If $A\ne I$, let $a_{ij}>0$ be a non-zero off-diagonal term which is closest to the diagonal, {\it i.e.}
$a_{lj}=0$ if $l\ne j$ and $l>i$ and $a_{il}=0$ if $l\ne i$ and $l<j$. If $l\in (i,j)$, we consider the minor
$$\begin{bmatrix}  a_{il} & a_{ij} \\ a_{ll} & a_{lj} \end{bmatrix}=
\begin{bmatrix}  0 & a_{ij} \\ 1& 0 \end{bmatrix}$$
which has determinant $-a_{ij}$, so contradicts the fact that $A$ is totally non-negative.

\end{proof}
\end{proof}

\subsection{Nesting of components of positivity}

We will need a strict containment property for components of positivity associated to positive
quintuples.

\begin{proposition}
\label{bruhat lemma}
If $(a,x,z,y,b)$ is a positive quintuple in $\mathcal F_d$, then 
$$\overline{V(x,z,y)}\subset V(a,z,b).$$
\end{proposition}

We begin by establishing nesting properties for components of positivity associated to positive quadruples.

\begin{lemma}
\label{inclusion property}
If  $(a,x,y,b)$ is a positive quadruple in $\mathcal F_d$, then
$$V(x,y,b)\subset V(a,y,b)\quad {\rm and}\quad V(a,x,y)\subset V(a,x,b)$$
\end{lemma}

\begin{proof}
Since $(a,x,y,b)$ is a positive quadruple, 
there exists a basis 
$\ba{a}{b}$ for $(a,b)$ and \hbox{$u,v\in U(\ba{a}{b})_{> 0}$}
so that $y=u(b)$  and $x=u(v(b))$. Since $U(\ba{a}{b})_{>0}$ is
a semi-group $uv\in U(\ba{a}{b})_{>0}$ and \hbox{$x,y\in V(\ba{a}{b})=U(\ba{a}{b})_{>0}(b)$}.

Notice that $\ba{a}{y}=u(\ba{a}{b})=\{u(e_i)\}$ is a basis consistent with $(a,y)$ since
$u(a)=a$, $u(b)=y$ and $\braket{e_i}=a^i\cap b^{d-i+1}$, so 
$$\braket{u(e_i)}=u(a^i)\cap u(b^{d-i+1})=a^i\cap y^{d-i+1}.$$
Let $W=uU(\ba{a}{b})_{>0}u^{-1}$, so 
$W=U(\ba{a}{y})_{>0}$. Therefore,
$$V(\ba{a}{y})=W(y)=u U(\ba{a}{b})_{>0}(u^{-1}(y))=\left(uU(\ba{a}{b})_{>0}\right)(b)
\subset U(\ba{a}{b})_{>0}(b)=V(\ba{a}{b})$$
where the inclusion follows from the fact that $U(\ba{a}{b})_{>0}$ is a semi-group and
$u\in U(\ba{a}{b})_{> 0}$. Moreover,
$$x\in V(\ba{a}{y})=\left(uU(\ba{a}{b})_{>0}\right)(b)\subset V(\ba{a}{b})$$ 
since $uv\in uU(\ba{a}{b})_{>0}$ and $x=u(v(b))$, so 
$$V(a,x,y)=V(\ba{a}{y})\subset V(\ba{a}{b})=V(a,x,b).$$

Since $(b,y,x,a)$ is also a positive quadruple, the same argument
shows that $V(b,y,x)\subset V(b,y,a)$. Since $V(b,y,x)=V(x,y,b)$ and $V(b,y,a)=V(a,y,b)$, 
we  conclude that 
$$V(x,y,b)\subset V(a,y,b).$$ 
\end{proof}

We now analyze the limiting behavior of sequences of components of positivity.

\begin{lemma}
\label{Hausdorff limit} 
Suppose that  $\{c_n\}$ is  a sequence of flags converging to $b$ and $(y_1,a, y_0, c_n,z_n,b)$ is a positive sextuple
for all $n$. Then the Hausdorff limit of $\{\overline{V(c_n,z_n,b)}\}$ is the singleton $\{b\}$.
\end{lemma}

\begin{proof} 
Since $(a,c_n,z_n,b)$ and $(c_n,z_n,b,a)$ are positive, Lemma \ref{inclusion property} implies that
$$V(c_n,z_n,b)\subset V(a,z_n,b)\cap V(c_n,z_n,a)$$
for all $n$, so
$$\overline{V(c_n,z_n,b)}\subset \overline{V(a,z_n,b)}\cap \overline{V(c_n,z_n,a)}.$$

After extracting a subsequence, we may assume that $\left\{\overline{V(c_n,z_n,b)}\right\}$ converges to a Hausdorff limit $H$. It is enough to  prove that $H=\{b\}$. Notice that, since each
$V(c_n,z_n,b)$ is connected, $H$ must be connected.

Notice that, for all $n$,   $V(a,z_n,b)=V(a,y_0,b)$, since $(a,y_0,z_n,b)$ is positive, and $V(c_n,z_n,a)=V(c_n,y_1,a)$,
since $(c_n,z_n,y_1,a)$ is positive.
Since $\{B_{c_n}\}$ converges to $B_c$,
$\left\{\overline{V(c_n,z_n,a)}\right\}=\left\{\overline{V(c_n,y_1,a)}\right\}$ converges to $\overline{V(b,y_1,a)}$.  
Therefore,
$$\{b\}\subset H\subset \overline{V(a,y_0,b)}\cap \overline{V(b,y_1,a)}.$$
However, Lemma \ref{bruhat complement} and Proposition \ref{circle prop} together imply that
$$B_a\cap  \overline{V(a,y_0,b)}\cap \overline{V(b,y_1,a)}=\{b\}.$$
Since $B_a$ is an open neighborhood of $b$ and
$H$ is connected, we conclude that $H=\{b\}$.
\end{proof}

\medskip\noindent
{\em Proof of Proposition \ref{bruhat lemma}.}
We note that if $(a,x_n,\ldots,x_1,b)$ is positive with respect to the basis $\ba{a}{b}$ with $x_n = vb$ for 
 $v \in U(\ba{a}{b})_{>0}$, if
$u\in U(\ba{a}{b})_{>0}$ then $(a,vu(x_n), x_{n},\ldots,x_1,b)$ is positive. Since positivity 
is invariant under cyclic permutations, we may add flags in any position to a positive $n$-tuple to obtain
a positive $(n+1)$-tuple.

Choose $c$ and $e$ so that $(a,c,x,z,y,e,b)$ is positive and let $g$ be an element in $\Delta(\ba{c}{e})_{>0}$.
We observe that $(a,c,g(y),g(z),e,b)$ is positive.

\begin{lemma}\label{pos-diag2}
If $(a,c,x,z,e,b)$ is  a positive sextuple in $\mathcal F_d$ and $g\in\Delta(\ba{c}{e})_{>0}$, then 
$(a,c,g(x),g(z),e,b)$ is positive.
\end{lemma}

\begin{proof} 
Identify $(a,c,g(x),g(z),e,b)$ with the cyclically ordered vertices of an ideal hexagon in $\mathbb H^2$ and
consider the triangulation $\mathcal T$ all of whose internal edges have an endpoint at $e$.
Proposition \ref{prop:triangul} implies that it suffices to check that 
$(c,g(x),g(z),e)$, $(c,g(x),e,a)$, and $(a,c,e,b)$ are positive quadruples, to guarantee
that $(a,c,g(x),g(z),e,b)$ is positive.

Since $(c,x,z,e)$ is positive, there exists $u,v\in U(\ba{c}{e})_{>0}$ so that $x=vu(e)$ and $z=v(e)$.
If we let $u'=gug^{-1}$ and $v'=gvg^{-1}$, then $u',v'\in U(\ba{c}{e})_{>0}$ 
(see property (2) in Section \ref{bruhat cells}). One checks that
$$v'u'(e)=(gvg^{-1})(gug^{-1})=g(vu)(g^{-1}(e))=g(vu(e))=g(x),\ \ \textrm{and}$$
$$v'(e)=(gvg^{-1})(e)=gv(g^{-1}(e))=g(v(e))=g(z),$$
so $(c,g(x),g(z),e)$ is a positive quadruple.

Since $(c,x,e,a)$ is a positive quadruple, there exists $u,v\in U(\ba{c}{a})_{>0}$ so that $x=vu(a)$ and $e=v(a)$.
Notice that $v(\ba{c}{a})=\ba{c}{e}$, so $v^{-1}gv\in \Delta(\epsilon^c_a)$, which
implies that $u'=(v^{-1}gv)u(v^{-1}gv)^{-1}\in U(\ba{c}{a})_{>0}$.
Notice that
$$g(x)=gvu(a)=v (v^{-1}gv)u(a)=v (v^{-1}gv)u(v^{-1}gv)^{-1}(a)=vu'(a) \ \ \textrm{and}\ \ e=v(a),$$
so $(c,g(x),e,a)$ is positive. Since we already know that $(a,c,e,b)$ is positive, this completes the proof.
\end{proof}

Since $(x,z,y,e)$ and $(c,x,z,e)$ are positive,
Lemma \ref{inclusion property} implies that
$$V(x,z,y)\subset V(x,z,e)\subset V(c,z,e).$$
We may further choose $g$ so that $e$ is an attractive point,
in which case, its basin of attraction is $B_c$. In particular, since $x,z\in V(c,z,e)\subset B_c$, 
$$
\lim_{n\to\infty}g^n(x)=\lim_{n\to\infty}g^n(z)=e.
$$
Proposition \ref{Hausdorff limit} and Lemma \ref{pos-diag2} then imply that 
$$
\left\{\overline{V(g^n(x),g^n(z),e)}\right\}\longrightarrow\{e\},$$ 
as $n\to \infty$. Since $V(x,z,y)\subset V(x,z,e)$,
$$V(g^n(x),g^n(z),g^n(y))=g^n(V(x,z,y)\subset g^n(V(x,z,e))=V(g^n(x),g^n(z),e),$$
so
$$
\left\{\overline{V(g^n(x),g^n(z),g^n(y))}\right\}\longrightarrow\{e\}.$$
Since $B_c$ contains a neighborhood of $e$, we see that $$\overline{V(g^n(x),g^n(z),g^n(y))}\subset B_c,$$
for all large enough $n$. So,
$$\overline{V(x,z,y)}
=g^{-n}
\left(\overline{V(g^n(x),g^n(z),g^n(y))}\right)
\subset g^{-n}(B_{c})=B_c.$$
Symmetric arguments show that 
$$\overline{V(x,z,y)}\subset B_e$$
 So, $\overline{V(x,z,y)}$ is a connected subset of $B_c\cap B_e$ which contains $z$.
Therefore,  
$$\overline{V(x,z,y)} \subset V(c,z,e).$$
Since $(a,c,z,e)$ and $(a,z,e,b)$ are positive, Lemma \ref{inclusion property} gives that
$$V(c,z,e)\subset V(a,z,e)\subset V(a,z,b)$$
which completes the proof.
\eproof

\subsection{Rearrangements of flags}

Given a pair $(x,y)$ of transverse flags in $\mathcal F_d$, one obtains a decomposition 
of $\mathbb R^d$ into lines $\{L_i(x,y)\}$. By rearranging the ordering of the lines, one
obtains a collection of flags including $x$ and $y$. Formally, if
$P$  is a permutation of $\{1,\ldots,d\}$, then one obtains flags ${\bf F}_0(P(x,y))$
and ${\bf F}_1(P(x,y))$ given by
$${\bf F}_0(P(x,y))^r=\braket{L_{P(1)}(x,y),\ldots,L_{P(r)}(x,y)}$$
and 
$$ {\bf F}_1(P(x,y))^r=\braket{L_{P(d)}(x,y),\ldots,L_{P(d-r+1)}(x,y)}$$
for all $r$.

We will see that if $(a,x,y,b)$ is positive, then $(a,{\bf F}_1(P(x,y)),b)$ is also positive.
We begin by considering the case where $P$ is a transposition.

\begin{lemma}
\label{transpositionstillpositive}
If  $(a,x,z,y,b)$ is a positive  quintuple in $\mathcal F_d$,
$j>i$ and $P_{i,j}$ is a transposition interchanging $i$ and $j$, then
$${\bf F}_1(P_{i,j}(x,y))\in \overline{V(x,z,y)}\subset V(a,z,b).$$
\end{lemma}

\begin{proof}
Let $\ba{x}{y}$ be a basis for $(x,y)$ so that 
$V(x,z,y)=V(\ba{x}{y})$ and let $\ba{x}{y}=\{e_i\}$.
Let $J_{ij}(t)$ be the elementary Jacobi matrix 
with respect to $\{e_i\}$,  {\it i.e.} $J_{ij}(t)(e_j)=e_j+te_i$ and $J_{ij}(t)(e_k)=e_k$ if $k\ne j$.
Since
$$y^{d-k}= \braket{e_{k+1},\ldots,e_d},$$
we see that
$$J_{i,j}(t)(y^{d-k}) = \braket
{e_{k+1}, \ldots, e_i, \ldots, e_j+te_i,\ldots,e_n}
   = y^{d-k}={\bf F}_1(P_{i,j}(x,y))^{d-k}$$
for all $k < i$,
$$J_{ij}(t)(y^{d-k}) = \braket{e_{k+1},\ldots,e_d}   = y^{d-k}={\bf F}_1(P_{i,j}(x,y))^{d-k}$$
for all $k\ge j$, and
$$J_{ij}(t)(y^{d-k}) = \braket{e_{k+1},  \ldots, e_j+te_i,\ldots,e_d}$$
for all $i \le k < j$.
Therefore, 
$$\lim_{t\rightarrow \infty} J_{ij}(t)(y^{d-k}) = \braket{e_{k+1},\ldots,e_{j-1},e_i,e_{j+1},\ldots,e_d}={\bf F}_1(P_{i,j}(x,y))^{d-k}.$$
for all $i\le k < j$, so 
$$ \lim_{t\rightarrow \infty} J_{ij}(t)(y) = {\bf F}_1(P_{i,j}(x,y)).$$
Since $J_{ij}(t)\in U(\ba{x}{y})_{\ge 0}$ for all $t>0$ and 
$U(\ba{x}{y})_{\ge 0}U(\ba{x}{y})_{>0}\subset U(\ba{x}{y})_{> 0}$, by Lemma \ref{U inclusion},
$J_{ij}(t)(y)\in V(x,z,y)$ for all $t>0$, so ${\bf F}_1(P_{i,j}(x,y))\in\overline{V(x,z,y)}$.
Lemma \ref{bruhat lemma} implies that $\overline{V(x,z,y)}\subset V(a,z,b)$,
so ${\bf F}_1(P_{i,j}(x,y))\in V(a,z,b)$. 
\end{proof}

\medskip

With the help of an elementary group-theoretic lemma, we may generalize the argument
above to handle all permutations.

\begin{lemma}
\label{permutationstillpositive}
If  $(a,x,z,y,b)$ is a positive  quintuple in $\mathcal F_d$ and $P$ is a permutation of
$\{1,\ldots,d\}$, then
$${\bf F}_1(P(x,y))\in \overline{V(x,z,y)}\subset V(a,z,b).$$
\end{lemma}

\begin{proof}
Let $\ba{x}{y}$ be a basis for $(x,y)$ so that 
$V(x,z,y)=V(\ba{x}{y})$ and let $\ba{x}{y}=\{e_i\}$.
Suppose that $Q$ is a permutation such that
$${\bf F}_1(Q(x,y))\subset \overline{V(x,z,y)} \subseteq V(a,z,b).$$
We first observe, as in the proof of Lemma \ref{transpositionstillpositive}, that if $n>m$, then
$$\lim_{t\to\infty}J_{m,n}(t){\bf F}_1(Q(x,y))={\bf F}_1(\hat Q(x,y))$$
where $\hat Q=Q$ if $Q^{-1}(n)>Q^{-1}(m)$ and
$\hat Q=P_{m,n}Q$ if not.
Since  $J_{mn}(t) \in U(\ba{x}{y})_{\ge 0}$ if $t>0$ and
\hbox{$U(\ba{x}{y})_{\ge 0}U(\ba{x}{y})_{>0}\subset U(\ba{x}{y})_{>0}$},
$$J_{m,n}(t)(V(x,z,y))\subset V(x,z,y)\ ,$$
for all $t>0$, 
which implies that
$$J_{m,n}(t)\left(\overline{V(x,z,y)}\right)\subset \overline{V(x,z,y)}$$
for all $t>0$.
Therefore,
$${\bf F}_1(\hat Q(x,y))\subset \overline{V(x,z,y)} \subseteq V(a,z,b).$$

We use the following elementary combinatorial lemma.

\begin{lemma}
\label{permutation fact}
If $P$ is a permutation of $\{1,\ldots,d\}$, then we may write
\begin{equation*}
	P = P_{i_k,j_k}\cdots  P_{i_1,j_i}\ .
	\end{equation*}
	So that $i_l<j_l$ for all $l$ and moreover
	\begin{equation*}
	Q_{l-1}^{-1}(i_l)<Q_{l-1}^{-1}(j_l)\ ,\end{equation*}
	 where $
	 Q_{l-1}\defeq P_{i_{l-1},j_{l-1}}\cdots  P_{i_1,j_1}$.	   
\end{lemma}

We now complete the proof using Lemma \ref{permutation fact}. 
Let $P=P_{i_k,j_k}\cdots  P_{i_1,j_i}$ as in Lemma \ref{permutation fact}.
Lemma \ref{transpositionstillpositive}
implies that 
$${\bf F}_1(Q_1(x,y))\subset \overline{V(x,z,y)} \subseteq V(a,z,b)$$
and we may iteratively apply the observation above to conclude that
$${\bf F}_1(Q_l(x,y))\subset \overline{V(x,z,y)} \subseteq V(a,z,b)$$
for all $l$, which implies that
$${\bf F}_1(P(x,y))\subset \overline{V(x,z,y)} \subseteq V(a,z,b)$$
which completes the proof of Lemma \ref{permutationstillpositive}.
\end{proof}

\medskip\noindent
{\em Proof of Lemma \ref{permutation fact}.}
We proceed by induction on $d$. So assume our claim hold for permutations of $\{1,\ldots,d-1\}$.

Let $r=P^{-1}(1)$ and, if $r\ne 1$, let
$$P_1=P_{1,r}P_{1,r-1}\cdots P_{1,2}$$
and let $P_1=id$ if $r =1$.
Notice that $P_1$ has the desired form, $P_1^{-1}(1)=r$ and if $m,n\in\{1,\ldots,d\}-\{r\}$ and $m<n$,
then $P_1^{-1}(m)<P_1^{-1}(n)$.
Let $\hat P_2$ be the restriction of $PP_1^{-1}$ to $\{2,\ldots,d\}$.
By our inductive claim, $\hat P_2=\hat P_{i_k,j_k}\cdots  \hat P_{i_1,j_i}$ where $i_l<j_l$ for all $l$ and
if \hbox{$\hat Q_{l-1}\defeq\hat P_{i_{l-1}j_{l-1}}\cdots  \hat P_{i_1j_1}$}, 
then $\hat Q_{l-1}^{-1}(i_l)<\hat Q_{l-1}^{-1}(j_l)$.
One may extend each $\hat P_{i_l,j_l}$ to a transposition $P_{i_1,j_l}$ of $\{1,\ldots,d\}$ by letting $1$ be taken
to itself. We then note that
$$P=(P_{i_k,j_k}\cdots  P_{i_1,j_i})P_{1,r}P_{1,r-1}\cdots P_{1,2}$$
has the desired form. 
\eproof

\noindent
{\bf Remark}
Notice that Lemma \ref{transpositionstillpositive} is enough to prove Theorem \ref{transverse-bases-general1} in
the case that you choose exactly one line from $\{L_i(x,y)\}$ and $d-1$ lines from amongst $\{L_i(a,b)\}$. 
(If we choose $z$ so that $(a,x,z,y,b)$ is an positive  quintuple of flags, Lemma \ref{transpositionstillpositive} implies that
${\bf F}_1(P_{j,d}(x,y)) \in V(a,z,b)$, so $(a,{\bf F}_1(P_{j,d}(x,y)),b)$ is a transverse triple of flags. So,
for any $j$ and $k$,
$a^{k-1}\oplus {\bf F}_1(P_{j,d}(x,y))^{1}\oplus b^{d-k}=\mathbb R^d$, which is enough to establish the special case of
Theorem \ref{transverse-bases-general1}.)
This simple case is enough to prove all the results in section \ref{correlation}. The full statement is only used in
the proof of Lemma \ref{length to trace}, and this use of the general result may be replaced by an application
of Labourie's Property H, see \cite{labourie-anosov}.

\subsection{Transverse bases for quadruples}
\label{transverse-bases-proof}
We now restate and prove Theorem  \ref{transverse-bases-general1}.\\

\noindent{\bf Theorem  \ref{transverse-bases-general1}} 
{\em
Let  $(a,x,y,b)$ be a positive quadruple in $\mathcal F_d$, then any  $d$  lines in
$$\{ L_1(a,b),\ldots,L_d(a,b),L_1(x,y),\ldots,L_d(x,y)\}$$
are in general position.
}

\begin{proof}
If 
$$I\in\mathcal I=\{(i_1,\ldots,i_k)\in\mathbb Z^k \ |\ 1\leq i_1<\cdots < i_k\leq d\}.$$
Let
$$e_I(a,b)=e_{i_1}(a,b)\wedge\cdots\wedge e_{i_k}(a,b).$$
Then our claim is equivalent to  the claim that $e_I(a,b)\wedge e_J(x,y)\ne 0$ if $I,J\in\mathcal I$ and
$|I|+|J|=d$ (where $|(i_1,\ldots,i_k)|=k$).

Let $A$ be  the matrix with coefficients $A^i_j = \braket{e^i(a,b) | e_j(x,y)}$.
If $I, K \in\mathcal I$ and $|I| =|K|$,
then let $A^I_K$ be the submatrix of $A$ given by the intersection of the rows with labels in $I$ and  the columns with
labels in $K$.
 
If $I,J \in\mathcal I$ and  $|I|+|J| = d$, then, since
$$e_{j}(x,y) = \sum_{i=1}^d A^{i}_{j} e_i(a,b),$$ 
we see that
$$e_{I}(a,b)\wedge e_J(x,y) =  \det (A^{D-I}_{J}) e_D(a,b)$$
where $D=(1,2,\ldots,d)$.
So, it suffices to prove that all the minors of $A$ are non-zero. Notice that since our bases are well-defined
up to (non-zero) scalar multiplication of the elements, the fact that the minors are non-zero is independent
of our choice of bases.

We first show that all initial minors are non-zero. A square submatrix $A^K_J$
is called {\em initial} if both $J$ and $K$ are contiguous blocks in $D$ and $J\cup K$ contains $1$, 
{\it i.e.} it is  square submatrix which borders the first column or row. An {\em initial minor} is the determinant of an
initial square submatrix.

If $A^{D-I}_J$ is initial and $J$ contains $1$, then 
$$J=(1,\ldots,l)\  \ {\rm and}\ \  
I = (1,2,\ldots,r,d-s+1, d-s+2,\ldots,d)$$
where $r+s + l=d$.
(Notice that either $r$ or $s$ may be $0$.) Since $(a,b,x)\in\mathcal F_d^{(3)}$,
$$a^r\oplus b^s\oplus x^l=\mathbb R^d,$$
so
$$e_I(a,b)\wedge e_J(x,y)\ne 0$$
which implies that $\det(A^{D-I}_J)\ne 0$.

If $D-I$ contains a $1$ and $J$ does not contain a $1$, then 
\begin{eqnarray*}
I&=& (l+1, l+2,\ldots,d)\,  \\
D-I&=&(1,\ldots,l)\ ,\\
J &=& (j+1,j+2,\ldots,j+l)\ ,	
\end{eqnarray*}
where $j,l\ge 1$ and $j+l\le d$. Let $P$ be any permutation such that 
$${\bf F}_1(P(x,y))^{l}=\braket{e_{j+1}(x,y),\ldots,e_{j+l}(x,y)}.$$ 

Then, by  Lemma \ref{permutationstillpositive}, 
$(a,{\bf F}_1(P(x,y)),b)$ is a transverse triple of flags. It follows that 
$$b^{d-l}\oplus {\bf F}_1(P(x,y))^{l}=\mathbb R^d,$$
and hence that
$$e_I(a,b)\wedge e_J(x,y)\ne 0,$$
so again $\det(A^{D-I}_J)\ne 0$.
Therefore we have shown that all the initial minors of $A$ are non-zero.

We claim that if $\xi_0=\nu_\tau$ is the Veronese embedding with respect to an irreducible representation $\tau_d$ and  
$(a_0,x_0,y_0,b_0)$ is an ordered quadruple in $\xi_0({\bf P}^1(\mathbb R))$,
then one may choose bases $\{e_i(a_0,b_0)\}$ and $\{e_i(x_0,y_0\}\}$ so that all the initial minors of
the associated matrix $A_0$ are positive.
We may assume that $a_0=\xi_0(\infty)$, $x_0=\xi_0(t)$, $y_0=\xi_0(1)$ and $b_0=\xi_0(0)$ where $t>1$. Observe that one can
choose bases $\{e_i(0,\infty)\}$ and $\{e_i(1,t)\}$ for $\mathbb R^2$ so that $M_0=\left(\braket{e^i(0,1) | e_j(1,t)}\right)$
is totally positive.  If we choose the bases
$$\{e_i(a_0,b_0)=e_1(0,\infty)^{d-i}e_2(0,\infty)^{i-1}\}\quad{\rm and}\quad \{e_i(x_0,y_0)=e_1(1,t)^{d-i}e_2(1,t)^{i-1}\}$$
for $\mathbb R^d$,
then $A_0=\tau_d(M_0)$. The claim then follows from the fact
that the the image under $\tau_d$  of  a totally positive matrix in $\pslt$ is totally positive in $\psln$,
see \cite[Prop. 5.7]{fock-goncharov}.

We can now continuously deform  $(a,x,y,b)=(a_1,x_1,y_1,b_1)$, through positive quadruples $(a_t,x_t,y_t,b_t)$,
to a positive quadruple $(a_0,x_0,y_0,b_0)$ in the image of $\xi_0=\nu_\tau$. 
One may then continuously choose bases
$\{e_i(a_t,b_t)\}$ and $\{e_i(x_t,y_t)\}$ beginning at $\{e_i(a_0,b_0)\}$ and $\{e_i(x_0,y_0\}$ and terminating
at bases $\{e_i(a,b)\}$ and $\{e_i(x,y)\}$ which we may assume are the bases used above.  One
gets associated matrices $\{ A_t\}$ interpolating between $A_0$ and $A$. Since the initial minors of $A_t$ are
non-zero for all $t$ and positive for $t=0$, we see that the initial minors of $A$ must be positive.

Gasca and Pena \cite[Thm. 4.1]{gasca-pena} (see also Fomin-Zelevinsky \cite[Thm. 9]{fomin-zelevinsky})
proved that a matrix is totally positive if and only if all its  initial minors are positive.
Therefore, $A$ is totally positive, so all its minors are positive, hence non-zero, which completes the proof.
\end{proof}

\section{Correlation functions for Hitchin representations}
\label{correlation}

We define correlation functions which offer measures of the transversality of bases
associated to images of collections of elements in $\pi_1(S)$. The results of the previous
section can be used to give conditions guaranteeing that many of these correlation functions are non-zero.
We then observe that, if we restrict to certain 3-generator subgroups of $\pi_1(S)$, then the restriction of
the Hitchin representation function to the subgroup is determined, up to conjugation, by correlation
functions associated to the generators and the eigenvalues of the images of the generators.

If $\{\alpha_1,\ldots,\alpha_n\}$ is a collection of non-trivial elements of $\pi_1(S)$,
$i_j\in \{0,1,\ldots,d\}$ for all $1\le j\le n$, and $\rho\in\Hn$, we define the {\em correlation function} 
\footnote{The name ``correlation function'' does not  bear any physical meaning here and just reflects  the fact that the 
correlation function between eigenvalues of quantum observables is the  trace of products of 
projections on the corresponding eigenspaces.}

$${\bf T}_{i_1,\ldots,i_n}(\alpha_1,\ldots,\alpha_n)(\rho) \defeq
\tr\left(\prod_{j=1}^n \p_{i_j}(\rho(\alpha_j))\right).$$
where we adopt the convention that
$$\p_0(\rho(\alpha))=\rho(\alpha).$$
Notice that if all the indices are non-zero, then ${\bf T}_{i_1,\ldots,i_n}(\alpha_1,\ldots,\alpha_n)(\rho)$ is
well-defined, while if some indices are allowed to be zero, ${\bf T}_{i_1,\ldots,i_n}(\alpha_1,\ldots,\alpha_n)(\rho)$
is only well-defined up to sign. These correlations functions are somewhat more general than the correlation
functions defined in the introduction as we allow terms which are not projection matrices.

\subsection{Nontriviality of correlation functions}
We say that a collection $\{\alpha_1,\ldots,\alpha_n\}$ of non-trivial elements of $\pi_1(S)$
has {\em non-intersecting axes} if whenever $i\ne j$, $(\alpha_i)_+$ and $(\alpha_i)_-$ lie
in the same component of $\partial_\infty\pi_1(S)-\{(\alpha_j)_+,(\alpha_j)-\}$.
Notice that $\{\alpha_1,\ldots,\alpha_n\}$ have non-intersecting axes whenever they are
represented by mutually disjoint  and non-parallel simple closed curves on $S$.

Theorem \ref{transverse-bases-general0} has the following immediate consequence.

\begin{corollary}
\label{productnonzero}
If $\rho\in\Hn$, $\alpha,\beta\in\pi_1(S)-\{1\}$ and $\alpha$ and $\beta$ have non-intersecting
axes,  then any $d$ elements of
$$\{e_1(\rho(\alpha)),\ldots,e_{d}(\rho(\alpha)),e_{1}(\rho(\beta)),\ldots,e_d(\rho(\beta))\}$$
span $\mathbb R^d$. In particular,
$$\braket{e^i(\rho(\alpha)) | e_j(\rho(\beta))} \neq 0.$$ 
\end{corollary}

One can use Corollary \ref{productnonzero}  to establish
that a variety of correlation functions are non-zero. Notice that the
assumptions of Lemma \ref{tracenonzero-i0} will be satisfied whenever $\alpha$
is represented by a simple curve and $\alpha$ and $\gamma$ are co-prime.

\begin{lemma}
\label{tracenonzero-i0}
If $\rho\in\Hn$ , $\alpha, \gamma\in\pi_1(S)-\{1\}$, 
$\alpha$ and $\gamma\alpha\gamma^{-1}$ have non-intersecting axes, and
$i\in\{1,\ldots,d\}$, then
$${\bf T}_{i,0}(\alpha,\gamma)(\rho)=\tr\left(\p_i(\rho(\alpha))\rho(\gamma)\right)\ne 0.$$
\end{lemma}

\begin{proof}
Since 
$$\tr(\p_i(\rho(\alpha))\rho(\gamma))= \braket{e^i(\rho(\alpha)) | \rho(\gamma)(e_i(\rho(\alpha)))}=\braket{e^i(\rho(\alpha)) | e_i(\rho(\gamma\alpha\gamma^{-1}))},$$
the lemma follows immediately from  Corollary \ref{productnonzero}
\end{proof}

The next result deals with correlation functions which naturally arise when studying
configurations of elements of $\pi_1(S)$ used in the proof of Theorem \ref{lengthrigidity}, see
Figure \ref{4curves}.

\begin{proposition}
\label{tracenonzero}
Suppose that $\rho\in\Hn$ , $\alpha, \beta,\delta\in\pi_1(S)-\{1\}$ have non-intersecting axes, and
$i,j,k\in\{1,\ldots,d\}$. Then
\begin{enumerate}
\item
$${\bf T}_{ij}(\alpha,\beta)(\rho)=\tr(\p_i(\rho(\alpha))\p_j(\rho(\beta)))\ne 0,$$
and
\item
$${\bf T}_{i,j,k}(\alpha,\beta,\delta)(\rho)=\tr(\p_{i}(\rho(\alpha))\p_j(\rho(\beta))\p_{k}(\rho(\delta))) \ne 0 .$$

Moreover, if $\gamma\in\pi_1(S)-\{1\}$ and $\beta$ and $\gamma\delta\gamma^{-1}$ have non-intersecting
axes, then
\item
$${\bf T}_{i,0,j}(\beta,\gamma,\delta)(\rho)=\tr(\p_{i}(\rho(\beta))\rho(\gamma)\p_{j}(\rho(\delta))) \ne 0,$$
and
\item
$${\bf T}_{i,j,0,k}(\alpha,\beta,\gamma,\delta)(\rho) = 
{\bf T}_{j,0,k}(\beta,\gamma,\delta)(\rho)\left(\frac{{\bf T}_{i,j,k}(\alpha,\beta,\delta)(\rho)}{{\bf T}_{j,k}(\beta,\delta)(\rho)}\right)\ne 0.$$

\end{enumerate}
\end{proposition}

\begin{proof}
Notice that
$$\tr(\p_i(\rho(\alpha))\p_j(\rho(\beta))) = \braket{e^i(\rho(\alpha)) | e_j(\rho(\beta))} \braket{e^j(\rho(\beta)) | e_i(\rho(\alpha))}$$
for all $i$ and $j$. Both of the terms on the right-hand side are non-zero, by Corollary \ref{productnonzero},
so 
$${\bf T}_{ij}(\alpha,\beta)(\rho)=\tr(\p_i(\rho(\alpha))\p_j(\rho(\beta)))\ne 0.$$
Similarly,
$${\bf T}_{i,j,k}(\alpha,\beta,\delta)(\rho)=
\braket{e^i(\rho(\alpha)) | e_j(\rho(\beta))} \braket{e^j(\rho(\beta)) | e_k(\rho(\delta))} \braket{e^k(\rho(\delta)) | e_i(\rho(\alpha)) }$$
and  Corollary \ref{productnonzero} 
guarantees that each of the terms on the right hand side is non-zero, so (1) and (2) hold.

Since
\begin{eqnarray*}
\tr(\p_{i}(\rho(\beta))\rho(\gamma)\p_{j}(\rho(\delta)))  & = &
 \braket{e^i(\rho(\beta))|\rho(\gamma) (e_j(\rho(\delta)))}\braket{e^j(\rho(\delta)) | e_i(\rho(\beta))} \\
 & =& \braket{e^i(\rho(\beta))|e_j(\rho(\gamma\delta\gamma^{-1}))}\braket{e^j(\rho(\delta)) | e_i(\rho(\beta))}\\
 \end{eqnarray*}
Corollary \ref{productnonzero} again
guarantees that each of the terms on the right hand side is non-zero, so (3) holds.

Recall that if $P,Q,A\in\sln$ and $P$ and $Q$
are projections onto lines, then 
$$PAQ = \frac{\tr(PAQ)}{\tr(PQ)}PQ$$
if $\tr(PQ) \neq 0$. (Suppose that $P$ projects onto the line $\braket{v}$ with kernel the hyperplane $V$ 
and $Q$ project onto the line $\braket{w}$ with kernel the hyperplane $W$, then both $PAQ$ and $PQ$ map onto
the line $\braket{v}$ and have $W$ in their kernel and  are therefore multiples of one another.
The ratio of the traces detects this multiple.)

So, since $\tr(\p_j(\rho(\beta))\p_k(\rho(\delta))\ne 0$, 
$$\p_j(\rho(\beta)) \rho(\gamma) \p_k(\rho(\delta))= 
\left(\frac{\tr(\p_j(\rho(\beta))\rho(\gamma) \p_k(\rho(\delta)))}{\tr(\p_j(\rho(\beta))\p_k(\rho(\delta)))}\right)\p_j(\rho(\beta))\p_k(\rho(\delta)).$$
Therefore,
\begin{eqnarray*}
{\bf T}_{i,j,0,k}(\alpha,\beta,\gamma,\delta)(\rho) &=&
\tr(\p_i(\rho(\alpha)) \p_j(\rho(\beta))\rho(\gamma) \p_k(\rho(\delta)))\\
& = &
\tr\left(\p_i(\rho(\alpha))\left(\frac{\tr(\p_j(\rho(\beta))\rho(\gamma) \p_k(\rho(\delta))}{\tr(\p_j(\rho(\beta))\p_k(\rho(\delta))}\right)\p_j(\rho(\beta))\p_k(\rho(\delta))\right) \\
& = & \tr(\p_j(\rho(\beta)) \rho(\gamma)  \p_k(\rho(\delta))\left(\frac{\tr(\p_i(\rho(\alpha))\p_j(\rho(\beta))\p_k(\rho(\delta))}{\tr(\p_j(\rho(\beta))\p_k(\rho(\delta))}\right)\\
& = & {\bf T}_{j,0,k}(\beta,\gamma,\delta)(\rho)\left(\frac{{\bf T}_{i,j,k}(\alpha,\beta,\delta)(\rho)}{{\bf T}_{j,k}(\beta,\delta)(\rho)}\right).\\
\end{eqnarray*}
Since all the terms on the right hand side have already been proven to be non-zero, the entire expression is
non-zero, which completes the proof of (4).
\end{proof}

\subsection{Correlation functions and eigenvalues rigidity}

We now observe that correlation functions and eigenvalues of images of elements determine
the restriction of a Hitchin representation up to conjugation. Theorem \ref{conjugateontriples} is a special
case of Theorem \ref{conjugateontriples-general}.

\begin{theorem}
\label{conjugateontriples-general}
Suppose that $\rho,\sigma\in\Hn$ and $\alpha, \beta,\delta \in \pi_1(S)-\{1\}$ have non-intersecting axes.
If
\begin{enumerate}
	\item $\lambda_i(\rho(\eta))=\lambda_i(\sigma(\eta))$ for any  $\eta \in \{\alpha,\beta,\delta\}$
	and any $i\in\{1,\ldots,d\}$, and
	\item for all $i,j,k$ in $\{1,\ldots,d\}$
 $$\frac{{\bf T}_{i,j,k}(\alpha,\beta,\delta)}{{\bf T}_{j,k}(\beta,\delta)}(\rho)=
\frac{{\bf T}_{i,j,k}(\alpha,\beta,\delta)}{{\bf T}_{j,k}(\beta,\delta)}(\sigma),$$
\end{enumerate}
then $\rho$ and $\sigma$
are conjugate, in $\pgln$, on the subgroup 
$\braket{\alpha,\beta,\delta}$ of $\pi_1(S)$ generated by $\alpha$, $\beta$ and $\delta$.
\end{theorem}

\begin{proof}
We will work in lifts of the restrictions of $\rho$ and $\sigma$ to $\braket{\alpha,\beta,\delta}$ so that
the images of $\alpha$, $\beta$ and $\delta$ all have positive eigenvalues. We will abuse notation
by referring to these lifts by simply $\rho$ and $\sigma$. With this convention,
$\lambda_i(\rho(\eta))=\lambda_i(\sigma(\eta))$ for all $i$ and any $\eta\in\{\alpha,\beta,\delta\}$.
It suffices to prove that these lifts are conjugate in $\ms{GL}_d(\mathbb R)$.

Let $a_i=e_i(\rho(\alpha))$, $a^i=e^i(\rho(\alpha))$, $b_j=e_j(\rho(\beta))$, $b^j=e^j(\rho(\beta))$,
$d_k=e_k(\rho(\delta))$ and \hbox{$d^k=e^k(\rho(\delta))$} for all $i,j,k$. Similarly let $\hat a_i=e_i(\sigma(\alpha))$, 
\hbox{$\hat a^i=e^i(\sigma(\alpha))$}, $\hat b_j=e_j(\sigma(\beta))$, $\hat b^j=e^j(\sigma(\beta))$,
$\hat d_k=e_k(\sigma(\delta))$ and $\hat d^k=e^k(\sigma(\delta))$ for all $i,j,k$.
With this notation,
$$\frac{{\bf T}_{i,j,k}(\alpha,\beta,\delta)(\rho)}{{\bf T}_{j,k}(\beta,\delta)(\rho)}= \frac{\braket{a^i | b_j }\braket{b^j | d_k }\braket{d^k | a_i}}{\braket{b^j | d_k}\braket{d^k | b_j }} =\frac{\braket{a^i | b_j }\braket{d^k | a_i}}{\braket{d^k | b_j }} $$
and
$$\frac{{\bf T}_{i,j,k}(\alpha,\beta,\delta)(\sigma)}{{\bf T}_{j,k}(\beta,\delta)(\sigma)} =\frac{\braket{\hat a^i | \hat b_j }\braket{\hat d^k | \hat a_i}}{\braket{\hat d^k | \hat b_j }} ,$$
so, by assumption,
\begin{equation}
\frac{\braket{a^i | b_j }\braket{d^k | a_i}}{\braket{d^k | b_j }} = \frac{\braket{\hat a^i |\hat b_j }\braket{\hat d^k |\hat a_i}}{\braket{\hat d^k |\hat b_j }}
\label{coordinatesequal}
\end{equation}

We may conjugate $\sigma$  and choose $a_i$, $\hat a_i$, $b_1$ and $\hat b_1$ so that $a_i=\hat a_i$ for all $i$ 
(so $a^i=\hat a^i$ for all $i$),
$b_1=\hat b_1$ and $\braket{a^i | b_1}=1$ for all $i$.
(Notice that this is possible since, by  Corollary \ref{productnonzero}, $b_1$ does not
lie in any of the coordinate hyperplanes of the basis $\{a_i\}$ and similarly $\hat b_1$ does not lie in
any of the coordinate hyperplanes of the basis $\{\hat a_i\}=\{a_i\}$.)
Therefore, since $\lambda_i(\rho(\alpha))=\lambda_i(\sigma(\alpha))$ for all $i$,
we see that $\rho(\alpha)=\sigma(\alpha)$.

Corollary \ref{productnonzero} also assures us that  $\braket{  d^k | b_1 }$  and $\braket{\hat d^k  | \hat b_1 }$ are non-zero, so
we may additionally choose $\{d^k\}$ and $\{\hat d^k\}$ so that $\braket{  d^k | b_1 }=1$  and $\braket{\hat d^k  | \hat b_1 }=1$
for all $k$.
Therefore, taking $j=1$ in Equation (\ref{coordinatesequal}), we see that
$$ \braket{ d^k | a_i}=\braket{ \hat d^k | \hat a_i }=\braket{ \hat d^k | a_i }$$
for all $i$ and $k$.
It follows that  $d^k=\hat d^k$ for all $k$, which implies that $d_k=\hat d_k$ for all $k$.
Again,  since $\lambda_i(\rho(\delta))=\lambda_i(\sigma(\delta))$ for all $i$,  we see that
$\rho(\delta)=\sigma(\delta)$.

Equation (\ref{coordinatesequal}) then reduces to
$$\frac{\braket{a^i | b_j }}{\braket{  d^k | b_j } }= \frac{\braket{\hat a^i |\hat  b_j }}{\braket{\hat  d^k  | \hat b_j }}= \frac{\braket{ a^i |\hat  b_j }}{\braket{ d_k  | \hat b_j}}.$$
We may assume, again applying Corollary \ref{productnonzero}, that $\{b_j\}$ and $\{\hat b_j\}$ have been chosen so that
$$\braket{ a^1 | b_j }=\braket{ a^1 | \hat b_j }=1$$
for all $j$, so, by considering the above equation with $i=1$, we see that
$$\braket{ d^k | b_j}=\braket{ d^k   |\hat b_j }$$
for all $j$ and $k$, which implies that  $b_j=\hat b_j$ for all $j$, and, again since eigenvalues agree,
we may conclude that  $\rho(\beta)=\sigma(\beta)$, which completes the proof.
\end{proof}

\section{Asymptotic expansion of spectral radii}

In this section we establish a useful asymptotic expansion for the spectral radii of families of matrices
of the form $A^nB$.

\begin{lemma}
\label{specradiusexpansion}
Suppose that $A, B \in \sln$ and that $A$ is real-split and \hbox{2-proximal}.
If $(b^i_j)$ is the matrix of $B$ with respect to $\{e_i(A)\}_{i=1}^d$ 
and $b^1_1$, $b^1_2$, and $ b^2_1$ are non-zero, then
$$\frac{\lambda_1(A^nB)}{\lambda_1(A)^n} = b^1_1 + \frac{b^1_2b^2_1}{b^1_1} \left(\frac{\lambda_2(A)}{\lambda_1(A)}\right)^n+ 
o\left(\left(\frac{\lambda_2(A)}{\lambda_1(A)}\right)^n\right).$$
 \end{lemma}

We begin by showing that the spectral radius is governed by an analytic function.

\begin{lemma}
\label{spec analytic}
Suppose that $A, B \in \sln$ and that $A$ is real-split and proximal.
If $(b^i_j)$ is the matrix of $B$ with respect to $\{e_i(A)\}_{i=1}^d$ and 
$b^1_1$ is non-zero, then
there exists an open neighborhood $V \subseteq \Real^{d-1}$ of the origin and an analytic function 
$f:V \rightarrow \Real$ such that, for all sufficiently large $n$,
$$\frac{\lambda_1(A^nB)}{\lambda_1(A)^n} =f(z_1^n,\ldots,z_{d-1}^n)$$
where $z_i=\frac{\lambda_{i+1}(A)}{\lambda_1(A)}$ for all $i$.

Moreover, there exists an analytic function $X:V\to\mathbb R^d$ such that 
$X(z_1^n,\ldots,z_{d-1}^n)$
is an eigenvector of $A^nB$ with eigenvalue
$\lambda_1(A^nB)$ for all sufficiently large $n$.
\end{lemma}

\begin{proof}
The proof is based on the following elementary fact from linear algebra. A proof in the case that
$U$ is one-dimensional is given explicitly in Lax \cite[Section 9, Theorem 8]{lax-linear-algebra} but the proof
clearly generalizes to our setting.

\begin{lemma}
Suppose that $\{M(u)\}_{u\in U}$ is analytically varying family of \hbox{$d\times d$} matrices,
where $U$ is an open neighborhood of $0$ in $\mathbb R^n$.
If $M(0)$ has a simple real eigenvalue $\lambda_0 \neq 0$ with associated unit eigenvector $X_0$,
then there exists an open sub-neighborhood $V \subseteq U$ of $0$ and analytic functions $f: V \rightarrow \Real,$
and $X: V \rightarrow \Real^d$ such that $f(0) = \lambda_0$, $X(0) = X_0$ and $f(v)$ is a simple eigenvalue of $M(v)$ with eigenvector $X(v)$ for all $v \in V$. 
\label{eigenvalue-analytic}
\end{lemma}

Let $U=\mathbb R^{d-1}$ and, for all $u\in U$, let $D(u)$ be the diagonal matrix, with respect to $\{e_i(A)\}$,
with entries $(1,u_1,\ldots,u_{d-1})$ and let $M(u)= D(u)B$ for all $u\in U$. 
Then $M(0)$ has $b^1_1$ as its only non-zero eigenvalue with  associated unit eigenvector $e_1$. 
So we may apply Lemma \ref{eigenvalue-analytic} with $\lambda_0 = b^1_1$ and $X(0) = e_1$.
Let $V$ be the open neighborhood and $f:V\to \mathbb R$ and $X:V\to \mathbb R^d$ be the analytic
functions provided by that lemma. Further, as $M(0)$ has only one non-zero eigenvalue, we can choose $V$ such that the eigenvalue  $f(u)$ is the maximum modulus eigenvalue of $M(u)$. For sufficiently large $n$,
\hbox{$(z_1^n,\ldots,z_{d-1}^n)\in V$},
and $ \frac{A^nB}{\lambda_1(A)^n}=M(z_1^n,\ldots,z_{d-1}^n)$.
So, for all sufficiently large $n$, 
$f(z_1^n,\ldots,z_{d-1}^n)$ is the  eigenvalue of maximal modulus of $A^nB/\lambda_1(A)^n$
with associated eigenvector 
$X(z_1^n,\ldots,z_{d-1}^n).$
\end{proof}

\noindent
{\em Proof of Lemma \ref{specradiusexpansion}.}
Since $A$ is 2-proximal,
$$|\lambda_1(A)| > |\lambda_2(A)| \geq \lambda_{3}(A)|\ldots \geq \lambda_{d}(A)|.$$
Let $f:V\to \mathbb R$ be the function provided by Lemma \ref{spec analytic}.
If $z_i = \lambda_{i+1}(A)/\lambda_1(A)$, then \hbox{$(z_1^n, \ldots,z_{d-1}^n) \in V$}, so 
$$\frac{\lambda_1(A^nB)}{\lambda_1(A)^n} = f\left(z_1^n, \ldots,z_{d-1}^n\right)$$
for all large enough $n$. Since $f$ is analytic
$$f(u_1,\ldots,u_{d-1}) = f(0) + \sum_{i=1}^{d-1}\frac{\partial f}{\partial u_i}(0) u_i + O(u_i u_j).$$
If
$$g(s) = f(s,0,\ldots,0) = \lambda_1(D(1,s,0,\ldots,0)B) = \lambda_1\left(
\begin{bmatrix}  
b^1_1 & b^1_2 & b^1_3&\ldots &b^1_d \\ 
sb^2_1 & sb^2_2 & sb^2_3&\ldots &sb^2_d \\ 
0 & 0 & 0&\ldots &0 \\ 
\vdots & \vdots & \vdots  &\vdots &\vdots\\ 
0 & 0 & 0&\ldots &0 
\end{bmatrix} \right) = \lambda_1\left(\begin{bmatrix}  b^1_1 & b^1_2 \\ sb^2_1& sb^2_2 \end{bmatrix}\right),$$
then we see, by examining the characteristic equation, that
$$g(s)^2 -(b^1_1+sb^2_2)g(s)+s(b^1_1b^2_2-b^1_2b^2_1) =0$$
Differentiating and applying the fact that $g(0) = f(0) = b^1_1$ yields
$$0 = 2g(0)g'(0)-b^1_1g'(0)-b^2_2g(0) + (b^1_1b^2_2-b^1_2b^2_1)  = b^1_1g'(0) -b^1_2b^2_1,$$
so
\begin{equation*}
\frac{\partial f}{\partial u_1}(0) = g'(0)  = \frac{b^1_2b^2_1}{b^1_1}.
\label{pd1}
\end{equation*}
Since $|z_{i}| < |z_1|$ for all $i \geq 2$,
\begin{eqnarray*}
\frac{\lambda_1(A^nB)}{\lambda_1(A)^n} & = & f\left(z_1^n, \ldots,z_{d-1}^n\right)= f(0) + \sum_{i=1}^{d-1}\frac{\partial f}{\partial u_i}(0) z_i^n + o(z_1^n)\\
 &= & b^1_1 + \frac{b^1_2b^2_1}{b^1_1} z_1^n+ o\left(z_1^n\right).\\
\end{eqnarray*}
\qed

\section{Simple lengths and traces}
\label{lengthtrace}

We show that two Hitchin representations have the same simple non-separating length spectrum 
if and only if they have the same simple non-separating trace spectrum. Moreover, in either case all
eigenvalues of images of simple non-separating curves agree up to sign.

\begin{theorem}
\label{equivalence}
If $\rho, \in  {\mathcal H}_{d_1}(S), \sigma \in {\mathcal H}_{d_2}(S)$, then
$|\tr(\rho(\alpha))|=| \tr(\sigma(\alpha))|$ for any \hbox{$\alpha\in\pi_1(S)$} represented by a simple non-separating curve
on $S$ if and only if  \hbox{$L_\alpha(\rho) = L_\alpha(\sigma)$}
for any $\alpha\in\pi_1(S)$ represented by a simple non-separating curve on $S$.
In either case,  $d_1 = d_2$, and $\lambda_i(\rho(\alpha))=\lambda_i(\sigma(\alpha))$ for all $i$ and
any $\alpha\in\pi_1(S)$ represented by a simple non-separating curve on $S$.
\end{theorem}

Theorem \ref{equivalence} follows immediately from Lemma \ref{trace to length}, which shows that one can detect the length
of a curve from the traces of a related family of curves, and Lemma \ref{length to trace}, which obtains information
about traces and eigenvalues from information about length.

\begin{lemma}
\label{trace to length}
Suppose that $\alpha$ and $\beta$ are represented by  simple based loops on $S$ which intersect 
only at the basepoint  and have geometric intersection one.
If $\rho \in  {\mathcal H}_{d_1}(S) ,\sigma \in  {\mathcal H}_{d_2}(S)$ and
$|\tr(\rho(\alpha^n\beta))|=|\tr(\sigma(\alpha^n\beta))|$ for all $n$, then $d_1 = d_2$ and $L_\alpha(\rho) = L_\alpha(\sigma)$.
Moreover, $\lambda_i(\rho(\alpha))=\lambda_i(\sigma(\alpha))$ for all $i$.
\end{lemma}

\begin{proof}
We assume that $d_1 \leq d_2$. 
It suffices to prove our lemma for lifts of the restriction of $\rho$ and $\sigma$ to $\braket{\alpha,\beta}$ 
so that the all the eigenvalues of the images of $\alpha$ are positive.
We will abuse notation by calling these lifts $\rho$ and $\sigma$.

Since $\tr(\rho(\alpha^n\beta)) =\epsilon(n)\tr(\sigma(\alpha^n\beta))$ for all $n$, where $\epsilon(n) \in\{\pm 1\}$,
we may expand to see that
 $$\sum_{i=1}^{d_1} \lambda_i^n(\rho(\alpha))\tr(\p_i(\rho(\alpha))\rho(\beta)) = 
 \epsilon(n)\sum_{i=1}^{d_2} \lambda_i^n(\sigma(\alpha))\tr(\p_i(\sigma(\alpha))\sigma(\beta))$$
for all $n$. 
Lemma \ref{tracenonzero-i0} implies that $\tr(\p_i(\rho(\alpha))\rho(\beta))$ and
$\tr(\p_i(\sigma(\alpha))\sigma(\beta))$ are non-zero for all $i$.
There exists an infinite subsequence $\{n_k\}$  of integers, so that $\epsilon(n_k) = \epsilon$
is constant.  Passing to limits as $n \rightarrow \infty$, and comparing the leading terms  in descending order,
we see that
$\lambda_i(\rho(\alpha))=\lambda_i(\sigma(\alpha))$ if $1 \leq i \leq d_1$. 
In particular, $L_\alpha(\rho)=L_\alpha(\sigma)$. If $d_1 < d_2$, then 
$$\Pi_{i=1}^{d_1} \lambda_i(\rho(\alpha))=\Pi_{i=1}^{d_2} \lambda_i(\sigma(\alpha)) =  1$$
which is impossible, since $\lambda_i(\rho(\alpha))=\lambda_i(\sigma(\alpha))$ if $1 \leq i \leq d_1$ and
$$\lambda_i(\sigma(\alpha))\le\lambda_{d_1}(\sigma(\alpha))=\lambda_{d_1}(\rho(\alpha))<1$$
if $d_1<i\le d_2$.
Therefore $d_1 = d_2$. 
\end{proof}

\begin{lemma}
\label{length to trace}
Suppose that $\gamma$ and $\delta$ are represented by  simple based loops on $S$ which intersect 
only at the basepoint  and have geometric intersection  one.
If $\rho \in  {\mathcal H}_{d_1}(S), \sigma \in  {\mathcal H}_{d_2}(S)$ and
$L_{\beta}(\rho)=L_{\beta}(\sigma)$ whenever $\beta\in\braket{\gamma,\delta}$ is represented by
a simple non-separating based loop, then $d_1 = d_2$,  
$|\tr(\rho(\alpha))|=|\tr(\sigma(\alpha))|$ and
$\lambda_i(\rho(\alpha))=\lambda_i(\sigma(\alpha))$ for all $i$ 
whenever $\alpha\in\braket{\gamma,\delta}$ is represented by
a simple non-separating based loop.
\end{lemma}

\begin{proof} 
We assume that $d_1 \leq d_2$.  If $\alpha\in\braket{\gamma,\delta}$ is represented by a simple, non-separating based loop, then there exists
$\beta\in\braket{\gamma,\delta}$  so that $\beta$ is represented by a simple based loop which intersects
$\alpha$ only at the basepoint and $\alpha$ and $\beta$ have geometric intersection  one, so $\alpha^n\beta$ is simple
and non-separating for all $n$.
It again suffices to prove our lemma for lifts of the restriction of $\rho$ and $\sigma$ to $\braket{\alpha,\beta}$ 
so that the all the eigenvalues of the images of $\alpha$ are positive.

Let $A = \rho(\alpha)$,  $B = \rho(\beta)$, $\hat A = \sigma(\alpha)$,  and $\hat B = \sigma(\beta)$.
Let $\lambda_i=\lambda_i(A)$ and $\hat\lambda_i=\lambda_i(\hat A)$.
Let $e_i=e_i(A)$ and $\hat e_i=e_i(\hat A)$ and let $(b^i_j)$ be the matrix of $B$ with respect to $\{e_i\}_{i=1}^{d_1}$
and $(\hat b^i_j)$ be the matrix of $\hat B$ with respect to $\{\hat e_i\}_{i=1}^{d_2}$.
Let $\Omega = e_1\w e_2\w \ldots \w e_{d_1} \neq 0$ be the volume form associated to the basis $\{e_1\}_{i=1}^{d_1}$ for
$\mathbb R^{d_1}$.

We begin by showing that $\lambda_2=\hat\lambda_2$.
Notice that $A$ and $A^nB$  are real-split and 2-proximal for all $n$.
We need the result of the following lemma to be
able to apply Lemma \ref{specradiusexpansion}.

\begin{lemma}
\label{nonzero coefficients}
Suppose that $\alpha$ and $\beta$ are represented by  simple based loops on $S$ which intersect 
only at the basepoint  and have geometric intersection  one. If $\rho\in\Hn$ and $B=(b_i^j)$ is a
matrix representing $\rho(\beta)$ in the basis $\{e_i(\rho(\alpha)\}$, then 
$b^1_1$, $b^1_2$, and $b^2_1$ are all non-zero.
\end{lemma}

\begin{proof}
Notice that $B(e_1)\w (e_{2}\w \ldots  \w e_d) = b^1_1\Omega$. So, if $b^1_1=0$, then 
$B(e_1)$, which is a non-trivial multiple of $e_1(\rho(\beta\alpha\beta^{-1}))$,
lies in the hyperplane spanned by
$\{e_2,\ldots,e_d\}=\{e_2(\rho(\alpha)),\ldots,e_d(\rho(\alpha)\}$,
which contradicts Corollary \ref{productnonzero} (and also hyperconvexity).
Notice that the fixed points of $\beta\alpha\beta^{-1}$ must lie in the same component of $\xi_\rho(S^1)-\{\alpha^+,\alpha^-\}$,
since $\alpha$ is simple.
Therefore, $b_1^1\ne 0$.

Similarly, $B(e_1)\w(e_{1}\w e_{3}\w \ldots  \w e_d) = -b^2_1\Omega$. 
So, if $b^2_1 = 0$, then $e_1(\rho(\beta\alpha\beta^{-1}))$,
lies in the hyperplane spanned by
$\{e_1(\rho(\alpha)),e_3(\rho(\alpha)),\ldots,e_d(\rho(\alpha))\}$,
which again contradicts Corollary \ref{productnonzero}.
Therefore, $b^2_1 \neq 0$.

Moreover,
$B(e_2)\w(e_2\w e_{3}\w \ldots  \w e_d) = b^1_2 \Omega$. 
So, if $b^1_2 = 0$, then $e_2(\rho(\beta\alpha\beta^{-1}))$,
lies in the hyperplane spanned by
$\{e_1(\rho(\alpha)),e_3(\rho(\alpha),\ldots,e_d(\rho(\alpha)\}$,
which again contradicts Corollary \ref{productnonzero}.
Thus, $b^1_2\neq 0$.
\end{proof}

By assumption $|\lambda_1(A^nB)| = |\lambda_1(\hat A^n\hat B)|$ for all $n$. 
Lemma \ref{specradiusexpansion} then implies that
$$\left| b^1_1 + \frac{b^1_2b^2_1}{b^1_1} \left(\frac{\lambda_2}{\lambda_1}\right)^n+ o\left(\left(\frac{\lambda_2}{\lambda_1}\right)^n\right) \right|= \left|\hat b^1_1 + \frac{\hat b^1_2\hat b^2_1}{\hat b^1_1} \left(\frac{\hat\lambda_2}{\hat\lambda_1}\right)^n+o\left(\left(\frac{\hat\lambda_2}{\hat\lambda_1}\right)^n\right)\right|,$$
so
$|b^1_1| = |\hat b^1_1|.$
Comparing  the second order terms, we see that
$$\frac{\lambda_2}{\lambda_1}= \frac{\hat\lambda_2}{\hat\lambda_1}.$$
Since, by assumption, $\lambda_1=\hat\lambda_1$, we see that $\lambda_2=\hat\lambda_2$.

We now assume  that for some $k=2,\ldots, d_1-1$,
$\lambda_i(\rho(\beta))=\lambda_i(\sigma(\beta))$ for all $i\le k$ whenever $\beta\in\braket{\gamma,\delta}$ is represented
by a simple, non-separating based loop.
We will prove that this implies that 
$\lambda_i(\rho(\alpha))=\lambda_i(\sigma(\alpha))$ for all $i\le k+1$ whenever $\alpha\in\braket{\gamma,\delta}$ is represented
by a simple, non-separating based loop.
Applying this iteratively  will allow us to complete the proof.

Let $E^{k}(\rho)$ be the $k^{th}$-exterior product representation. 
If $\alpha\in\braket{\gamma,\delta}$ is represented by a simple non-separating based loop, we again choose
$\beta\in\braket{\gamma,\delta}$  so that $\beta$ is represented by a simple based loop which intersects
$\alpha$ only at the basepoint and $\alpha$ and $\beta$ have geometric intersection one. We adapt the notations
and conventions from the second paragraph of the proof.

Let $C= E^{k}(\rho)(\alpha)$, $D=E^k(\rho)(\beta)$, $\hat C=E^k(\sigma)(\alpha)$ and
$\hat D=E^k(\sigma)(\beta)$. Notice that $C$ and $C^nD$ are real-split and 2-proximal for all $n$.
If $c_i=e_i(C)$, then we may assume that each $c_i$ is a $k$-fold wedge product of
distinct $e_j$. In particular, we may take $c_1=e_1\w e_2\w \ldots \w e_k$ and
$c_2=e_1\w e_2\w \ldots \w e_{k-1}\w e_{k+1}$. Notice that  
$\lambda_1(C)=\lambda_1\cdots\lambda_k$ and $\lambda_2(C)=\lambda_1\cdots\lambda_{k-1}\lambda_{k+1}$.
Let $(d^i_j)$ be the matrix for $D$ in the basis $\{c_i\}$. We define $\hat c_i$ and $(\hat d^i_j)$ completely
analogously.

Notice that $D(e_1\w e_2\w \ldots \w e_k)\w (e_{k+1}\w \ldots  \w e_{d_1}) = d^1_1\Omega$. So, if $d^1_1=0$, then 
$$B(\xi_\rho^k(\alpha_+))\oplus\xi_\rho^{n-k}(\alpha_-)  = \xi_\rho^k(\beta(\alpha_+)) \oplus \xi_\rho^{n-k}(\alpha_-) \neq \Real^{d_1}.
$$
which would contradict the hyperconvexity of $\xi_\rho$. Therefore, $d_1^1\ne 0$.
 
Furthermore, $D(e_1\w e_2\w \ldots \w e_k)\w(e_{k}\w e_{k+2}\w \ldots  \w e_{d_1}) = -d^2_1\Omega$. 
So, if $d^2_1 = 0$, then
$$\left\{L_1(\rho(\beta\alpha\beta^{-1})),\ldots,L_k(\rho(\beta\alpha\beta^{-1})),L_k(\rho(\alpha)),L_{k+2}(\rho(\alpha)),\ldots,
L_{d_1}(\rho(\alpha))\right\}$$
does not span $\mathbb R^{d_1}$, which contradicts Corollary \ref{productnonzero}.
Therefore, $d^2_1 \neq 0$

Similarly, $D(e_1\w e_2\w \ldots \w e_{k-1} \w e_{k+1})\w(e_{k+1}\w e_{k+2}\w \ldots  \w e_{d_1}) = d^1_2 \Omega$. 
So, if $d^1_2 = 0$, then
$$\left\{L_1(\rho(\beta\alpha\beta^{-1})),\ldots,L_{k-1}(\rho(\beta\alpha\beta^{-1})),L_{k+1}(\rho(\beta\alpha\beta^{-1})),L_{k+1}(\rho(\alpha)),\ldots,
L_{d_1}(\rho(\alpha))\right\}$$
does not span $\mathbb R^{d_1}$, which contradicts Corollary \ref{productnonzero}.
Thus $d^1_2\neq 0$.

Analogous arguments imply that $\hat d^1_1$, $\hat d_1^2$ and $\hat d_2^1$ are all non-zero. Moreover,
by our iterative assumption
$$|\lambda_1(C^nD)|=|\lambda_1(A^nB)\cdots\lambda_k(A^nB)|=
|\lambda_1(\hat A^n\hat B)\cdots\lambda_k(\hat A^n\hat B)|=| \lambda_1(\hat C^n\hat D)|$$
for all $n$. We may
again apply Lemma \ref{specradiusexpansion} to conclude that
$$\frac{\lambda_{k+1}}{\lambda_k} = \left|\frac{\lambda_2(C)}{\lambda_1(C)}\right|= \left| \frac{\lambda_2(\hat C)}{\lambda_1(\hat C)}\right|= 
\frac{\hat\lambda_{k+1}}{\hat\lambda_k}.$$
Since, by our inductive assumption, $\lambda_k=\hat\lambda_k$, we conclude that
$\lambda_{k+1}=\hat\lambda_{k+1}$.\
Therefore, after iteratively applying our argument, we conclude that 
$\lambda_i(\rho(\alpha))=\lambda_i(\sigma(\alpha))$ for all $1\leq i \leq d_1$. 
As in the proof of Lemma \ref{trace to length} it follows that $d_1=d_2$.
Therefore, $|\tr(\rho(\alpha))|=|\tr(\sigma(\alpha))|$.
\end{proof}


\section{Simple length rigidity}

We are now ready to establish our main results on simple length and simple trace rigidity.
We begin by studying  
configurations of curves in the form pictured in Figure \ref{4curves}.

\begin{figure}[htbp] 
   \centering
   \includegraphics[width=4.5in]{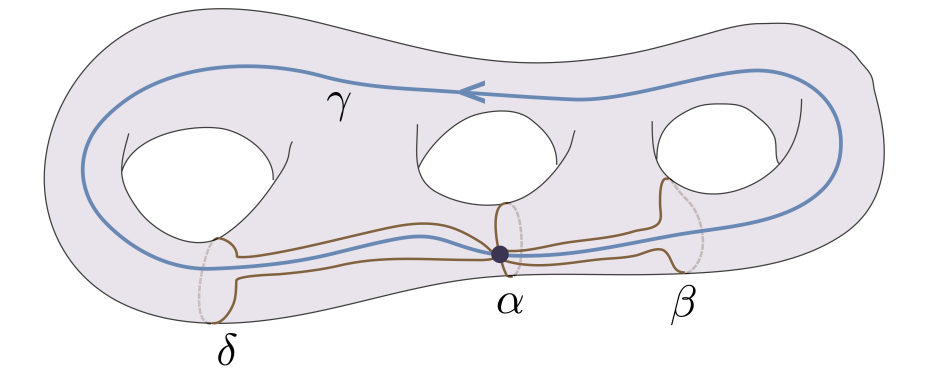} 
   \caption{Curves $\alpha, \beta, \gamma, \delta$}
   \label{4curves}
\end{figure}

\begin{theorem}
\label{conjugateontriples-trace}
Suppose that $F$ is an essential, connected subsurface of $S$, and that
$\alpha, \beta,\delta \in \pi_1(F)\subset S$ are 
represented by based simple loops in F which
intersect only at the basepoint,  and are freely homotopic to a collection of 
mutually disjoint and non-parallel, non-separating closed curves
in $F$ which do not bound a pair of pants in $F$.
If $\rho,\sigma\in\mathcal H_d(S)$ and  $|\tr(\rho(\eta))|=|\tr(\sigma(\eta))|$ whenever
$\eta\in\pi_1(S)$ is represented by a simple closed curve in $F$, 
then $\rho$ and $\sigma$
are conjugate, in $\pgln$, on the subgroup $<\alpha,\beta,\delta>$ of $\pi_1(S)$.
\end{theorem}

\begin{proof}
We first show that we can replace $\alpha$, $\beta$ and $\delta$ with based loops  in $F$,
configured as
in Figure \ref{4curves}, which generate the same subgroup of $\pi_1(S)$. We
then show that if $\alpha$, $\beta$, $\gamma$ and $\delta$ have the form in Figure \ref{4curves}, then
$\rho$ and $\sigma$ are conjugate on  $\braket{\alpha,\beta,\delta}$.

\begin{lemma} 
\label{good configuration}
Suppose that $F$ is an essential, connected subsurface of $S$, and that
$\alpha, \beta,\delta \in \pi_1(F)\subset S$ are 
represented by based simple loops in $F$ which
intersect only at the basepoint,  and 
are freely homotopic to a collection of  mutually disjoint and non-parallel, non-separating closed curves
in $F$ which do not bound a pair of pants in $F$. Then there exist based loops $\hat\alpha$, $\hat\beta$,
$\hat\gamma$ and $\hat\delta$ in $F$ which intersect only at the basepoint so  that $\hat\alpha$, $\hat\beta$ and $\hat\delta$
are freely homotopic to a collection of mutually disjoint and non-parallel, non-separating closed curves,
each has geometric intersection one with $\hat\gamma$
and
$$\braket{\hat\alpha,\hat\beta,\hat\delta}=\braket{\alpha,\beta,\delta}.$$
\end{lemma}

\begin{proof}
We first assume one of the curves, say $\beta$, has the property that the other two curves lie on opposite sides of
$\beta$, i.e. there exists a regular neighborhood $N$ of $\beta$, so that $\alpha$ intersects only one component of
$N-\beta$ and $\delta$ only intersects the other (see Figure \ref{config1}). 

\begin{figure}[htbp] 
   \centering
   \includegraphics[width=1.5in]{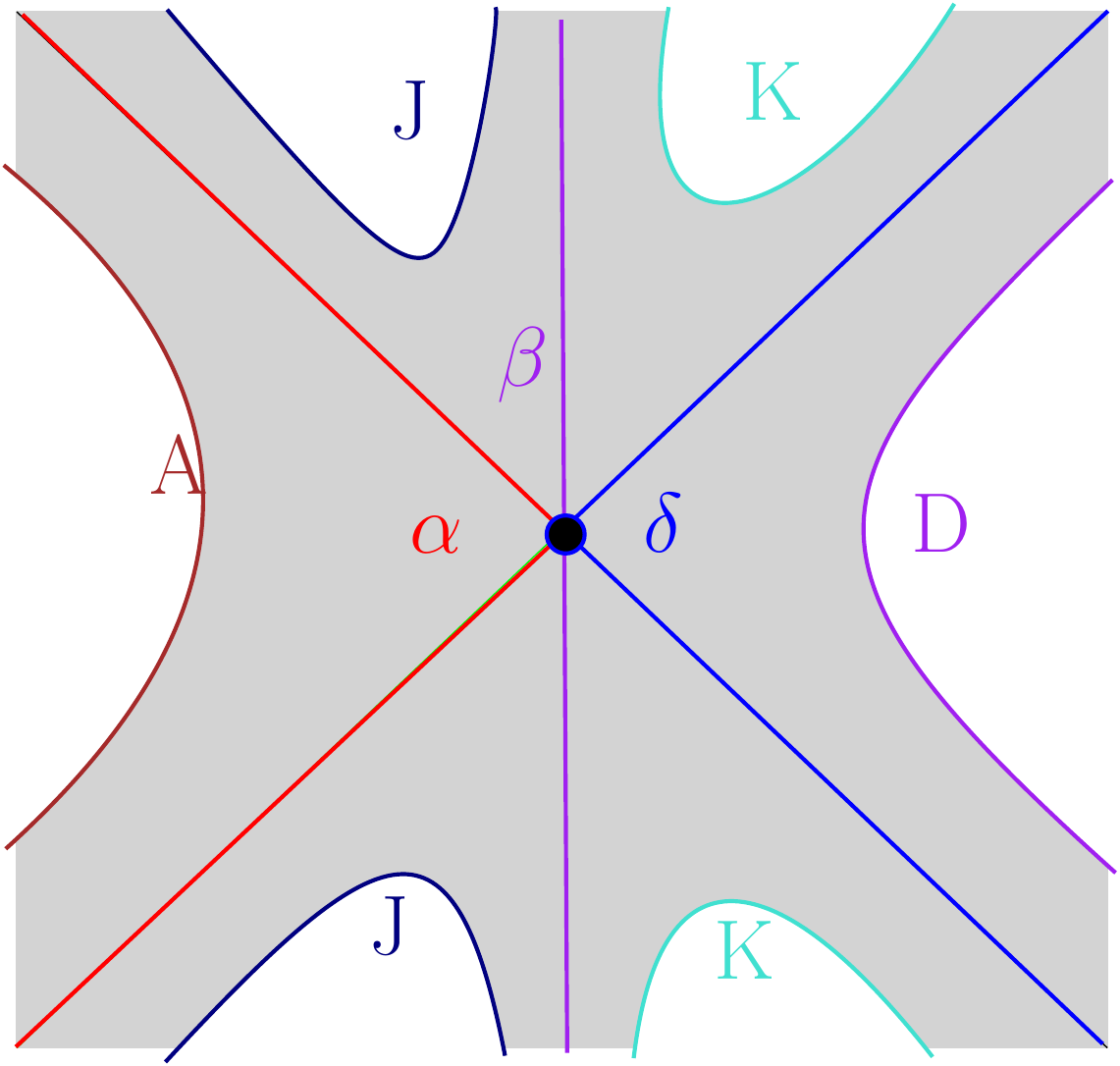} 
   \caption{A regular neighborhood of $\alpha\cup\beta\cup\delta$ when $\beta$ locally separates $\alpha$ and $\delta$}
   \label{config1}
\end{figure}

Let $F_1$ be a regular neighborhood of $T=\alpha\cup\beta\cup\delta.$ Then $F_1$ is a four-holed sphere and
each component of $F_1-T$ is an annulus.
We label the boundary components $A$, $D$, $J$ and $K$, where $A$ is parallel to $\alpha$, $D$ is parallel to $\delta$,
$J$ is parallel to the based loop $\beta\alpha^{\epsilon_1}$ and $K$ is parallel to the based loop $\beta\delta^{\epsilon_2}$
for some $\epsilon_1,\epsilon_2\in\{\pm 1\}$.

If $A$ and $D$ lie in the boundary of the same component of $F-F_1$,
then one may extend an arc in $F-F_1$ joining $A$ to $D$ to a closed curve $\hat\gamma$ which
intersects $T$ only at the basepoint and intersects each of $\alpha$, $\beta$ and $\delta$ with geometric
intersection one. In this case, we simply take $\hat\alpha=\alpha$, $\hat\beta=\beta$ and $\hat\delta=\delta$.
We assume from now on that $A$ and $D$ do not lie in the same boundary component of $F-F_1$.

Since $\alpha$ is non-separating, $A$ must lie in the boundary of a component $G$ of $F-F_1$ which
also has either $J$ or $K$ in its boundary. If the boundary of $G$ contains $J$ but not $K$, then
$\beta$ would separate $F$ which would contradict our assumptions, so the boundary of $G$ must contain $K$.
(Recall that by assumption, the boundary of $G$ cannot contain $D$.)

We may then extend an arc in $G$
joining $A$ to $K$ to a closed curve $\hat\gamma$ which intersects $T$ only at the basepoint and has geometric
intersection one with $\alpha$, $\beta$ and $K$. Moreover, we may choose a based loop $\hat\delta$ in
the (based) homotopy class of $\beta\delta^{\epsilon_2}$ which intersects $\alpha$, $\beta$ and $\hat\gamma$ only at the basepoint.
In this case, let $\hat\alpha=\alpha$ and $\beta=\hat\beta$. $A$, $\beta$ and $K$ are simple, disjoint non-separating
curves freely homotopic to $\hat\alpha$, $\hat\beta$ and $\hat\delta$. If $K$ is parallel to $A$,
then disjoint representative of $\alpha$, $\beta$ and $\delta$ would bound a pair of pants, which is disallowed.
Moreover, since $K$ is homotopic to $\beta\delta^{\epsilon_2}$ and $\beta$ and $\delta$ are non-parallel simple
closed curves, $K$ cannot be parallel to $\beta$ or $\delta$.
Since $A$ and $\beta$ are non-parallel, by assumption,
$A$, $\beta$ and $K$ are  mutually non-parallel as required.

We may now assume that if $\nu\in\{\alpha,\beta,\delta\}$, then there is a regular neighborhood of $\nu$, so that
the other two based loops only intersect one component of the regular neighborhood. Let $F_1$ be a regular
neighborhood of $T$. Again, $F_1$ is a four-holed sphere and each component of $F_1-T$ is an annulus.
We label the components of the boundary of $F_1$ by $A$, $B$, $D$ and $E$, where $A$ is parallel to $\alpha$,
$B$ is parallel to $\beta$, and $D$ is parallel to $\delta$ (see Figure \ref{config2}).
\begin{figure}[htbp] 
   \centering
   \includegraphics[width=1.5in]{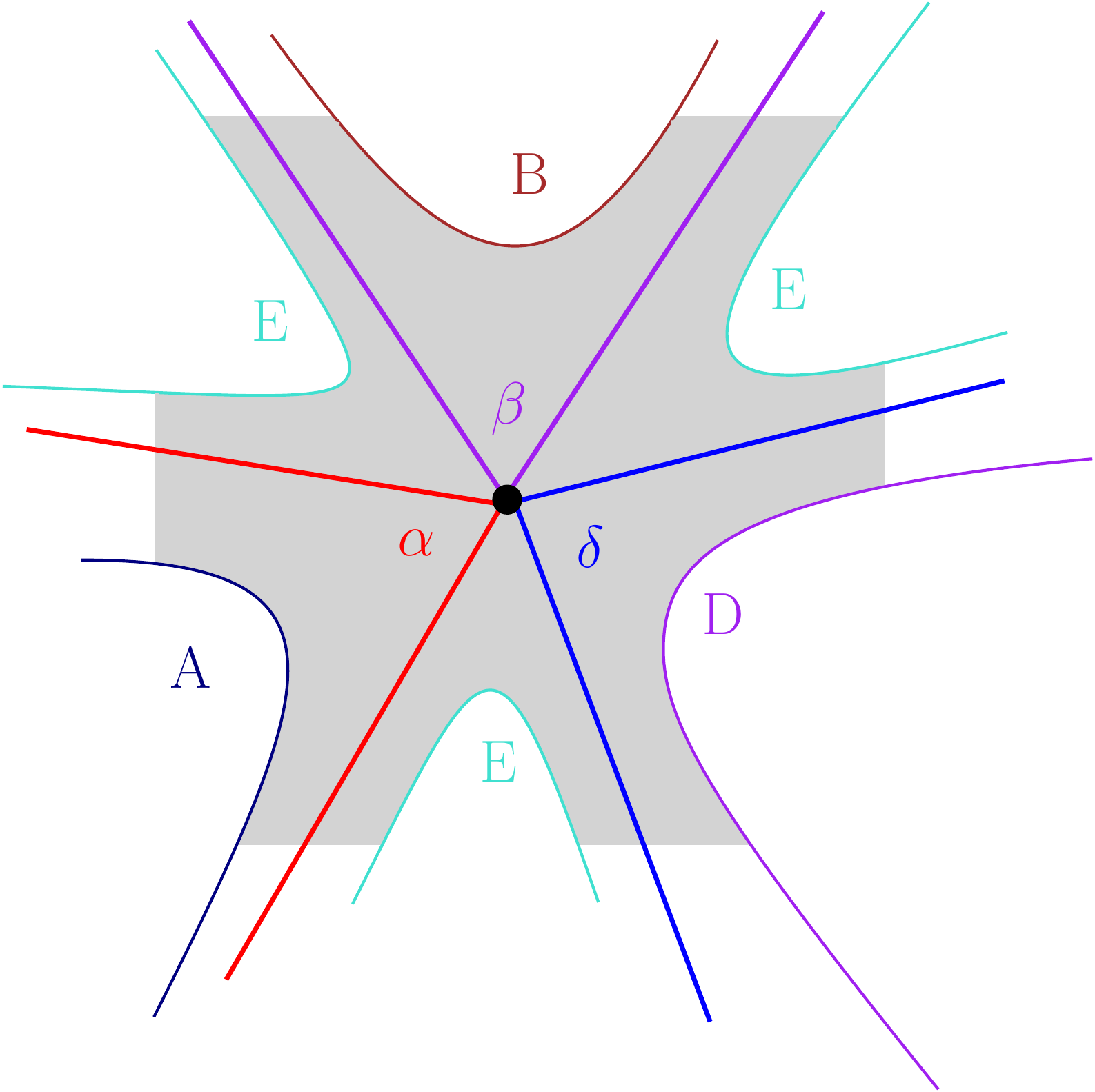} 
   \caption{A regular neighborhood of $\alpha\cup\beta\cup\delta$ when no curve locally separates}
   \label{config2}
\end{figure}
Since $\alpha$ is non-separating in $F$, there exists a component $G$ of $F-F_1$ whose boundary contains $A$
and at least one other component of the boundary of $F_1$. If the boundary of $G$ contains $B$,
then one may extend an arc in $G$ joining $A$ to $B$ to a curve $\hat\gamma$ which intersects $T$ only at the basepoint
and has geometric intersection one with $\alpha$ and $\beta$ and geometric intersection zero with $\delta$.
Let $\hat\delta$ be a simple based loop in $F_1$ in the (based) homotopy class of $\alpha\delta^\epsilon$ for some $\epsilon\in\{\pm1\}$
which intersects $\hat\gamma$ and $T$ only at the basepoint. Since $\hat\delta$ has algebraic intersection $\pm 1$ with
$\hat\gamma$, it must have geometric intersection one with $\hat\gamma$. Let $\hat\alpha=\alpha$ and $\hat\beta=\beta$,
then $\hat\alpha$, $\hat\beta$ and $\hat\delta$ are  freely homotopic to the collection $\{A,B,\hat\delta\}$ of mutually
disjoint, non-separating curves. Notice that $A$ and $B$ are non-parallel by our original assumption, while if
$\hat\delta$ is parallel to $A$ or $B$, then our original collection of curves would be freely homotopic to
the boundary of  a pair of pants, contradicting our original assumption. 
Therefore, $A$, $B$ and $\hat\delta$
are non-parallel as required.

If the boundary of $G$, contains $D$, then we may perform the same procedure reversing the roles of $\beta$
and $\delta$. Therefore, we may assume that the boundary of $G$ contains both $A$ and $E$, but not $B$ or $D$.
Since $\beta$ is non-separating and $B$ is not in the boundary of $G$, there must be another component $H$ of
$F-F_1$ which has both $B$ and $D$ in its boundary. We then simply repeat the procedure above to construct a curve
$\hat \gamma$ which intersects $T$ only at the basepoint which has geometric intersection one with $\beta$ and $\delta$
and geometric intersection zero with $\alpha$. We then let $\hat\alpha$ be a simple based loop in $F_1$ intersecting $\hat\gamma$
only at the basepoint, in the based homotopy class of $\beta\alpha^\epsilon$ for some $\epsilon\in\{\pm1\}$, which
has geometric intersection one with $\hat\gamma$. Letting $\hat\beta=\beta$ and $\hat\delta=\delta$, we may
complete the proof as in the previous paragraph.
\end{proof}

Notice that we may always re-order  the
curves produced by Lemma \ref{good configuration} so that $\hat\alpha^p\hat\beta^q\hat\gamma\hat\delta^r$ is represented
by a simple non-separating curve in $F$ for all $p,q,r\in\mathbb Z$.
Moreover, our assumptions imply that $\hat\alpha$, $\hat\beta$ and $\hat\delta$ have non-intersecting axes
and that $\hat\beta$ and $\hat\gamma\hat\beta\hat\gamma^{-1}$ have non-intersecting axes. 
Theorem \ref{conjugateontriples-trace} will then follow from the following result. 

\begin{proposition}
\label{pqr result}
Suppose that $\alpha, \beta,\gamma, \delta \in \pi_1(S)-\{1\}$, $\alpha$, $\beta$ and $\delta$ have non-intersecting axes and
that $\beta$ and $\gamma\beta\gamma^{-1}$ have non-intersecting axes.
If $\rho,\sigma\in\mathcal H_d(S)$ and
$|\tr(\rho(\alpha^p\beta^q\gamma\delta^r))|=|\tr(\sigma(\alpha^p\beta^q\gamma\delta^r))|$
for all $p,q,r\in\mathbb Z$, then
$\rho$ and $\sigma$
are conjugate, in $\pgln$, on the subgroup $<\alpha,\beta,\delta>$ of $\pi_1(S)$.
\end{proposition}

\begin{proof}
We may apply Lemma \ref{trace to length} to the pairs $(\alpha,\gamma)$, $(\beta,\gamma)$ and $(\delta,\gamma)$
to conclude that 
$\lambda_i(\rho(\eta))=\lambda_i(\sigma(\eta))$ for all $i$ and any $\eta\in \{\alpha,\beta,\delta\}$.
(Notice, for example, that for the pair $(\alpha,\gamma)$ our assumptions imply that 
$|\tr(\rho(\alpha^n\gamma))|=|\tr(\sigma(\alpha^n\gamma))|$ for all $n$, so the assumptions of Lemma
\ref{trace to length} are satisfied.)

Combining the expansions
$$\rho(\alpha)=\sum_{i=1}^d\lambda_i(\rho(\alpha)) \p_i(\rho(\alpha))\ \ {\rm and} \ \ \sigma(\alpha)=\sum_{i=1}^d\lambda_i(\sigma(\alpha)) \p_i(\sigma(\alpha))$$
with our assumption that $|\tr(\rho(\alpha^p\beta^q\gamma\delta^r))|=|\tr(\sigma(\alpha^p\beta^q\gamma\delta^r))|$
for all $p,q,r\in\mathbb Z$, we see that 
$$ \sum_{i=1}^d \lambda_i^{p}(\rho(\alpha)) \tr\left(\p_i(\rho(\alpha))\rho(\beta^q\gamma\delta^r)\right)
= \pm \sum_{i=1}^d \lambda_i^{p}(\sigma(\alpha)) \tr\left(\p_i(\sigma(\alpha))\sigma(\beta^q\gamma\delta^r)\right)$$
for all $p,q,r\in\N$.
Since $\rho(\alpha)$ and $\sigma(\alpha)$ are purely loxodromic and $\lambda_i(\rho(\alpha))=\lambda_i(\sigma(\alpha))$ for
all $i$,
we may fix $q$ and $r$, let $p$ tend to $+\infty$ and consider terms of the same order to conclude that
\begin{equation}
\label{keyexpansion}
\tr\left(\p_i(\rho(\alpha))\rho(\beta^q\gamma\delta^r)\right)
=\pm \tr\left(\p_i(\sigma(\alpha))\sigma(\beta^q\gamma\delta^r)\right)
\end{equation}
for all $i\in\{1,\ldots,d\}$ and all $q,r\in\N$.
Similarly, we expand Equation (\ref{keyexpansion}) to see that, for all $i$,
$$ \sum_{j=1}^d \lambda_j^{q}(\rho(\beta)) \tr\left(\p_i(\rho(\alpha))\p_j(\rho(\beta))\rho(\gamma\delta^r)\right)
= \pm\sum_{j=1}^d \lambda_j^{q}(\sigma(\beta)) \tr\left(\p_i(\sigma(\alpha))\p_j(\sigma(\beta))\sigma(\gamma\delta^r)\right)$$
and consider terms of the same order  as $q\to +\infty$ to conclude that
$$
\tr\left(\p_i(\rho(\alpha))\p_j(\rho(\beta))\rho(\gamma\delta^r)\right)
=\pm \tr\left(\p_i(\sigma(\alpha))\p_i(\sigma(\beta))\sigma(\gamma\delta^r)\right)$$
for all $i,j\in\{1,\ldots,d\}$ and $r\in\N$.
Expanding this last equation and letting $r$ tend to $+\infty$, we finally conclude that
$$\tr(\p_{i}(\rho(\alpha))\p_{j}(\rho(\beta))\rho(\gamma)\p_{k}(\rho(\delta)))
=\pm\tr(\p_{i}(\sigma(\alpha))\p_{j}(\sigma(\beta))\sigma(\gamma)\p_{k}(\sigma(\delta)))$$
for all $i,j,k\in\{1,\ldots,d\}$, i.e.
\begin{equation}
\label{key expansion2}
{\bf T}_{i,j,0,k}(\alpha,\beta,\gamma,\delta)(\rho)=\pm{\bf T}_{i,j,0,k}(\alpha,\beta,\gamma,\delta)(\sigma)
\end{equation}
for all $i,j,k\in\{1,\ldots,d\}.$

We similarly expand the equation 
$$\tr(\rho(\beta^q\gamma\delta^r))=\pm \tr(\sigma(\beta^q\gamma\delta^r))$$
to see that
\begin{equation}
\label{key expansion3}
{\bf T}_{j,0,k}(\beta,\gamma,\delta)(\rho)=\pm {\bf T}_{j,0,k}(\beta,\gamma,\delta)(\sigma)
\end{equation}
for all $j$ and $k$.

Recall, from  part (4) of Proposition \ref{tracenonzero}, that
$${\bf T}_{i,j,0,k}(\alpha,\beta,\gamma,\delta)(\rho) = 
{\bf T}_{j,0,k}(\beta,\gamma,\delta)(\rho)\left(\frac{{\bf T}_{i,j,k}(\alpha,\beta,\delta)(\rho)}{{\bf T}_{j,k}(\beta,\delta)(\rho)}\right)\ne 0$$
for all $\rho\in\Hn$ and $i,j,k\in\{1,\ldots,d\}$, so we may conclude from Equations (\ref{key expansion2}) and 
(\ref{key expansion3})
that
$$\frac{{\bf T}_{i,j,k}(\alpha,\beta,\delta)(\rho)}{{\bf T}_{j,k}(\beta,\delta)(\rho)}=\pm
\frac{{\bf T}_{i,j,k}(\alpha,\beta,\delta)(\sigma)}{{\bf T}_{j,k}(\beta,\delta)(\sigma)}$$
for all $i,j,k\in\{1,\ldots,d\}$.

We may join $\rho$ to $\sigma$ by a path $\{\rho_t\}$ of Hitchin representations. So, since
$\frac{{\bf T}_{i,j,k}(\alpha,\beta,\delta)(\rho_t)}{{\bf T}_{j,k}(\beta,\delta)(\rho_t)}$ 
is non-zero for all $t$, again by Proposition \ref{tracenonzero}, and varies continuously, it follows that 
$$\frac{{\bf T}_{i,j,k}(\alpha,\beta,\delta)(\rho)}{{\bf T}_{j,k}(\beta,\delta)(\rho)}=
\frac{{\bf T}_{i,j,k}(\alpha,\beta,\delta)(\sigma)}{{\bf T}_{j,k}(\beta,\delta)(\sigma)}$$
for all $i,j,k\in\{1,\ldots,d\}$.
Therefore, since we have already seen that $\lambda_i(\rho(\eta))=\lambda_i(\sigma(\eta))$ for all $i$
if $\eta\in\{\alpha,\beta,\gamma\}$,
Theorem \ref{conjugateontriples-general} implies that
$\rho$ and $\sigma$
are conjugate, in $\pgln$, on the subgroup $<\alpha,\beta,\delta>$ of $\pi_1(S)$.
\end{proof}
\end{proof}

We are now ready to establish that  the restriction of the marked trace spectrum  to
the simple non-separating curves determines a Hitchin representation.

\begin{theorem}
Let $S$ be a closed orientable surface of genus $g\ge 3$.
If \hbox{$\rho \in {\mathcal H}_{d_1}(S)$}, $\sigma\in\mathcal H_{d_2}(S)$ and $|\tr(\rho(\alpha))| = |\tr(\sigma(\alpha))|$ 
whenever \hbox{$\alpha \in \pi_1(S)$}  is represented by a simple non-separating curve,
then $d_1=d_2$ and $\rho = \sigma$.
\label{fulltracerigidity}
\end{theorem}

\begin{proof}
Notice that  Theorem \ref{equivalence} immediately implies that $d_1=d_2$, so we may assume that
$d=d_1=d_2$.
Consider the standard generating set 
$$\mathcal S=\{ \alpha_1, \beta_1,\ldots,\alpha_g,\beta_g\}$$ 
for $\pi_1(S)$ so that $\prod_{i=1}^g [\alpha_i,\beta_i]=1$, each generator is represented by a based loop,
and any two such based loops intersect only at the basepoint.
 
Notice that the generators are freely homotopic to simple, non-separating closed curves so that the representative
of $\alpha_i$ is disjoint from the representative of every other generator except $\beta_i$ and that
the representative of $\beta_i$ is disjoint from the representative of every other generator except $\alpha_i$.
Moreover,   no three of the representatives which are disjoint bound
a pair of pants.  Therefore,
Theorem \ref{conjugateontriples-trace} implies that  we may assume 
that  $\rho$ and  $\sigma$ agree on $<\alpha_1,\alpha_2,\alpha_3>$.

If $\eta\in \mathcal S-\{\alpha_1,\alpha_2,\beta_1,\beta_2\}$, then Theorem \ref{conjugateontriples-trace}
implies  that there exists \hbox{$C\in\pgln$} so that $\rho$ and 
$C\sigma C^{-1}$ agree on $<\alpha_1,\alpha_2,\eta>$. Since $\rho$ and $\sigma$ agree on $\alpha_1$ and
$\alpha_2$, the following lemma, which we memorialize 
for repeated use later in the paper,  assures that $C=I$, so $\rho(\eta)=\sigma(\eta)$.

\begin{lemma}
\label{conjugacytrivial}
Suppose that $S$ is a closed surface of genus at least two,
$\rho:\pi_1(S)\to\psln$ and $\sigma:\pi_1(S)\to\psln$ are Hitchin representations, 
and there exists a subgroup $H$ of $\pi_1(S)$ and $C\in\psln$ so that $\rho|_H=C\sigma|_HC^{-1}$.
If there exists $\nu_1,\nu_2\in H$ with non-intersecting axes, so that $\rho(\nu_1)=\sigma(\nu_1)$
and $\rho(\nu_2)=\sigma(\nu_2)$, then $C=I$, so $\rho|_H=\sigma|_H$.
\end{lemma}

\begin{proof}
Since $\rho$ and $\sigma$ agree on $\nu_1$ and $\nu_2$,
$C$ must commute with $\rho(\nu_1)$ and $\rho(\nu_2)$. Thus $C$ is diagonalizable over $\Real$ with respect to  both 
$\{e_i(\rho(\nu_1))\}$ and $\{e_i(\rho(\nu_2)\}$.
If $C \ne I$, then $\mathbb R^d$ admits a non-trivial decomposition into eigenspaces of $C$ with distinct eigenvalues.
Any such eigenspace $W$ is spanned by a sub-collection of $\{e_i(\rho(\nu_1))\}$ and by a sub-collection of
$\{e_j(\rho(\nu_2))\}$.
In particular, some $e_i(\rho(\nu_1))$ is in the subspace spanned by a subcollection of $\{e_j(\rho(\nu_2))\}$. 
Since $\nu_1$ and $\nu_2$ have non-intersecting axes, this contradicts Corollary \ref{productnonzero}. 
Therefore, $C= I$.
\end{proof}

In order to prove that $\rho(\beta_1)=\sigma(\beta_1)$, we similarly apply Theorem \ref{conjugateontriples-trace}
and Lemma \ref{conjugacytrivial}
to the elements $\alpha_2$, $\alpha_3$ and $\beta_1$, while
to prove that $\rho(\beta_2)=\sigma(\beta_2)$ we consider the elements $\alpha_1$, $\alpha_3$ and $\beta_2$.
Since we have established that $\rho$ and $\sigma$ agree on every element in the generating set $\mathcal S$,
we conclude that $\rho=\sigma$.
\end{proof}

Marked simple length rigidity, Theorem \ref{lengthrigidity}, is an immediate consequence of 
Theorems \ref{fulltracerigidity} and \ref{equivalence}. 

We  may further use the Noetherian property of polynomial rings to prove the final statement in Theorem \ref{tracerigidity},
which asserts that Hitchin representations of the same dimension
are determined by the traces of a finite set of simple non-separating curves.

\begin{proof}[Proof of Theorem \ref{tracerigidity}]
We consider the affine algebraic variety 
$$V(S) = \hom(\pi_1(S),\sln)\times \hom(\pi_1(S),\sln).$$
Let $\{\gamma_i\}_{i=1}^\infty \subset \pi_1(S)$ be an ordering of the collection of (conjugacy classes of)
elements  of $\pi_1(S)$ which are represented by simple, non-separating curves, and define, for each $n$,
$$V_n(S) = \left\{ (\rho,\sigma) \in V(S)\ | \ \tr(\rho(\gamma_i))=  \tr(\sigma(\gamma_i))\ {\rm if}\ \ i \leq n\right\}$$
and let
$$V_\infty =\bigcap_{n=1}^\infty V_n .$$
Then each $V_n(S)$ is a subvariety of $V(S)$ and by the Noetherian property of polynomial rings, 
there exists  $N$ so that $V_N = V_\infty$.
We define $\mathcal L_d(S) = \{\gamma_i\}_{i=1}^N$. 

There exists a component $\xHn$  of  $\hom(\pi_1(S),\sln)$ consisting of lifts of Hitchin representations
so that $\Hn$ is identified with the quotient of $\xHn$ by $\sln$, see Hitchin \cite{hitchin}.
Since traces of elements in images of (lifts of) Hitchin representations are non-zero, for all $\gamma\in\pi_1(S)$,
$\tr(\nu(\gamma))$ is either positive  for all $\nu\in\xHn$ or  negative for all $\nu\in\xHn$, for all $\gamma\in\pi_1(S)$.
Therefore, if the marked trace spectra of $\rho,\sigma \in  \mathcal H_d(S)$ 
agree on $\mathcal L_d(S)$, they admit lifts $\tilde\rho$ and $\tilde\sigma$ in $\xHn$ so that $(\tilde\rho,\tilde\sigma)\in V_N$. 
Since $V_N=V_\infty$, the marked trace spectra of $\rho$ and $\sigma$ agree on all simple, non-separating curves.
Therefore, by Theorem \ref{fulltracerigidity}, $\rho=\sigma\in\Hn$.
\end{proof}

\noindent{\bf Remark:} The set $\mc L_d(S)$ contains at least $\dim(\Hn)=-\chi(S)(d^2-1)$ curves,
but our methods do not provide any upper bound on the size of $\mc L_d(S)$.

\section{Isometries of intersection}\label{sec:isom}

In this section, we investigate isometries of the intersection function which is used to construct
the pressure metric on the Hitchin component. Our main tool will be Bonahon's theory of geodesic
currents and his reinterpretation of Thurston's compactification of Teichm\"uller space in this
language, see Bonahon \cite{bonahon}.

\subsection{Intersection and the pressure metric}
\label{intersection}
Given $\rho\in\Hn$, let
$$R_T(\rho) = \left\{ [\gamma]\ \in [\pi_1(S)]\ |\ L_\gamma(\rho) \leq T\right\}$$
be the set of conjugacy classes of elements of 
$\pi_1(S)$ whose images have length at most $T$. One may then define the {\em entropy} 
$$h(\rho) = \lim_{T\rightarrow \infty} \frac{\log(\#R_T(\rho))}{T}.$$
Given $\rho, \sigma\in\Hn$,  their {\em intersection} is given by
$$\II(\rho,\sigma) = \lim_{T\rightarrow \infty} \frac{1}{\#R_T(\rho)} \sum_{[\gamma] \in R_T(\rho)} \frac{L_\gamma(\sigma)}{L_\gamma(\rho)}.$$
and their {\em renormalized intersection} is given by
$$\JJ(\rho,\sigma) = \frac{h(\sigma)}{h(\rho)}I(\rho,\sigma).$$
One may show that all the quantities above give rise to analytic functions.

\begin{theorem}{\rm (Bridgeman--Canary--Labourie--Sambarino \cite[Thm. 1.3]{BCLS})}
If $S$ is a closed surface of genus greater than 1,
the entropy $h$, 
the intersection $\II $,  and renormalized intersection  $\JJ$ are analytic functions on $\Hn$, 
 $\Hn\times\Hn$ and  $\Hn\times\Hn$ respectively.
\end{theorem}

Let $\JJ_\rho:\Hn\to \mathbb R$ be defined by $\JJ_\rho(\sigma) = \JJ(\rho,\sigma)$. The analytic function $\JJ_\rho$
has a minimum at $\rho$ (see \cite[Thm. 1.1]{BCLS})
and hence its Hessian gives rise to an non-negative quadratic form on $T_\rho(\Hn)$,
called the {\em pressure metric}.  Bridgeman, Canary, Labourie
and Sambarino proved that the resulting quadratic form is positive definite.
A result of Wolpert \cite{wolpert} implies that the restriction of the pressure
metric to the Fuchsian locus is a multiple of the classical Weil-Petersson metric.
(See \cite{BLS-survey} for a survey of this theory.)

\begin{theorem}{\rm (Bridgeman--Canary--Labourie--Sambarino \cite[Cor. 1.6]{BCLS})}
If $S$ is a closed surface of genus greater than 1,
the pressure metric is a mapping class group invariant, analytic, Riemannian metric on $\Hn$
whose restriction to the Fuchsian locus is a multiple of
the Weil-Petersson metric.
\end{theorem}

Recall that a diffeomorphism $f:\Hn\to\Hn$ is said to be an {\em isometry of intersection} if
$\II(f(\rho),f(\sigma)) = \II(\rho,\sigma)$ for all $\rho,\sigma \in {\mathcal H}_d(S)$. Let ${\rm Isom}_\II(\Hn)$
denote the group of isometries of $\II$. Notice that, by construction, the extended mapping class
group ${\rm Mod}(S)$ is a subgroup of ${\rm Isom}_\II(\Hn)$. (The extended mapping class group ${\rm Mod}(S)$
can be identified with the group ${\rm Out}(\pi_1(S))$ of outer automorphisms of $\pi_1(S)$ and acts
naturally on $\Hn$ by pre-composition.)

The entire discussion of intersection, renormalized intersection and the pressure metric restricts
to $\ms H(S,\ms G)$ when $\ms{G}$ is $\ms{PSp}(2d,\mathbb R),$ \hbox{$\ms{PSO}(d,d+1)$},
or $\ms{G}_{2,0}$.

\subsection{Basic properties}
 
We first show that isometries of intersection preserve entropy and hence preserve 
renormalized intersection, so are isometries of the pressure metric.
 
\begin{proposition}
\label{isometriesareisometries}
If $S$ is a closed orientable surface of genus greater than 1,
$\ms{G}$ is $\psln$, $\ms{PSp}(2d,\mathbb R),$ $\ms{PSO}(d,d+1)$,
or $\ms{G}_{2,0}$
and \hbox{$f:\mathcal H(S,\ms G)\to\mathcal H(S,\ms G)$} is an isometry of 
intersection $\II$, then \hbox{$h(\rho)=h(f(\rho))$} for all \hbox{$\rho\in \mc H(S,\ms G)$}. Therefore, 
\hbox{$\JJ(f(\rho),f(\sigma))=\JJ(\rho,\sigma)$} for
all $\rho,\sigma\in\mc H(S,\ms G)$, and $f$ is an isometry of $\mc H(S,\ms G)$ with respect to the pressure metric.
\end{proposition}

\begin{proof}
Suppose that $\rho \in {\mathcal H}(S,\ms G)$, $v \in \ms T_\rho({\mathcal H}(S,\ms G))$ and
$v=\frac{d}{dt}\rho_t =\dt{\rho}_0$ for a smooth path $\{\rho_t\}_{t\in (-1,1)}$ in $\Hn$.
Then,
$$\II_\rho(\rho_t) = \II(\rho,\rho_t) = \II(f(\rho),f(\rho(t)) = \II_{f(\rho)}(f(\rho_t)),$$
so
$${\rm D}\II_\rho(v) = {\rm D}\II_{f(\rho)}({\rm D}f_\rho(v)).$$
Since $\JJ_\rho$ has a minimum at $\rho$, ${\rm D}\JJ_\rho(v)=0$, so
$${\rm D}\JJ_\rho(v)=\frac{{\rm D}h_\rho(v)}{h(\rho)}\II_\rho(\rho)+\frac{h(\rho)}{h(\rho)}{\rm D}\II_\rho(v)=
\frac{{\rm D}h_\rho(v)}{h(\rho)}+{\rm D}\II_\rho(v)=0$$
which implies that
$${\rm D}\II_\rho(v)=-\frac{{\rm D}h_\rho(v)}{h(\rho)} = -{\rm D}(\log{h})(v).$$
Thus, for all $v \in T_\rho({\mathcal H}(S,\ms G))$
$${\rm D}(\log{h})(v) ={\rm D} (\log(h\circ f))(v),$$
so
$(h\circ f)/h$ is constant, since ${\mathcal H}(S,\ms G)$ is a connected manifold. 
However, since $h$ is a bounded positive function, it must
be the case that $h\circ f=h$.

It follows, by the definition of renormalized intersection, that $f$ preserves renormalized intersection.
Since the pressure metric is obtained by considering the Hessian of  renormalized intersection,
$f$ is also an isometry of $\mc H(S,\ms G)$ with respect to the pressure metric.
\end{proof}

Potrie and Sambarino \cite{potrie-sambarino} proved that the entropy function achieves its maximum exactly on
the Fuchsian locus, so we have the following immediate corollary.

\begin{corollary}
\label{fuchsianpreserved}
If $S$ is a closed orientable surface of genus greater than 1,
$\ms{G}$ is $\psln$, $\ms{PSp}(2d,\mathbb R),$ $\ms{PSO}(d,d+1)$,
or $\ms{G}_{2,0}$ and $f:\mathcal H(S,\ms G)\to\mathcal H(S,\ms G)$ is an isometry of 
intersection $\II$, then $f$ preserves the Fuchsian locus.
\end{corollary}

\subsection{Geodesic currents}
We identify $S$ with a fixed hyperbolic surface $\hh/\Gamma$, which in turn identifies
$\pi_1(S)$ with $\Gamma$ and $\partial_\infty\pi_1(S)$ with $\partial_\infty \hh$. One can identify
the space  $G(\hh)$ of unoriented geodesics in $\Ht$ with \hbox{$(\partial_\infty \hh\times\partial_\infty \hh-\Delta)/\mathbb Z_2$},
where $\Delta$ is the diagonal in $\partial_\infty \hh\times\partial_\infty \hh$ and $\mathbb Z_2$ acts by interchanging coordinates.
A {\em geodesic current} on $S$ is  a $\Gamma$-invariant Borel measure on $G(\hh)$ and ${\mathcal C}(S)$ is the space
of geodesic currents on $S$, endowed with the weak$^*$ topology. 

If $\alpha$ is a closed geodesic on $S$, one obtains a geodesic current $\delta_\alpha$
by taking the sum of the Dirac measures on the pre-images of $\alpha$. 
The set of currents
which are scalar multiples of closed geodesics is dense in $\mathcal C(S)$, see Bonahon \cite[Proposition 2]{bonahon}.
If $\rho\in \mathcal T(S)=\mathcal H_2(S)$ has associated limit map $\xi_\rho:\partial\pi_1(S)\to \partial \mathbb H_2$,
one defines the {\em Liouville measure} of $\rho$ by
$$m_\rho ([a,b]\times[c,d]) = \left|\log\frac{\big(\xi_\rho(a)-\xi_\rho(c)\big)\big(\xi_\rho(b)-\xi_\rho(d)\big)}
{\big(\xi_\rho(a)-\xi_\rho(d)\big)\big(\xi_\rho(b)-\xi_\rho(c)\big)}\right|.$$

\begin{theorem}{\rm (Bonahon \cite[Propositions 3, 14, 15]{bonahon})}
Let $S$ be a closed oriented surface of genus $g \geq 2$ and $\rho\in\mathcal T(S)=\mathcal H_2(S)$.
Then there exist continuous functions $\ell_\rho:{\mathcal C}(S) \rightarrow \Real$ and 
$i:{\mathcal C}(S)\times {\mathcal C}(S) \rightarrow \Real$ which are linear on rays
such that if $\alpha$ and $\beta$ are closed geodesics, then
$$i(m_\rho, \delta_\alpha) = \ell_\rho(\alpha), \quad i(m_\rho,m_\rho) = \pi^2|\chi(S)|,$$
and $i(\alpha,\beta)$ is the geometric intersection between $\alpha$ and $\beta$.
\label{bonahon-currents}\end{theorem}

Moreover, Bonahon defines an embedding 
$$Q: {\mathcal T}(S) \rightarrow {\mathcal PC}(S)$$
of Teichm\"uller space 
into the space of projective classes of geodesic currents given by $Q(\rho) = [m_\rho]$. Bonahon shows
that the closure of $Q(\mathcal T(S))$ is homeomorphic to a closed ball of dimension $6g-6$, and the boundary
of $Q(\mathcal T(S))$ is the space $\mathcal{PML}(S)$
of projective classes of measured laminations. (Recall that a {\em measured lamination}
may be defined to be a geodesic current of self-intersection  0.) In particular, the geodesic current
associated to  any simple closed curve lies in the boundary of $Q(\mathcal T(S))$. Moreover, Bonahon \cite[Theorem 18]{bonahon}
shows that this compactification of Teichm\"uller space agrees with Thurston's compactification.

\subsection{Length functions for Hitchin representations}

If $\rho\in \Hn$, then there is a H\"older function 
\hbox{$f_\rho:T^1S \rightarrow \Real_+$} such that if $\alpha$ is a closed oriented geodesic on
$S=\hh/\Gamma$, then
$$\int_{\alpha} f_\rho\  dt =  L_\alpha(\rho)$$
where $dt$ is the Lebesgue measure along $\alpha\subset T^1(S)$, see \cite[Prop. 4.1]{BCLS}
or Sambarino \cite[Sec. 5]{sambarino-quantitative}.
Given $\mu\in\mathcal C(S)$, one may define a $\Gamma$-invariant measure $\tilde\mu$ on $\ms T^1\hh$
which has the local form $\mu\times dt$ where 
$dt$ is Lebesgue measure along the flow lines of $\ms T^1\hh$ (which are oriented geodesics in $\hh$),
so $\tilde \mu$ descends to a measure $\hat\mu$ on $T^1(S)$.
One may then define a length function 
$\ell_\rho:{\mathcal C}(S) \rightarrow \Real$ by letting 
$$\ell_\rho(\mu)=\int_{T^1(S)}f_\rho \ d\hat\mu.$$
Notice that if $\alpha$ is a simple closed geodesic on $S$, then
$$\ell_\rho(\delta_\alpha)=L^H_\alpha(\rho)=L_\alpha(\rho)+L_{\alpha^{-1}}(\rho)$$
since $\hat\delta_\alpha$ is Dirac measure support on the closed orbits of geodesics
associated to $\alpha$ and $\alpha^{-1}$.
Moreover, by the definition of the weak$^*$ topology,
$\ell_\rho$ is  clearly continuous, since $T^1S$ is compact.

Recall that (see Bowen \cite{bowen-equilibrium} or Margulis \cite{margulis-thesis}) if $\sigma\in\mathcal T(S)=\mathcal H_2(S)$ then the Liouville current satisfies
$$\frac{m_\sigma}{\ell_\sigma(m_\sigma)} = \lim_{T\rightarrow \infty} \frac{1}{\#R_T(\sigma)} \sum_{R_T(\sigma)} \frac{\delta_\alpha}{\ell_\sigma(\delta_\alpha)}
=  \lim_{T\rightarrow \infty} \frac{1}{\#R_T(\sigma)} \sum_{R_T(\sigma)} \frac{\delta_\alpha}{2L_\alpha(\sigma)}
.$$

Since $\tau_d$ multiplies the logarithm of the spectral radius by  $d-1$, 
 if $\rho\in\Hn$, then \begin{eqnarray*}
\frac{\ell_\rho(m_\sigma)}{\ell_\sigma(m_\sigma)} & = &
\lim_{T\rightarrow \infty} \frac{1}{\#R_T(\sigma)} \sum_{R_T(\sigma)} \frac{L^H_\alpha(\rho)}{ 2L_\alpha(\sigma)}\\
& = &
\left(d-1\right)\ \lim_{T\rightarrow \infty} \frac{1}{\#R_{(d-1)T}(\tau_d\circ\sigma)} \sum_{R_{(d-1)T}(\tau_d\circ\sigma)} \frac{L_\alpha(\rho)}{L_\alpha(\tau_d\circ\sigma)}\\ & = & \left(d-1\right)\ \II(\tau_d\circ\sigma,\rho).\\
\end{eqnarray*}
Here we use the fact that, since $\sigma\in\mathcal T(S)$,
$L_\alpha(\sigma)=L_{\alpha^{-1}}(\sigma)$, so
$$\frac{L^H_\alpha(\rho)}{ 2L_\alpha(\sigma)}+\frac{L^H_{\alpha^{-1}}(\rho)}{ 2L_{\alpha^{-1}}(\sigma)}=
\frac{L_\alpha(\rho)}{ L_\alpha(\sigma)}+\frac{L_{\alpha^{-1}}(\rho)}{ L_{\alpha^{-1}}(\sigma)}$$
for all $\alpha\in\pi_1(S)$.

\subsection{Isometries of intersection and the simple Hilbert length spectrum}

We next observe that any isometry of  intersection preserves the simple marked Hilbert
length spectrum.

\begin{proposition}
\label{isometries and Hilbert length}
If $S$ is a closed surface of genus $g\ge 2$, $\ms{G}=\psln$, $\ms{PSp}(2d,\mathbb R),$ $\ms{PSO}(d,d+1)$,
or $\ms{G}_{2,0}$
and $f:\mathcal H(S,\ms{G})\to\mathcal H(S,\ms{G})$ is an isometry of  intersection,
then there exists an element $\phi$ of
the extended mapping class group so that if $\rho\in\mc H(S,\ms G)$, then $\rho$ and $f\circ\phi(\rho)$ have
the same simple marked Hilbert length spectrum.
\end{proposition}

\begin{proof}
Recall, from Corollary \ref{fuchsianpreserved}, that $f$ preserves the Fuchsian locus.
Since any isometry of $\mathcal T(S)$ with the Weil-Petersson metric agrees with an element of the extended 
mapping class group, by a result of Masur-Wolf \cite{masur-wolf}, and  the restriction of the pressure
metric to the Fuchsian locus is a multiple of the Weil-Petersson metric,  
the restriction of $f$ to the Fuchsian locus agrees with the action of an element $\phi$ of the extended
mapping class group. We can thus consider $\hat f=f\circ\phi^{-1}$, which is an isometry of the intersection function
that fixes the Fuchsian locus.

If $\alpha\in\pi_1(S)$ is represented by a simple curve, we may choose a sequence $\{\sigma_n\}$ in $\mathcal T(S)$
such that $\{Q(\sigma_n)\}$ converges to $[\delta_\alpha]\in\mathcal{PC}(S)$,
so there exists a sequence $\{c_n\}$ of real numbers so that $\lim c_n=+\infty$ and
$$\lim \frac{m_{\sigma_n}}{c_n}=\delta_\alpha.$$
Therefore, if $\rho\in \mathcal H(S,\ms{G})\subset\Hn$, then
$$L^H_\alpha(\rho)=\ell_\rho(\delta_\alpha)=\lim\ell_\rho\left(\frac{m_{\sigma_n}}{c_n}\right)=
\lim \left(\frac{(d-1)\ell_{\sigma_n}(m_{\sigma_n})}{c_n}\II(\tau_d\circ\sigma_n,\rho)\right).$$
By Theorem \ref{bonahon-currents}, as $\sigma_n \in \mathcal T(S)$, then $\ell_{\sigma_n}(m_{\sigma_n}) = i(m_{\sigma_n}, m_{\sigma_n}) = \pi^2|\chi(S)|$.
If $\rho\in\mathcal H(S,\ms{G})$ and $\alpha\in\pi_1(S)$, then
since $\II(\tau_d\circ\sigma_n,\rho)=\II(\tau_d\circ\sigma_n,\hat f(\rho))$ for all $n$,
$L^H_\alpha(\rho)=L^H_\alpha(\hat f(\rho))$. Therefore, $\rho $ and $\hat f(\rho)$ have the same simple marked
Hilbert length spectrum.
\end{proof}

Recall that if $\rho$ lies in $\mathcal H(S,\ms{G})$ and $\ms{G}$ is
$\ms{PSp}(2d,\mathbb R)$, $\ms{PSO}(d,d+1)$ or $\ms{G}_{2,0}$, then
$L^H_\alpha(\rho)=2L_\alpha(\rho)$ for all $\alpha\in\pi_1(S)$.  
Therefore, we may combine Theorem
\ref{lengthrigidity} and Proposition \ref{isometries and Hilbert length} to obtain:

\begin{corollary}
\label{self dual isometries}
If $S$ is a closed surface of genus $g\ge 3$, then any isometry of 
the intersection $\II$ on $\mathcal H(S,\ms{PSp}(2d,\mathbb R))$,
$\mathcal H(S,\ms{PSO}(d,d+1))$, or $\mathcal H(S,\ms{G}_{2,0})$ agrees
with an element of the extended mapping class group.
\end{corollary}

Notice that Corollary \ref{self dual isometries} is a generalization of Theorem \ref{selfdualisometries}
which was stated in the introduction.

 
\section{Hilbert Length Rigidity}
\label{SL3}

Proposition \ref{isometries and Hilbert length} suggests the following potential 
generalization of our main simple length rigidity result.

\medskip\noindent
{\bf Conjecture:} {\em If $\rho,\sigma\in\Hn$ have the same marked simple Hilbert length
spectrum then they either agree or differ by the contragredient involution.}

\medskip

We establish this conjecture when $d=3$. 

\begin{theorem} 
\label{hilbertrigidity}
If $S$ is a closed orientable surface of genus greater than 2,
\hbox{$\rho, \sigma \in  {\mathcal H}_3(S)$}  and $L^H_{\alpha}(\rho) = L^H_{\alpha}(\sigma)$ 
for any $\alpha \in \pi_1(S)$  which is represented by a simple non-separating curve,
then $\rho = \sigma$ or $\rho = \sigma^*$.
\end{theorem}

The classification of the isometries of  intersection on $\mathcal H_3(S)$, 
Theorem \ref{isometriesofintersectionnumber}, is an immediate consequence of 
Theorem \ref{hilbertrigidity} and Proposition \ref{isometries and Hilbert length}.

\begin{proof}
Notice that $\ms{PSL}_3(\mathbb R)=\ms{SL}_3(\mathbb R)$ and that 
if \hbox{$\gamma\in\pi_1(S)$}, then all the eigenvalues of $\rho(\gamma)$ are  positive, since
eigenvalues vary continuously over $\mathcal H_3(S)$ and are positive on the Fuchsian locus.
In particular, if $L^H_\alpha(\rho)=L^H_\alpha(\sigma)$, then
$$\frac{\lambda_1(\rho(\alpha))}{\lambda_3(\rho(\alpha))} = \frac{\lambda_1(\sigma(\alpha))}{\lambda_3(\sigma(\alpha))}>1.$$

We first show that for individual elements the traces and eigenvalues either agree or are consistent
with the contragredient involution.

\begin{lemma} 
\label{eigenvalues equal or}
If $\alpha$ and $\beta$ are represented by simple, non-separating based loops on $S$ which intersect only at
the basepoint and have geometric intersection one,
and $L^H_{\alpha^n\beta}(\rho) = L^H_{\alpha^n\beta}(\sigma)$ for all $n$, then either
\begin{enumerate}
\item 
$\lambda_i(\rho(\alpha)) = \lambda_i(\sigma(\alpha))$  for all $i$, so $\tr(\rho(\alpha))=\tr(\sigma(\alpha))$, or
\item 
$\lambda_i(\rho(\alpha)) = \lambda_i(\sigma(\alpha^{-1}))=\lambda_i(\sigma^*(\alpha))$  for all $i$,
so $\tr(\rho(\alpha))=\tr(\sigma^*(\alpha))$.
\end{enumerate} 
\end{lemma}

\begin{proof}
As in the proof of Lemma \ref{length to trace}, let
$A = \rho(\alpha)$, $B = \rho(\beta)$ and $A^nB= \rho(\alpha^n\beta)$ and $\lambda_i(n)=\lambda_i(A^nB)$. 
Similarly, let $\hat A = \sigma(\alpha)$, $\hat B = \sigma(\beta)$ and $\hat A^n\hat B = \sigma(\alpha^n\beta)$
and let $\hat\lambda_i(n)=\lambda_i(\hat A^n\hat B)$. 
If $(b^i_j)$ is the matrix of $B$ with respect to the basis $\{e_i(A)\}$, 
then, $b^1_1$, $b^1_2$, and $ b^2_1$ are all non-zero
by Lemma \ref{nonzero coefficients}, so
Lemma \ref{specradiusexpansion} implies that
$$\frac{\lambda_1(n)}{\lambda_1^n } = b^1_1 + \frac{b^1_2b^2_1}{b^1_1} \left(\frac{\lambda_2}{\lambda_1}\right)^n+ o\left(\left(\frac{\lambda_2}{\lambda_1}\right)^n\right)$$
where $\lambda_i=\lambda_i(A)$.
Similarly, applying Lemma \ref{specradiusexpansion} to $\rho^*$ and 
noting that $\lambda^{-1}_i(\rho^*(\gamma))=\lambda_{4-i}(\rho(\gamma))$ for all $\gamma\in\pi_1(S)$, gives that 
$$\frac{\lambda_3^n}{\lambda_3(n) } = d^1_1 + \frac{d^1_2d^2_1}{d^1_1} \left(\frac{\lambda_3}{\lambda_2}\right)^n+ o\left(\left(\frac{\lambda_3}{\lambda_2}\right)^n\right)$$
where  $(d^i_j)$ is the matrix of $(B^{-1})^T$ in the basis $\{e_i((A^{-1})^T)\}$.

Taking the product of the previous two equations gives
\begin{equation}
\begin{split}
\left(\frac{\lambda_1(n)}{\lambda_3(n)}\right) \ \left(\frac{\lambda_3}{\lambda_1}\right)^n   = & b^1_1d^1_1 + \frac{d^1_1b^1_2b^2_1}{b^1_1} \left(\frac{\lambda_2}{\lambda_1}\right)^ n +  \frac{b^1_1d^1_2d^2_1}{d^1_1} \left(\frac{\lambda_3}{\lambda_2}\right)^n \\
  &\quad \quad \quad \quad  + o\left(\left(\frac{\lambda_3}{\lambda_2}\right)^n\right) + o\left(\left(\frac{\lambda_2}{\lambda_1}\right)^n\right) .\\
\end{split}
\label{hilbert-taylor}
\end{equation}
One obtains an analogous equality for $\sigma$, and since the left hand sides are equal by assumption,
we see that
\begin{align}
\label{hilbert-taylor-equal}
 & b^1_1d^1_1 + \frac{d^1_1b^1_2b^2_1}{b^1_1} \left(\frac{\lambda_2}{\lambda_1}\right)^ n +  \frac{b^1_1d^1_2d^2_1}{d^1_1} \left(\frac{\lambda_3}{\lambda_2}\right)^n+ o\left(\left(\frac{\lambda_3}{\lambda_2}\right)^n\right) + o\left(\left(\frac{\lambda_2}{\lambda_1}\right)^n\right) \nonumber \\ 
 = & \hat b^1_1\hat d^1_1 + \frac{\hat d^1_1\hat b^1_2\hat b^2_1}{\hat b^1_1} \left(\frac{\hat\lambda_2}{\hat\lambda_1}\right)^ n +  \frac{\hat b^1_1\hat d^1_2\hat  d^2_1}{\hat d^1_1} \left(\frac{\hat\lambda_3}{\hat\lambda_2}\right)^n+ o\left(\left(\frac{\hat\lambda_3}{\hat\lambda_2}\right)^n\right) + o\left(\left(\frac{\hat\lambda_2}{\hat\lambda_1}\right)^n\right)
 \end{align}
where $\hat \lambda_i=\lambda_i(\hat A)$ and  $(\hat b^i_j)$ and $(\hat d^i_j)$ are the matrix representatives of $\hat B$  
and $(\hat B^{-1})^T$ with respect to the bases $\{ e_i(\hat A)\}$ and $\{e_i((A^{-1})^T)\}$ respectively.
Since $\lim\frac{\lambda_{i+1}^n}{\lambda_{i}^n}= 0$ and $\lim\frac{\hat\lambda_{i+1}^n}{\hat\lambda_{i}^n}=0$ 
for $i=1,2$, we see that $b^1_1 d_1^1 = \hat b^1_1 \hat d_1^1 .$

Lemma \ref{nonzero coefficients} implies that all the coefficients in Equation (\ref{hilbert-taylor-equal}) are non-zero.
We further show that they are all positive.

\begin{lemma}
\label{positive coefficients}
Suppose that $\alpha$ and $\beta$ are represented by  simple based loops on $S$ which intersect 
only at the basepoint  and have geometric intersection  one. If $\rho\in\mc H_3(S)$ and $B=(b_i^j)$ is a
matrix representing $\rho(\beta)$ in the basis $\{e_i(\rho(\alpha)\}$, then 
$b^1_1$ and $b^1_2b^2_1$ are positive.
\end{lemma}

\begin{proof}
We may normalize $\rho$ so that $\{e_i(\rho(\alpha)\}$ is the standard basis for $\mathbb R^3$.
The coefficients  $b^1_1$, $b^1_2$ and $b^2_1$ give non-zero
functions on $\mathcal H_3(S)$, so have well-defined signs. 
If $\sigma_0=\tau_3\circ\rho_0$ lies in the Fuchsian locus,
then we may assume that
\begin{eqnarray*}
\sigma_0(\alpha)&=& \tau_3\left( \begin{bmatrix}
    \lambda      &0 \\
    0 & \lambda^{-1} \\
\end{bmatrix} \right)  = \begin{bmatrix}
    \lambda^2      &0&0 \\
    0& 1 & 0 \\
    0& 0 & \lambda^{-2} \\
\end{bmatrix} \cr
 \qquad  \sigma_0(\beta) &=& \tau_3\left( \begin{bmatrix}
    a      &b  \\
    c & d \\
\end{bmatrix} \right) =  \begin{bmatrix}
   a^2    &ab&b^2 \\
    2ac& ad+bc &2bd  \\
    c^2& cd & d^2 \\
\end{bmatrix} 
\end{eqnarray*}
Since $\alpha$ and $\beta$ intersect essentially, the fixed points $z_1$ and $z_2$ of  $z\to \frac{az+b}{cz+d}$ lie on opposite sides of $0$ in $\widehat{\mathbb R}=\partial_\infty \hh$.
Since $z_1$ and $z_2$ are the roots of $cz^2 +(d-a)z+b = 0$, we see that $\frac{b}{c} = -z_1z_2 > 0$, so $bc > 0$. 
Therefore, $b^1_1(\sigma_0) = a^2 > 0$ and $b^2_1b^1_2(\sigma_0) = 2a^2bc > 0$. 
It follows that $b^1_1$ and $b^2_1b^1_2$ are positive on all of $\mathcal H_3(S)$.
\end{proof}

Notice that
$\frac{\lambda_3}{\lambda_2}=\frac{\lambda_2}{\lambda_1}(\rho(\alpha^{-1}))$ and
$\frac{\hat\lambda_3}{\hat\lambda_2}=\frac{\lambda_2}{\lambda_1}(\sigma(\alpha^{-1}))$.
Then, by considering the second order terms in Equation \ref{hilbert-taylor-equal}, we see that
there exists $\epsilon_1,\epsilon_2\in\{\pm 1\}$ such that 
$$
\frac{\lambda_2}{\lambda_1}(\rho(\alpha^{\epsilon_1}))=\frac{\lambda_2}{\lambda_1}(\sigma(\alpha^{\epsilon_2})).
$$
Since we have assumed that
$$
\frac{\lambda_3}{\lambda_1}(\rho(\alpha^{\epsilon_1}))=L_{\alpha^{\epsilon_1}}^H(\rho)=
L_{\alpha^{\epsilon_1}}^H(\sigma)=L_{\alpha^{\epsilon_2}}^H(\sigma)=
\frac{\lambda_3}{\lambda_1}(\sigma(\alpha^{\epsilon_2}))
$$
and 
$$(\lambda_1\lambda_2\lambda_3)(\rho(\alpha^{\epsilon_1}))=(\lambda_1\lambda_2\lambda_3)(\sigma(\alpha^{\epsilon_2}))=1,$$
we see that 
$\left(\lambda_1(\rho(\alpha^{\epsilon_1})\right)^3=\left(\lambda_1(\sigma(\alpha^{\epsilon_2})\right)^3$,
so $\lambda_1(\rho(\alpha^{\epsilon_1}))=\lambda_1(\sigma(\alpha^{\epsilon_2}))$, 
hence $\lambda_i(\rho(\alpha^{\epsilon_1}))=\lambda_i(\sigma(\alpha^{\epsilon_2}))$ for all $i$. 
If $\epsilon_1=\epsilon_2$, then we are in case (1), while if $\epsilon_1=-\epsilon_2$ we are in case (2).

\end{proof}

We next show that if $\tr(\rho(\alpha))=\tr(\sigma(\alpha))$ and $\tr(\rho(\alpha))\ne\tr(\sigma^*(\alpha))$,
then we may control the traces of images of simple based loops having geometric intersection one with
$\alpha$.

\begin{lemma}
\label{tracesagree}
Suppose that $S$ is a closed orientable surface of genus greater than 1,
\hbox{$\rho, \sigma \in  {\mathcal H}_3(S)$}  and $L^H_{\gamma}(\rho) = L^H_{\gamma}(\sigma)$ 
for any $\gamma \in \pi_1(S)$  which is represented by a simple, non-separating curve.
If $\alpha\in\pi_1(S)$ is represented by a simple, non-separating based loop,
$$\tr(\rho(\alpha))=\tr(\sigma(\alpha))\quad {\rm and}\quad \tr(\rho(\alpha))\ne\tr(\sigma^*(\alpha))$$
and $\beta\in\pi_1(S)$ is represented by a simple non-separating based loop intersecting $\alpha$ only at the basepoint
and having geometric intersection one with $\alpha$, then
$\tr(\rho(\beta)) = \tr(\sigma(\beta))$.
 \end{lemma}
 
\begin{proof}
We adopt the notation of Lemma \ref{eigenvalues equal or}, and notice
that  Lemma \ref{eigenvalues equal or} implies that
that $\lambda_i=\lambda_i(\rho(\alpha))=\lambda_i(\sigma(\alpha))=\hat\lambda_i$ for all $i$.

If there is an infinite sequence $\{n_k\}$ of positive numbers such that $\tr(\rho(\alpha^{n_k}\beta))=\tr(\sigma(\alpha^{n_k}\beta))$,
then,
$$\lambda_1^{n_k} b^1_1 + \lambda^{n_k}_2 b^2_2 + \lambda^{n_k}_3 b^3_3 = \lambda_1^{n_k} \hat b^1_1 + \lambda^{n_k}_2 \hat b^2_2 + \hat\lambda^{n_k}_3 \hat b^3_3$$
for all $n_k$. So, by considering the leading terms, we see that $b_1^1=\hat b_1^1$. 
Considering the remaining terms, we conclude that $b_2^2=\hat b_2^2$ and $b_3^3=\hat b_3^3$, so
$\tr (\rho(\beta))=\tr(\sigma(\beta))$. 

If not,  then, by Lemma \ref{eigenvalues equal or},
$\tr(\rho(\alpha^n\beta))=\tr(\sigma^*(\alpha^n\beta))$ for all sufficiently large $n$, so
$$\lambda_1^{n} b^1_1 + \lambda^{n}_2 b^2_2 + \lambda^{n}_3 b^3_3 = \lambda_3^{-n} \hat d^1_1 + \lambda^{-n}_2 \hat d^2_2 + \lambda^{-n}_1 \hat d^3_3$$ 
for all sufficiently large $n$. Since $b^1_1\ne 0$ and $\hat d_1^1\ne 0$, we conclude, by considering leading terms, that
$\lambda_1 = \lambda_3^{-1}$,  so $\lambda_2 = 1$. However, this implies that 
$\lambda_i(\rho(\alpha)) = \lambda_i(\sigma^*(\alpha^{-1}))$ for all $i$, so $\tr(\rho(\alpha))=\tr(\sigma^*(\alpha))$,
which contradicts our assumptions.
\end{proof}

If $\tr(\rho(\alpha))=\tr(\sigma(\alpha))$ for any $\alpha$ represented by a simple non-separating curve,
then Theorem \ref{tracerigidity} implies that $\rho=\sigma$. Similarly, if $\tr(\rho(\alpha))=\tr(\sigma^*(\alpha))$ 
for any $\alpha$ represented by a simple non-separating curve,
then Theorem \ref{tracerigidity} implies that $\rho=\sigma^*$. Therefore,
we may assume that there exists a simple non-separating based loop $\alpha$ so that 
$\tr(\rho(\alpha)) = \tr(\sigma(\alpha))$ and  $\tr(\rho(\alpha)) \ne \tr(\sigma^*(\alpha))$.

Let $\beta$ be a simple, non-separating based loop intersecting $\alpha$ only at the basepoint which
has geometric intersection one with $\beta$. Since $\tr(\rho(\alpha)) \ne \tr(\sigma^*(\alpha))$
and $\tr(\rho(\beta))$ and $\tr(\sigma(\beta))$ are non-zero,
there exists $n$ so that  $\tr(\rho(\alpha^n\beta)) \ne \tr(\sigma^*(\alpha^n\beta))$. Moreover,
Lemma \ref{tracesagree} implies that $\tr(\rho(\alpha^n\beta)) = \tr(\sigma(\alpha^n\beta))$.
Extend $\alpha,\alpha^n\beta$ to a standard set of generators $\mathcal S=\{\alpha_1,\beta_1,\ldots,\alpha_g,\beta_g\}$
so that $\alpha=\alpha_1$ and $\alpha^n\beta=\beta_1$.

The remainder of the proof now mimics the proof of Theorem \ref{tracerigidity}. 
Notice that for the standard generators, if $j>i>1$, then $\alpha_i\alpha_j$ and $\alpha_i\beta_j^{-1}$ can,
and for the remainder of the proof will be, represented
by simple non-separating based loops which intersect $\alpha_1$ and $\alpha_i$ only at the basepoint,
with geometric intersection zero.
There exists a based loop $\gamma$ which intersects each curve in the collection
$\{\alpha_1,\alpha_2,\alpha_2\alpha_3,\ldots,\alpha_2\alpha_g,\alpha_2\beta^{-1}_3,\ldots,\alpha_2\beta_g^{-1}\}$ 
only at the basepoint and with geometric intersection one, see Figure \ref{generators}.
Moreover, if $\eta$ is either $\alpha_2\alpha_i$ or $\alpha_2\beta_i^{-1}$, with $i\ge 3$, then every curve of
the form $\eta^p\alpha_2^q\gamma\alpha_1^r$ is freely homotopic to a simple based loop, in the based
homotopy class of $\alpha_1^r\eta^p\alpha_2^q\gamma$,
which has geometric intersection one with $\alpha_1$ and intersects $\alpha_1$ only at the basepoint.
It then follows
from Lemma \ref{tracesagree}  that 
$$\tr(\rho(\eta^p\alpha_2^q\gamma\alpha_1^r))=\tr(\sigma(\eta^p\alpha_2^q\gamma\alpha_1^r))$$
for all $p,q,r\in\mathbb Z$. 
Proposition \ref{pqr result} then implies that  $\rho$ and $\sigma$
are conjugate on $\braket{\eta,\alpha_2,\alpha_1}$.
In particular, we
may assume that $\rho$ and $\sigma$ agree on $\braket{\alpha_1,\alpha_2,\alpha_3}=\braket{\alpha_2\alpha_3,\alpha_2,\alpha_1}$.
\begin{figure}[htbp] 
   \centering
   \includegraphics[width=3in]{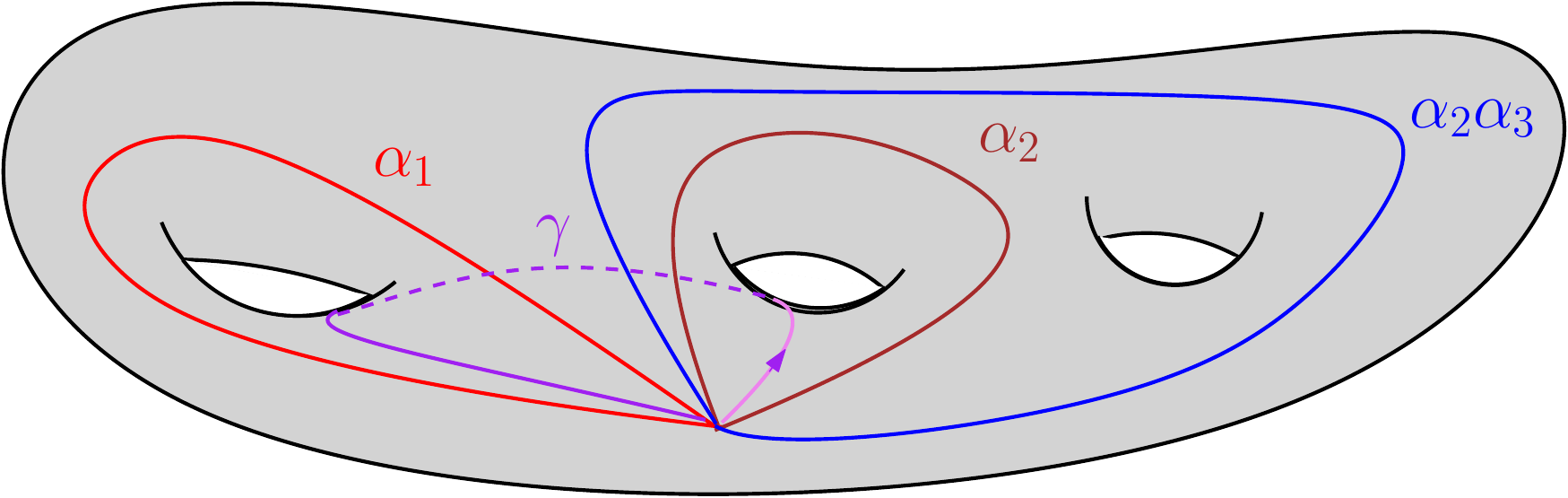} 
   \caption{The curves $\alpha_1$, $\alpha_2$, $\alpha_2\alpha_3$ and $\gamma$ on a surface of genus 3}
   \label{generators}
\end{figure}
If $\eta=\alpha_2\alpha_i$, with $i\ge 4$, then, since $\rho$ and $\sigma$ agree on
$\braket{\alpha_1,\alpha_2,\alpha_3}$ and are conjugate on $\braket{\eta,\alpha_2,\alpha_1}$, Lemma
\ref{conjugacytrivial} implies that they agree on $\eta$ and hence on $\alpha_2^{-1}\eta=\alpha_i$.
Similarly, if $\eta=\alpha_2\beta_i^{-1}$, with $i\ge 3$, we can use Lemma \ref{conjugacytrivial}
to show that $\rho$ and $\eta$ agree on $\eta$ and hence on $\beta_i$.

It remains to check that $\rho$ and $\sigma$ agree on $\beta_1$ and $\beta_2$. Recall that there exists a
homeomorphism $h:S\to S$ so that $h\circ \alpha_i=\beta_i$ and $h\circ \beta_i=\alpha_i$. Then
$\hat\rho=\rho\circ h_*$ and $\hat\sigma=\sigma\circ h_*$ are Hitchin representations. The above
argument shows that $\hat\rho$ and $\hat\sigma$ are conjugate on $\braket{\alpha_1,\alpha_2,\alpha_3,\beta_3}$,
which implies that $\rho$ and $\sigma$ are conjugate on $\braket{\beta_1,\beta_2,\beta_3,\alpha_3}$. Since
$\rho$ and $\sigma$ agree on $\alpha_3$ and on $\beta_3\alpha_3\beta_3^{-1}$ (which have
non-intersecting axes), Lemma \ref{conjugacytrivial} implies that $\rho$
and $\sigma$ agree on $\beta_1$ and $\beta_2$, which completes the proof.
\end{proof}

\section{Infinitesmal Simple Length Rigidity}
\label{infinitesmal}

In this section, we prove that the differentials of simple length functions generate the cotangent space of a Hitchin component.
In earlier work \cite[Prop. 10.3]{BCLS} we showed that the differentials of all length functions
generate the cotangent space, and that result played a key role in the proof that the pressure metric
on the Hitchin component is non-degenerate.

\begin{proposition}
Suppose that $S$ is a closed orientable surface of genus greater than 2 and $\rho \in \Hn$.
If $v \in \ms T_\rho({\mathcal H}_d(S))$ and ${\rm D}L_\alpha(v) = 0$ for every simple non-separating curve $\alpha$,
then $v = 0$.

Moreover, if ${\rm D}\tr_\alpha(v) = 0$ for every simple non-separating curve $\alpha$,
then $v = 0$.
\label{cotangent}
\end{proposition}

\begin{proof}
We recall that there exists a component $\xHn$ of ${\rm Hom}(\pi_1(S),\sln)$ which is
an analytic manifold, so
that the projection map $\pi:\xHn\to \Hn$ is real analytic and is obtained by
quotienting out by the action of $\sln$ by conjugation, see Hitchin \cite{hitchin}.
Any smooth path in $\Hn$ 
lifts to a smooth  path in $\xHn$. The real-valued functions 
$\widetilde\tr_\alpha$ and $\widetilde\lambda_{i,\alpha}$ on $\xHn$ given by
$\widetilde{\tr_\alpha}(\tilde\rho)=\tr(\tilde\rho(\alpha))$ and
$\widetilde{\lambda_{i,\alpha}}(\tilde\rho)=\lambda_i(\tilde\rho(\alpha))$ are analytic and $\sln$-invariant,
so descend to real analytic functions $\tr_\alpha$ and $\lambda_{i,\alpha}$ on $\Hn$.
(Notice that if we chose a different component of ${\rm Hom}(\pi_1(S),\sln)$ as $\xHn$, then
$\tr_\alpha$ and $\lambda_{i,\alpha}$ could differ up to sign.)

The proof of Proposition \ref{cotangent} has the same basic structure as the proof of our simple length rigidity result.
We first establish an infinitesimal version of Theorem \ref{equivalence}. 

\begin{lemma}
\label{trace derivative}
If $S$ is a closed orientable surface of genus more than 1, $\rho \in {\mathcal H}_d(S)$ and
$v \in \ms T{\mathcal H}_d(S)$ then ${\rm D}L_\alpha(v) = 0$ for every simple non-separating curve $\alpha$
if and only if  ${\rm D}\tr_\alpha(v) = 0$ for every simple non-separating curve $\alpha$. In both cases 
${\rm D}\lambda_{i,\alpha}(v) = 0$ for all $i$.
\end{lemma}

\begin{proof}
Let $\{\rho_t\}_{t\in (-1,1)}$ be an analytic path in $\xHn$ such that if
$\dt{\rho}_0=\frac{{\rm d}}{{\rm d}t}\big|_{t=0}\rho_t$ then $d\pi(\dt{\rho}_0)=v$. 

First assume that ${\rm D}L_\alpha(v) = 0$ for every simple non-separating curve $\alpha$.
Choose a simple based loop $\beta$ 
which intersects $\alpha$ only at the basepoint and has geometric intersection one with $\alpha$.
Let $A(t)=\rho_t(\alpha)$, $B(t)=\rho_t(\beta) $ and $\lambda_i(t)=\lambda_{i,\alpha}(\rho_t)$.
Let $\lambda(n,t)=|\lambda_1(A(t)^nB(t))|$ and notice that our assumptions imply that
$$\dt{\lambda}(n,0)=\frac{{\rm d}}{{\rm d}t}\Big|_{t=0}\lambda(n,t)=0$$ 
for all $n$.
Let $(b^i_j(t))$ be the matrix representative of $B(t)$ in the basis $\{e_i(A(t))\}$ and notice that we may choose
$\{e_i(A(t))\}$ to vary analytically, so that the coefficients $(b^i_j(t))$ vary analytically.

If $v\in\mathbb R^{d-1}$,  let $D(v)\in\sln$ be chosen so that its matrix is diagonal with respect to the basis $\{e_i(A(t))\}$
with diagonal entries $(1,v_1,\ldots,v_{d-1})$, 
then $M(v,t)=D(v)B(t)$ depends analytically on $v$ and $t$.
Notice that $M(\vec 0,0)$ has a  simple eigenvalue $b^1_1(0)$ with eigenvector $e_1$. 
By Lemma \ref{eigenvalue-analytic}  
there exists an open neighborhood $V$ of the origin in $\mathbb R^{d-1}\times \mathbb R$ and an
an analytic function $F:V\to \mathbb R$ so that
$$\lambda_1(M(v,t)) = F(v,t).$$
Since
$$\frac{A(t)^nB(t)}{\lambda_1(t)^n}  = M\left(\left(\frac{\lambda_2(t)}{\lambda_1(t)}\right)^n,\ldots,\left(\frac{\lambda_d(t)}{\lambda_1(t)}\right)^n,t\right)$$
and
$$\left(\left(\frac{\lambda_2(t)}{\lambda_1(t)}\right)^n,\ldots,\left(\frac{\lambda_d(t)}{\lambda_1(t)}\right)^n,t\right) \in V,$$
for all sufficiently large $n$ and $t$ sufficiently close to 0,
$$\frac{\lambda(n,t)}{\lambda_1(t)^n} = \frac{\lambda_1(A^n(t)B(t))}{\lambda_1(t)^n} = \lambda_1\left(\frac{A^n(t)B(t)}{\lambda_1(t)^n}\right) =   F\left(\left(\frac{\lambda_2(t)}{\lambda_1(t)}\right)^n,\ldots,\left(\frac{\lambda_d(t)}{\lambda_1(t)}\right)^n, t\right).$$
Letting $u_i(t) = \frac{\lambda_{i+1}(t)}{\lambda_1(t)}$, we see that
$$
\lambda(n,t) =   \lambda_1(t)^nF\left(u_1(t)^n,\ldots,u_{d-1}(t)^n, t\right) .
$$
Since $\dt\lambda_1(0) = 0$ and $\dt\lambda(n,0)=0$,
$$\frac{{\rm d}}{{\rm d}t}\Big|_{t=0} F\left(u_1(t)^n,\ldots,u_{d-1}(t)^n, t\right)  = 0$$
for all large enough $n$.
Therefore,
\begin{equation}
\label{partialderivs}
\frac{\partial F}{\partial t}(u_1(0)^n,\ldots,u_{d-1}(0)^n,0)+\sum_{i=1}^{d-1}\frac{\partial F}{\partial v_i}(u_1(0)^n,\ldots,u_{d-1}(0)^n,0) n u_i^{n-1}(0)\dt u_i(0) =0,
\end{equation}
for all large enough $n$, so
$$ \frac{\partial F}{\partial t}(0,\ldots,0,0) = 0.$$
Moreover, since $\frac{\partial F}{\partial t}$ is analytic,
$$\frac{\partial F}{\partial t}(u_1(0)^n,\ldots,u_{d-1}(0))^n,0)=\sum_{i=1}^{d-1}\left(\frac{\partial^2 F}{\partial v_i\partial t}(\vec 0,0)u_i(0)^n+o(u_i(0)^n)\right)$$
so, since $1>|u_1(0)|>|u_i(0)|>0$ for all $i\ge 2$,
$$ \lim_{n\rightarrow \infty}\frac{1}{nu_1(0)^{n-1}}\frac{\partial F}{\partial t}(u_1(0)^n,\ldots,u_{d-1}(0))^n,0)=
\lim_{n\rightarrow \infty}\sum_{i=1}^{d-1}\frac{u_i(0)^n}{nu_1(0)^{n-1}}\left(\frac{\partial^2 F}{\partial v_i\partial t}(0,\ldots,0,0)\right)=0$$
Equation (\ref{partialderivs}) then implies that
$$\lim_{n\rightarrow \infty}\frac{1}{nu_1(0)^{n-1}}\left(\sum_{i=1}^{d-1}\frac{\partial F}{\partial v_i}(u_1(0)^n,\ldots,u_{d-1}(0))^n,0) n u_i^{n-1}(0)\dt u_i(0)\right)=
 \frac{\partial F}{\partial v_1}(0,\ldots,0,0)\dt u_1(0) =0$$
As in the proof of Lemma \ref{specradiusexpansion},  we calculate that
$$\frac{\partial F}{\partial v_1}(0,\ldots,0, 0) = \frac{{\rm d}}{{\rm d}s}\Big|_{s=0}F(s,0,\ldots,0) = \frac{{\rm d}}{{\rm d}s}\Big|_{s=0} \lambda_1\left(D(1,s,0,\ldots,0)B(0)\right) =  \frac{{\rm d}}{{\rm d}s}\Big|_{s=0}
\lambda_1\left(\begin{bmatrix}  b^1_1(0) & b^1_2(0) \\ s b^2_1(0)& s b^2_2(0) \end{bmatrix}\right),$$
so
$$ \frac{\partial F}{\partial v_1}(\vec 0,0) = \frac{b^1_2(0)b^2_1(0)}{b^1_1(0)}.$$
Lemma \ref{nonzero coefficients} implies that $b_1^1(0)$, $b_2^1(0)$ and $b_1^2(0)$ are
non-zero, so \hbox{$\frac{\partial F}{\partial v_1}(0,\ldots,0,0) \neq 0$}. 
Therefore, $\dt{u}_1(0) = 0$ and,
since $\dt{\lambda}_1(0) = 0$, we have
$$0 = \dt{u}_1(0) = \frac{{\rm d}}{{\rm d}t}\Bigg|_{t=0}\left( \frac{\lambda_2(t)}{\lambda_1(t)}\right) = \frac{\dt{\lambda}_2(0)\lambda_1(0) - \dt{\lambda}_1(0)\lambda_2(0)}{\lambda_1(0)^2}= \frac{\dt{\lambda}_2(0)}{\lambda_1(0)},$$
so $\dt{\lambda}_2(0) = 0$.

We may iteratively consider the 1-parameter families of  representations given by $\{E^k(\rho_t)\}$
and apply the same analysis to conclude that 
$\dt{\lambda}_{i,\alpha}(0) = 0 $ for all $i$, and thus that $D\tr_\alpha(v) = 0$.

Now assume  that $D\tr_\alpha(v) = 0$ for every $\alpha\in\pi_1(S)$ represented by a simple non-separating curve.
Given a simple, non-separating curve $\alpha$ represented by a simple based loop, we again 
choose a simple based loop $\beta$ which intersects $\alpha$ only at the basepoint and has geometric
intersection one with $\alpha$.
Notice that
$$Tr(\rho_t(\alpha^n\beta)) = \sum_{i=1}^d \lambda_i^n(\rho_t(\alpha))\tr(\p_i(\rho_t(\alpha))\rho_t(\beta)) = 
\sum_{i=1}^d h_i(t)\lambda_i^n(t).$$
where $h_i(t) = \tr(\p_i(\rho_t(\alpha))\rho_t(\beta))\neq 0$ for all $t$.
Differentiating, and noting that $D\tr_{\alpha^n\beta}(v)=0$ for all $n$,  we see that
$$0 = \sum_{i=1}^d   \dt{h}_i(0)\lambda_i^n(0) + n h_i(0)\dt{\lambda}_i(0)\lambda_i(0)^{n-1}$$
for all $n$.
Since $h_i(0) \neq 0$ and $\lambda_i(0)\ne 0$, it must be that
$\dt{h}_1(0)=0$ and 
$\dt{\lambda}_1(0) = 0$, so $DL_\alpha(v) = 0$.
 \end{proof}

We next generalize the proof of Theorem \ref{conjugateontriples-trace} to obtain a criterion
guaranteeing that
$v$ is infinitesmally trivial on its restriction to  certain $3$-generator subgroups.

\begin{lemma}
\label{trivial derivative on triples}
Suppose that  $\rho \in {\mathcal H}_d(S)$, $v \in \ms T_\rho({\mathcal H}_d(S))$ and ${\rm D}L_\eta(v) = 0$ for 
every simple non-separating curve $\eta$ on $S$.
If $\alpha, \beta,\delta \in \pi_1(S)$ are represented by simple based loops which intersect only at the basepoint,
and are freely homotopic to  a collection of mutually disjoint and non-parallel, non-separating closed curves
which do not bound a pair of pants in $S$, and $\{\rho_t\}$ is a path in $\xHn$ so that ${\rm D}\pi(\dt{\rho}_0)=v$, then
there exists a path $\{C_t\}$ in $\sln$, so that  $C_0=I$ and if
$\eta\in\braket{\alpha,\beta, \delta}$, then 
$$\frac{{\rm d}}{{\rm d}t}\Big|_{t=0}(C_t\rho_t(\eta)C_t^{-1})=0\in\mathfrak{sl}(n,\mathbb R).$$
\end{lemma}

\begin{proof}
Lemma \ref{good configuration} guarantees that there exist based loops $\hat\alpha$, $\hat\beta$, $\gamma$ and $\hat\delta$
as in Figure \ref{4curves}, which intersect only at the basepoint, so that $\hat\alpha$, $\hat\beta$ and $\hat\delta$ are
freely homotopic to a collection of mutually disjoint, non-parallel, non-separating curves and $\gamma$ has
geometric intersection one with each such that
$$\braket{\alpha,\beta, \delta}=\braket{\hat\alpha,\hat\beta,\hat\delta}.$$
We may thus assume that $\alpha$, $\beta$ and $\delta$ already have this form. 

We may also,
by possibly re-ordering  $\alpha$, $\beta$ and $\delta$, assume that
$\alpha^p\beta^q\gamma\delta^r$ is represented by a simple non-separating curve for all $p,q,r\in\mathbb Z$.
We next generalize the proof of Proposition \ref{pqr result} to show that 
\hbox{${\rm D}\left(\frac{{\bf T}_{i,j,k}(\alpha,\beta,\delta)}{{\bf T}_{j,k}(\beta,\delta)}\right)(v)=0$}
for all $i$, $j$ and $k$.

Recall that 
$$\tr (\rho(\alpha^p\beta^q\gamma\delta^r)) = \sum_{i=1}^d \lambda_{i,\alpha}(\rho)^p \tr\left( \p_i(\rho(\alpha))\rho(\beta^q\gamma\delta^r)\right).$$
Differentiating and noting that, by Lemma \ref{trace derivative},
${\rm D}\tr_{\alpha^p\beta^q\gamma\delta^r}(v)=0$ for all $p$, $q$ and $r$ and  ${\rm D}\lambda_{i,\alpha}(v) = 0$ for all $i$, 
one sees that
$$ \sum_{i=1}^d \lambda_{i,\alpha}(\rho)^p {\rm D}{\bf T}_{i,0}(\alpha,\beta^q\gamma\delta^r)(v) = 0$$
for all $p$.
By examining terms of different orders and taking limits, we see that
$${\rm D}{\bf T}_{i,0}(\alpha,\beta^q\gamma\delta^r)(v) = 0$$
for all $i$, $q$ and $r$.
Repeating, as in the proof of Proposition \ref{pqr result}, we find that
$${\rm D}{\bf T}_{i,j,0,k}(\alpha,\beta,\gamma,\delta)(v) = 0$$
for all $i$, $j$, and $k$. 
Similarly, by considering $\beta^q\gamma\delta^r$, we see that  
$${\rm D}{\bf T}_{j,0,k}(\beta,\gamma,\delta)(v) = 0$$
for all $j$ and $k$.

Recall, from  part (4) of Proposition \ref{tracenonzero}, that
$${\bf T}_{i,j,0,k}(\alpha,\beta,\gamma,\delta)(\rho) = 
{\bf T}_{j,0,k}(\beta,\gamma,\delta)(\rho)\left(\frac{{\bf T}_{i,j,k}(\alpha,\beta,\delta)(\rho)}{{\bf T}_{j,k}(\beta,\delta)(\rho)}\right)\ne 0$$
for all $i$, $j$ and $k$.
Since we have established that the two leftmost terms in this expression are non-zero and have derivative $0$ in the direction $v$,
we conclude that
$${\rm D}\left(\frac{{\bf T}_{i,j,k}(\alpha,\beta,\delta)}{{\bf T}_{j,k}(\beta,\delta)}\right)(v)=0$$
for all $i$, $j$ and $k$.

Let $a_i(t)=e_i(\rho_t(\alpha))$, $a^i(t)=e^i(\rho_t(\alpha))$, $b_j(t)=e_j(\rho_t(\beta))$, $b^j(t)=e^j(\rho_t(\beta))$,
$d_k(t)=e_k(\rho_t(\delta))$ and $d^k(t)=e^k(\rho_t(\delta))$ for all $i,j,k$.
We will assume throughout, by replacing  $\{\rho_t\}$ by $\{C_t\rho_t C_t^{-1}\}$ where $\{C_t\}$ is a path in $\sln$ so
that $C_0=I$,
that $a_i(t)$ are constant as functions of $t$ for all $i$, $b_1(t)$ is constant as a function of $t$, and by scaling
the bases, that $\braket{a^i (t)| b_1(t)}=1$ for all $i$ and $t$, $\braket{a^1(t) | b_j(t)}=1$ for all $j$ and $t$, 
and $\braket{d^k(t) | b_1(t) }=1$   for all $k$ and $t$.
Since $a_i(t)$ is constant and $\frac{{\rm d}}{{\rm d}t}\big|_{t=0}\lambda_{i,\alpha}(\rho_t)=0$, by Lemma \ref{trace derivative},
$\frac{{\rm d}}{{\rm d}t}\big|_{t=0} \rho_t(\alpha)=0$.

Recall, from Proposition \ref{tracenonzero}, that
\begin{equation}
\label{useful}
\frac{{\bf T}_{i,j,k}(\alpha,\beta,\delta)(\rho_t)}{{\bf T}_{j,k}(\beta,\delta)(\rho_t)}=
 \frac{\braket{a^i(t) | b_j(t) }\braket{d^k(t) | a_i(t)}}{\braket{d^k(t) | b_j(t) }}.
\end{equation}
By considering  Equation (\ref{useful}) when $j=1$, we see that
$$\frac{{\bf T}_{i,1,k}(\alpha,\beta,\delta)(\rho_t)}{{\bf T}_{1,k}(\beta,\delta)(\rho_t)} = \braket{d^k(t)| a_i(t)},$$
so, since the left-hand side has derivative 0 at 0 and $a_i(t)$ is constant for all $i$,
$$\frac{{\rm d}}{{\rm d}t}\Big|_{t=0}\big(\braket{d^k(t)|a_i(t)}\big)= \braket{\dt{d^k}(0) | a_i(0)}=0$$
for all $i$ and $k$.
Therefore, $\dt{d^k}(0) = 0 $ for all $k$, so $\dt{d}_k=0$ for all $k$.
Since we also know, from Lemma \ref{trace derivative}, that  $\frac{{\rm d}}{{\rm d}t}\big|_{t=0}\lambda_{i,\delta}(\rho_t)=0$ for all $t$,
it follows that $\frac{{\rm d}}{{\rm d}t}\big|_{t=0} \rho_t(\delta)=0$.

Considering Equation (\ref{useful}) when $i =1$, one obtains
$$\frac{{\bf T}_{1,j,k}(\alpha,\beta,\delta)(\rho_t)}{{\bf T}_{j,k}(\beta,\delta)(\rho_t)}= 
\frac{\braket{a^1(t) | b_j(t) }\braket{d^k(t) | a_1(t)}}{\braket{d^k(t) | b_j(t) }}  =  \frac{\braket{d^k(t) | a_1(t)}}{\braket{d^k(t) | b_j(t) }}.$$
Since the derivative of the left hand side is 0 at 0,  $a_1(t)$ is constant, and $\dt{d^k}(0) = 0 $ for all $k$, we see that 
$$ \frac{\braket{d^k(0) | a_1(0)}}{\braket{d^k(0)| b_j(0)}^2}\braket{d^k(0) | \dt{b}_j(0)}=0,$$
so $\braket{d^k | \dt{b}_j(0)} = 0$  for all $j$ and $k$, so $\dt{b}_j(0) =0$ for all $j$. We may then argue, just as before, that
$\frac{{\rm d}}{{\rm d}t}\big|_{t=0} \rho_t(\beta)=0$.  Therefore, $\frac{{\rm d}}{{\rm d}t}\big|_{t=0} \rho_t(\eta)=0$
for all $\eta\in\braket{\alpha,\beta, \delta}$.
\end{proof}

We are now ready to complete the proof of
Proposition \ref{cotangent}.
Let \hbox{$\mathcal S=\{\alpha_1,\beta_1,\ldots,\alpha_g,\beta_g\}$} be a standard generating set for $\pi_1(S)$.
By Lemma \ref{trivial derivative on triples},
we may choose an analytic family $\{\rho_t\}$  in $ {\rm Hom}(\pi_1(S),\psln)$ 
so that $d\pi(\dt{\rho}_0)=v$ and $\frac{{\rm d}}{{\rm d}t}\big|_{t=0} \rho_t(\gamma)=0$
for all $\eta\in \braket{\alpha_1, \alpha_2, \alpha_3}$.  

For any $\delta \in \mathcal{S}-\{\alpha_1,\alpha_2,\alpha_3,\beta_1,\beta_2\}$, 
we may apply
Lemma \ref{trivial derivative on triples} to the triple $\{\alpha_1,\alpha_2,\eta\}$ to show that
there exists a family $\{C_t\}$ in $\psln$ so that $C_0=I$ and 
$\frac{{\rm d}}{{\rm d}t}\big|_{t=0} (C_t\rho_t(\gamma)C_t^{-1})=0$ for all $\gamma \in\braket{\alpha_1,\alpha_2,\delta}$.
In particular,
$$\dt{C}_0\rho_0(\alpha_i) C_0^{-1} -C_0\rho_0(\alpha_i)\dt{C}_0 + C_0 \left(\frac{{\rm d}}{{\rm d}t}\Big|_{t=0}\rho_t(\alpha_i)\right)C_0^{-1} =
\dt{C}_0\rho_0(\alpha_i) - \rho_0(\alpha_i)\dt{C}_0=0,$$
so $[\dt{C}_0,\rho_0(\alpha_i)] = 0$ 
for $i=1,2$, Thus, $\dt{C}_0$ is diagonalizable over $\Real$ with respect to  both
$\{e_i(\rho_0(\alpha_1))\}$ and $\{e_i(\rho_0(\alpha_2)\}$.

If $\dt{C}_0 \ne 0$, then $\mathbb R^d$ admits a non-trivial decomposition into eigenspaces of $\dt{C}_0$
with distinct eigenvalues.
Any such eigenspace $W$ is spanned by a sub-collection of $\{e_i(\rho_0(\alpha_1))\}$ and by a sub-collection of
$\{e_j(\rho_0(\alpha_2))\}$.
In particular, some $e_i(\rho_0(\alpha_1))$ is in the sub-space spanned by a subcollection of $\{e_j(\rho_0(\alpha_2))\}$. 
Since $\alpha_1$ and $\alpha_2$ are disjoint curves, this contradicts Theorem \ref{transverse-bases-general0}. 
Therefore, $\dt{C}_0=0$.
 
Since $\dt{C}_0 = 0$ and 
$\frac{{\rm d}}{{\rm d}t}\big|_{t=0}(C_t \rho_t(\delta) C_t^{-1}) = 0$, we  calculate that
$$\dt{C}_0\rho_0(\delta) C^{-1}_0 -C_0\rho_0(\delta)\dt{C}_0 + C_0\left(\frac{{\rm d}}{{\rm d}t}\Big|_{t=0}\rho_t(\delta)\right) C_0^{-1} =
\frac{{\rm d}}{{\rm d}t}\Big|_{t=0}\rho_t(\delta)=0.$$
By considering the subgroups $\braket{\alpha_2,\alpha_3, \beta_1}$ and $\braket{\alpha_1, \alpha_3, \beta_2}$, 
we similarly show that  
$$\frac{{\rm d}}{{\rm d}t}\Big|_{t=0}\rho_t(\beta_1)=0\qquad {\rm and} \qquad \frac{{\rm d}}{{\rm d}t}\Big|_{t=0}\rho_t(\beta_2)=0$$
Since $\frac{{\rm d}}{{\rm d}t}\big|_{t=0}\rho_t(\eta)=0$ for all $\eta\in\mathcal{S}$,
$$\dt{\rho}_0 =0\in \ms T\xHn.$$
Therefore, $v={\rm D}\pi \left(\dt{\rho}_0\right)=0$ as claimed.
\end{proof}

\section{Hitchin representations for surfaces with boundary}
\label{positivereps}

In this section, we observe that our main simple length rigidity result extends to Hitchin representations
of most compact surfaces with boundary.

If $S$ is a compact surface with boundary, we say that a representation $\rho:\pi_1(S)\to \psln$ is a
{\em Hitchin representation} if 
$\rho$ is the restriction of a Hitchin representation $\hat\rho$ of $\pi_1(DS)$ into $\psln$, where $DS$ is
the double of $S$. Labourie and McShane
\cite[Section 9]{labourie-mcshane} show that this is equivalent to assuming that $\rho$ is deformable
to the composition of a convex cocompact Fuchsian uniformization of $S$ and the irreducible representation
through representations so that the image of every peripheral element is purely loxodromic.
(Recall that a non-trivial element of $\pi_1(S)$ is peripheral if it is represented by a curve in $\partial S$.)
Fock and Goncharov \cite{fock-goncharov} refer to such representations as positive representations.

\begin{theorem}
Suppose that $S$ is a compact, orientable surface  of genus $g > 0$ with $p>0$ 
boundary components, and $(g,p)$ is not $(1,1)$  or $(1,2)$.
If $\rho$ and $\sigma$ are two Hitchin representations of $\pi_1(S)$ of dimension $d$
and $L_\rho(\alpha)=L_\sigma(\alpha)$ for any $\alpha$ represented by a simple non-separating curve on 
$S$, then $\rho$ and $\sigma$ are conjugate in $\pgln$.
\end{theorem}

Notice that our techniques don't apply to punctured spheres, since they contain no simple non-separating curves.
In the remaining excluded cases, there are no configurations of three non-parallel simple non-separating
closed curves which do not bound a pair of pants.

\begin{proof}
We choose a generating set 
$$\mathcal S=\{\alpha_1,\beta_1,\ldots\alpha_g,\beta_g,\delta_1,\ldots,\delta_{p-1}\}$$
represented by simple, non-separating based loops which intersect only at the basepoint
so that $\{\alpha_1,\beta_1,\ldots,\alpha_g,\beta_g\}$ is a standard generating set for the
surface of genus $g$ obtained by capping each boundary component of $S$ with a disk,
each $\delta_i$  has geometric intersection one with $\beta_1$ and zero with every  other generator,
as in  Figure \ref{surfacewithboundary}.
Notice that any collection of 3 based loops in $\mathcal S$ which have geometric intersection zero with
each other are freely homotopic to a mutually disjoint, non-parallel collection of simple closed curves
which do not bound a pair of pants.

\begin{figure}[htbp] 
   \centering
   \includegraphics[width=2.5in]{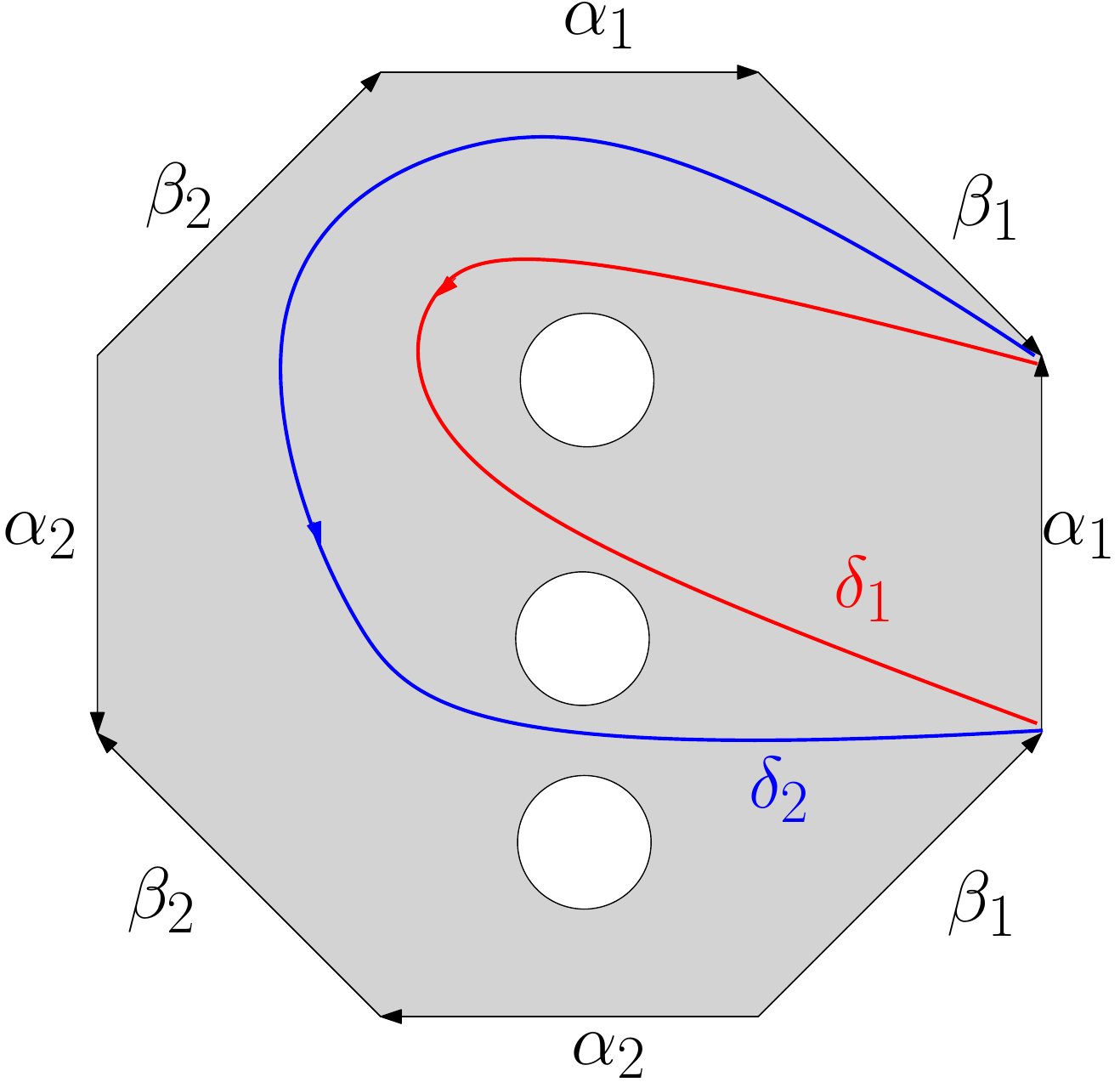} 
   \caption{Our generators on a surface with genus 2 and 3 boundary components}
   \label{surfacewithboundary}
\end{figure}

Throughout the proof we identify $S$ with a subsurface of $DS$ and apply our earlier results
to the representations $\hat\rho$ and $\hat\sigma$ of $\pi_1(DS)$.
Lemma \ref{length to trace} implies that if $\eta\in \pi_1(S)$ is represented
by a simple non-separating curve on $S$, then $|\tr(\rho(\eta))|=|\tr(\sigma(\eta))|$ and
$\lambda_i(\rho(\eta))=\lambda_i(\sigma(\eta))$ for all $i$.

If $g\ge 3$, the proof of Theorem \ref{tracerigidity} generalizes rather immediately. We first apply 
Theorem \ref{conjugateontriples-trace} to $\hat\rho$ and $\hat\sigma$, to see that we may assume,
after conjugation in $\pgln$,
that $\rho$ and $\delta$ agree on $\braket{\alpha_1,\alpha_2,\alpha_3}$.  If 
$\eta\in \mathcal{S}-\{\alpha_1,\alpha_2,\beta_1,\beta_2\}$, 
we may again apply Theorem \ref{conjugateontriples-trace} to show that $\rho$ and
$\sigma$ are conjugate on $\braket{\alpha_1,\alpha_2,\eta}$. Since $\hat\rho$ and $\hat\sigma$ agree on
$\alpha_1$ and $\alpha_2$, Lemma \ref{conjugacytrivial} implies
that $\rho$ and $\sigma$ agree on $\braket{\alpha_1,\alpha_2,\eta}$.
We then consider the triples $\{\alpha_2,\alpha_3,\beta_1\}$ and $\{\alpha_1,\alpha_3,\beta_2\}$ to show
that $\rho$ and $\sigma$ agree on $\beta_1$ and $\beta_2$, and hence that $\rho=\sigma$.

If $g=2$ and $p\ge 2$, we again use Theorem \ref{conjugateontriples-trace} to show that we may conjugate
$\rho$ and $\sigma$ so that they agree on $\braket{\alpha_1,\alpha_2,\delta_1}$. 
If $i\ge 2$, we may again apply Theorem \ref{conjugateontriples-trace} to show
that $\rho$ and $\sigma$  are conjugate on $\braket{\alpha_1,\alpha_2,\delta_i}$
and then Lemma \ref{conjugacytrivial} to show that $\rho$ and $\sigma$ agree on
$\braket{\alpha_1,\alpha_2,\delta_i}$. We  consider the triple 
$\{\alpha_1,\delta_1,\beta_2\}$ to show that $\rho$ and $\sigma$ agree on $\beta_2$.
Therefore, $\rho$ and $\sigma$ agree on $\mathcal S-\{\beta_1\}$. Recall that there exists a
homeomorphism $h:S\to S$ such that $h\circ\alpha_i=\beta_i$ and $h\circ\beta_i=\alpha_i$.
The above argument implies that the Hitchin representations $\rho\circ h_*$ and $\sigma\circ h_*$
are conjugate on $\braket{\alpha_1,\alpha_2,\beta_2}$ and hence that $\rho$ and $\sigma$ are
conjugate on $\braket{\beta_1,\beta_2,\alpha_2}$. Since $\rho$ and $\sigma$ agree on $\beta_2$ and
$\alpha_2$, Lemma \ref{conjugacytrivial} implies that they agree on $\beta_1$. So, we conclude
that $\rho=\sigma$.

If $g=1$ and $p\ge 3$, then $\mathcal{S}=\{\alpha_1,\beta_1,\delta_1,\ldots,\delta_{p-1}\}$.
We first apply Theorem \ref{conjugateontriples-trace} to show that we may conjugate
$\rho$ and $\sigma$ so that they agree on $\braket{\alpha_1,\delta_1,\delta_2}$. 
If $i\ge 3$, we may consider the triple $\{\alpha_1,\delta_1,\delta_i\}$ to see that
$\rho$ and $\sigma$ agree on $\delta_i$. It remains to check that $\rho$ and $\sigma$ agree on $\beta_1$.

\begin{figure}[htbp] 
   \centering
   \includegraphics[width=2in]{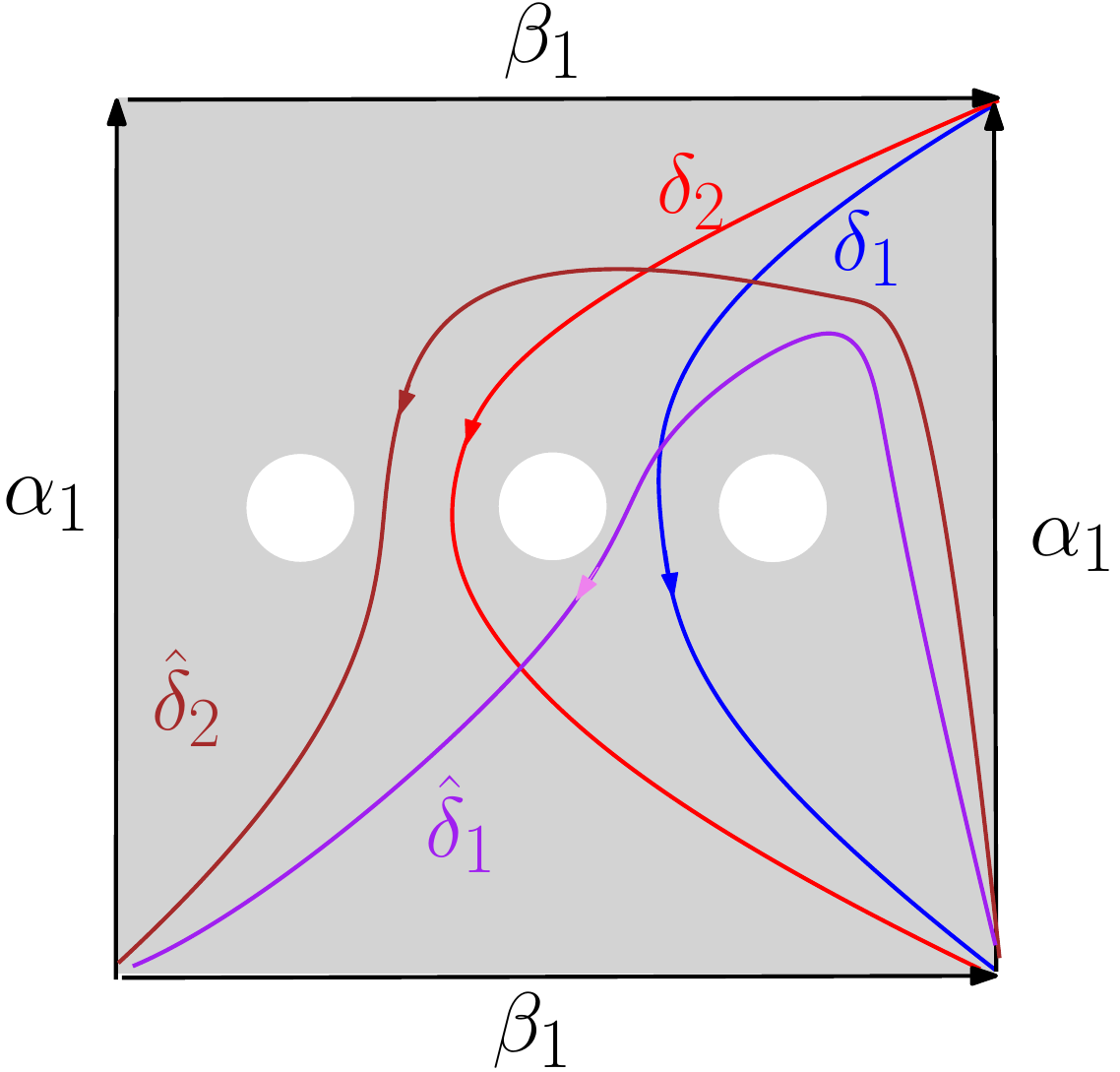} 
   \caption{Curves on a surface of type $(1,p)$ for $p \geq 3$}
   \label{g1pn}
\end{figure}

Let $\hat\delta_i$ be as in Figure \ref{g1pn}, 
so that if  $\mathcal S'=\{\alpha_1,\beta_1,\hat\delta_1,\ldots,\hat\delta_{p-1}\}$, then the based loops
in $\mathcal S'$ intersect only at the basepoint and each
$\hat\delta_i$ has geometric intersection one with $\alpha_1$ and has geometric intersection zero
with every other element of $\mathcal S'$. Notice that $\alpha_1\delta_i=\hat\delta_i\beta_1$ and let $u_i=\alpha_1\delta_i$.
Then, $\rho$ and $\sigma$ agree on the
subgroup  $\braket{\alpha_1,u_1,\ldots,u_{p-1}}$. We may apply the same argument as above to show
that $\rho$ and $\sigma$ are conjugate on  $\braket{\beta_1,\hat\delta_1,\ldots,\hat\delta_{p-1}}$.
Since this subgroup contains $u_1$ and $u_2$, $\rho$ and $\sigma$ agree on $u_1$ and $u_2$, and
$u_1$ and $u_2$ have non-intersecting axes in $\pi_1(DS)$, Lemma \ref{conjugacytrivial}, applied
to $\hat\rho$ and $\hat\sigma$, implies that $\rho$ and $\sigma$ agree on $\braket{\beta_1,\hat\delta_1,\ldots,\hat\delta_{p-1}}$
and hence on $\beta_1$, so $\rho=\sigma$.

\begin{figure}[htbp] 
   \centering
   \includegraphics[width=2.5in]{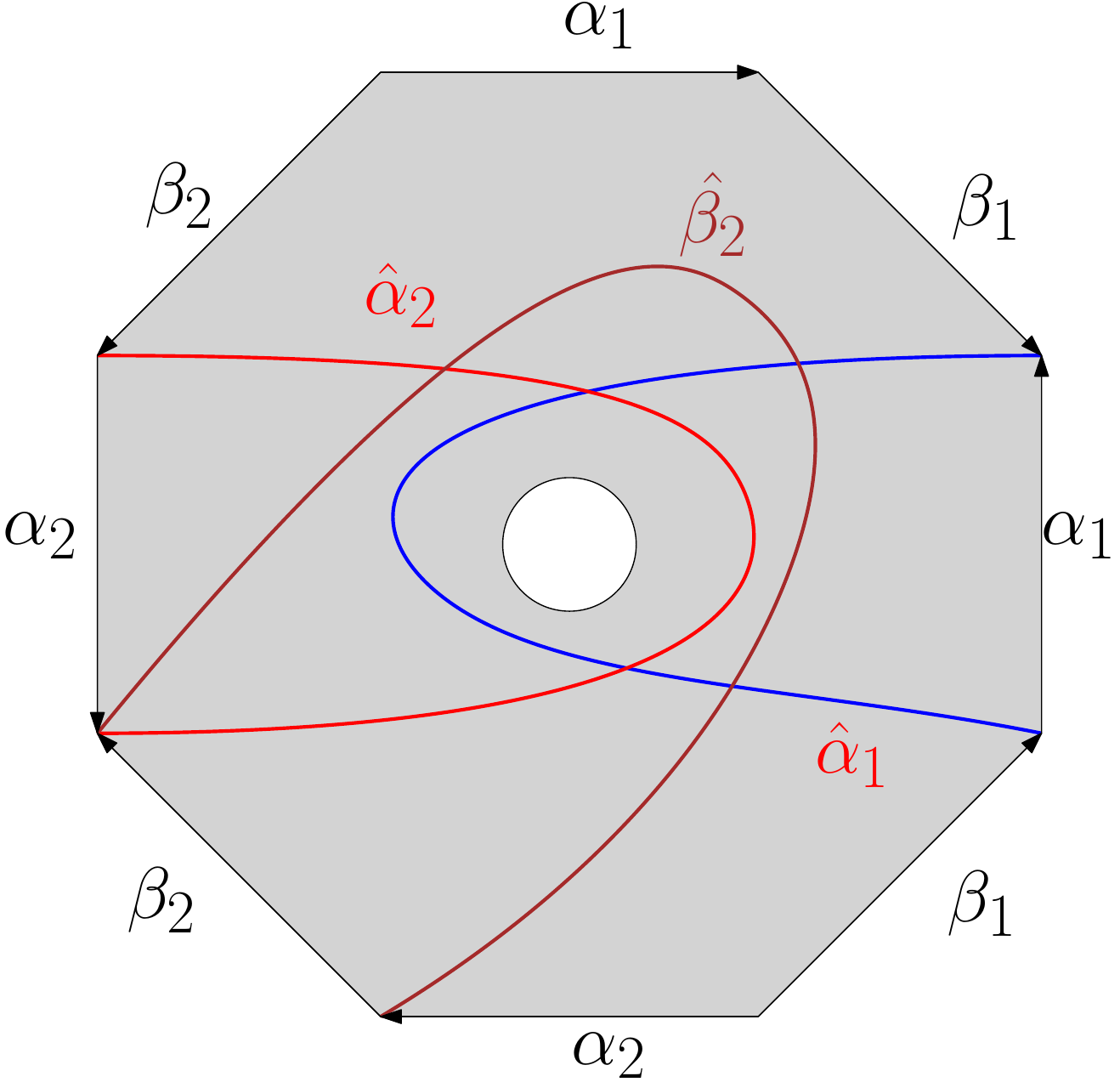} 
   \caption{Genus 2 with 1 puncture}
   \label{g2p1}
\end{figure}

If $g=2$ and $p=1$, then $\mathcal S = \{\alpha_1,\beta_1,\alpha_2,\beta_2\}$. 
We will consider the based loops $\hat\alpha_i$ and $\hat\beta_i$ as in Figure \ref{g2p1}. 
As the based loops $\{\alpha_1,\alpha_2,\hat\alpha_1\}$ are freely homotopic to a mutually disjoint, non-parallel collection
of simple, non-separating curves which do not bound a pair of pants,
Theorem \ref{conjugateontriples-trace} implies that we may assume that $\rho$ and $\sigma$ 
agree on $\braket{\alpha_1,\alpha_2,\hat\alpha_1}$.
Similarly, the representations are conjugate on $\braket{\alpha_1,\alpha_2,\hat\alpha_2}$,
and since they already agree on $\braket{\alpha_1,\alpha_2,\hat\alpha_1}$ and $\alpha_1$ and $\alpha_2$
have non-intersecting axes, Lemma \ref{conjugacytrivial} implies that they 
agree on $\braket{\alpha_1,\alpha_2,\hat\alpha_1,\hat\alpha_2}$. 
Next, by considering the triples $\{\alpha_1,\beta_2,\hat\alpha_1\}$ and $\{\alpha_1,\beta_2,\hat\beta_2\}$,
we see that $\rho$ and $\sigma$ are conjugate on 
$\braket{\alpha_1,\beta_2,\hat\alpha_1,\hat\beta_2}$.  Since $\rho$ and $\sigma$ agree on 
$\alpha_1$ and $\hat\alpha_1$, they agree on $\braket{\alpha_1,\beta_2,\hat\alpha_1,\hat\beta_2}$. 
By similarly considering the triples $\{\alpha_2,\beta_1,\hat\alpha_2\}$ and $\{\alpha_2,\beta_1,\hat\beta_1\}$,
we show that $\rho$ and $\sigma$ agree on $\beta_1$. Since we have shown that, after an initial
conjugation, $\rho$ and $\sigma$ agree on each generator, we have completed the proof in
the case that $(g,p)=(2,1)$.
\end{proof}

We similarly obtain the analogue of our Simple Trace Rigidity Theorem in this setting.

\begin{theorem}
Suppose that $S$ is a compact, orientable surface of genus $g>0$ with $p>0$ boundary components
and $(g,p)$  is not $(1,1)$  or $(1,2)$. Then, for all $d\ge 2$, there exists a finite
collection $\mathcal L_d(S)$ of elements of $\pi_1(S)$ which are represented by simple non-separating curves,
such that if $\rho$ and $\sigma$ are two Hitchin representations of $\pi_1(S)$ of dimension $d$
and $|\tr(\rho(\eta))|=|\tr(\sigma(\eta))|$ for any $\eta\in\mathcal L_d(S)$,
then $\rho$ and $\sigma$ are conjugate in $\pgln$.
\end{theorem}

\end{document}